\documentclass[12pt,a4paper,halfparskip]{scrartcl}

\RequirePackage{natbib}
\usepackage[english]{babel}
\usepackage{amsmath}
\usepackage{amssymb}
\usepackage{ulem}
\usepackage{comment}
\usepackage{url}
\usepackage{hyperref}
\usepackage{graphicx}
\usepackage{mathdots}
\usepackage{amsthm}
\usepackage{bbm}
\usepackage{enumitem}

\usepackage{framed}

\usepackage[T1]{fontenc}
\usepackage[utf8x]{inputenc}

\DeclareMathOperator{\argmin}{argmin}

\begin{document}

\newcommand{\Rho}{P}
\newcommand{\IN}{\mathbb{N}}
\newcommand{\IQ}{\mathbb{Q}}
\newcommand{\IZ}{\mathbb{Z}}
\newcommand{\IR}{\mathbb{R}}
\newcommand{\IC}{\mathbb{C}}
\newcommand{\Ima}{\mbox{Im}}
\newcommand{\dif}{\ \mbox{d}}
\newcommand{\cov}{\mbox{cov}}

\newcommand{\Lp}{\mathcal{L}}
\newcommand{\sI}{\mathcal{I}}
\newcommand{\sA}{\mathcal{A}}
\newcommand{\sB}{\mathcal{B}}
\newcommand{\sP}{\mathcal{P}}
\newcommand{\sE}{\mathcal{E}}
\newcommand{\sF}{\mathcal{F}}
\newcommand{\sG}{\mathcal{G}}
\newcommand{\sH}{\mathcal{H}}
\newcommand{\sT}{\mathcal{T}}
\newcommand{\sV}{\mathcal{V}}
\newcommand{\sL}{\mathcal{L}}
\newcommand{\SUP}{\text{SUP }}

\renewcommand{\Re}{\mbox{Re }}
\renewcommand{\Im}{\mbox{Im }}

\newcommand{\reff}[1]{(\ref{#1})}

\newcommand{\IP}{\mathbb{P}}
\newcommand{\IE}{\mathbb{E}}
\newcommand{\Ii}{\mathbbm{1}}
\newcommand{\supp}{\mbox{supp}}
\newcommand{\Hess}{\mbox{Hess}}
\newcommand{\Var}{\mbox{Var}}
\newcommand{\sX}{\mathcal{X}}
\newcommand{\Kov}{\mbox{Kov}}
\newcommand{\Cov}{\mbox{Cov}}
\newcommand{\tr}{\mbox{tr}}
\newcommand{\gdw}{\Leftrightarrow}
\newcommand{\pto}{\overset{\IP}{\to}}
\newcommand{\fsto}{\overset{f.s.}{\to}}
\newcommand{\dto}{\overset{D}{\to}}
\newcommand{\lto}{\overset{L^2}{\to}}
\newcommand{\sD}{\mathcal{D}}
\newcommand{\iid}{\overset{\mbox{iid}}{\sim}}
\renewcommand{\l}{\ell}
\renewcommand{\i}{|}

\renewcommand{\supp}{\text{supp}}

\newcommand{\err}{\mbox{err}}
\newcommand{\bias}{\mbox{bias}}

\newcommand{\norm}[1]{\left\lVert#1\right\rVert}

\newtheorem{theorem}{Theorem}[section]
\newtheorem{corollary}[theorem]{Corollary}
\newtheorem{definition}[theorem]{Definition}
\newtheorem{proposition}[theorem]{Proposition}
\newtheorem{lemma}[theorem]{Lemma}
\newtheorem{remark}[theorem]{Remark}
\newtheorem{exampleremark}[theorem]{Examples/Remarks}
\newtheorem{example}[theorem]{Example}
\newtheorem{assumption}[theorem]{Assumption}



\title{Cross validation for locally stationary processes}
\author{Stefan Richter, Rainer Dahlhaus}
\maketitle







\begin{abstract}

We propose an adaptive bandwidth selector via cross validation for local M-estimators in locally stationary processes. We prove asymptotic optimality of the procedure under mild conditions on the underlying parameter curves. The results are applicable to a wide range of locally stationary processes such linear and nonlinear processes. A simulation study shows that the method works fairly well also in misspecified situations.
\end{abstract}




\section{Introduction}

Inference for locally stationary time series models is strongly connected to the estimation of parameter curves which determine the degree of nonstationarity. The estimation of these curves was discussed for several specific models such as tvARMA processes (\cite{dahlhauspolonik2009}), the tvARCH and tvGARCH processes (\cite{suhasini2008}, \cite{suhasini2006}, \cite{dahlhaus2012}), and time-varying random coefficient models (\cite{suhasini2006b}). Of interest is also a time-varying TAR process which was considered in \cite{zhouwu2009}

Local estimators such as kernel estimators require the selection of a bandwidth. In opposite to nonparametric regression, there exist only very few theoretical results about adaptivity for locally stationary processes. We mention \cite{mallat1998} who discussed adaptive covariance estimation for a general class of locally stationary processes. Other results are constructed for specific models and are partly dependent on further tuning parameters: \cite{giraud2015} discussed online-adaptive forecasting of tvAR processes and \cite{arkoun2008}, \cite{arkoun2011} proposed methods for sequential and minimax-optimal bandwidth selection for tvAR processes of order 1.

In this paper we treat the problem for arbitrary locally stationary time series models determined by a time varying parameter curve. We focus on local M-estimators and use the functional dependence measure introduced in \cite{wu2005} to formulate mixing conditions. We propose an adaptive bandwidth selection procedure inspired by cross validation in the iid regression model which does not need any tuning parameters.
We discuss the theoretic behavior by proving asymptotic optimality of the selector (similar to \cite{haerdlemarron1985} where nonparametric regression has been treated). We also prove convergence towards the deterministic asymptotic optimal bandwidth. 

The technical core of the paper is martingaly theory applied in particular to the score function of the objective function and several bounds for moments of quadratic and cubic forms of locally stationary processes which are needed to provide convergence of expansions of the estimation error with suitable rates.

In Section \ref{sec2} we introduce the locally stationary time series model and formalize the separation of the process into a  parametric stationary process and unknown parameter curves. We define local M-estimators and the cross validation procedure. We introduce a Kullback-Leibler type distance measure which can be seen as an analogue to the averaged squared error in nonparametric regression.

In Section \ref{sec3} we prove asymptotic optimality of the cross validation procedure with respect to the Kullback-Leibler type distance measure and convergence of the cross validation bandwidth towards the deterministic asymptotic optimal bandwidth. The assumptions are stated in terms of a parametric stationary time series model which is connected to the locally stationary process. This allows for easy verification since most of the conditions are standard in M-estimation theory and were already shown for specific stationary models. 

In Section \ref{sec4} we discuss some processes where the main results are applicable. The performance of the method for different models such as tvAR, tvARCH and tvMA is studied in simulations.

In Section \ref{sec5} a short conclusion is drawn. All proofs are deferred without further reference to Section \ref{sec6} and the appendix.

\section{A cross validation method for locally stationary processes}
\label{sec2}

\subsection{The Model}
In this paper we discuss adaptive estimation of a multidimensional parameter curve $\theta_0:[0,1] \to \Theta \subset \IR^p$, i.e. we restrict to locally stationary processes $X_{t,n}$, $t = 1,...,n$ parameterized by curves. As usual we are working in the infill asymptotic framework with rescaled time $t/n \in [0,1]$, where $n$ denotes the number of observations.

Following the original idea of locally stationary processes, for fixed $u\in[0,1]$, $X_{t,n}$ should locally (i.e., for $|u-\frac{t}{n}| \ll 1$)  behave like a stationary process $\hat X_t(u)$. In this paper, we assume that the time dependence of the approximation $\hat X_t(u)$ is solely described by $\theta_0$, i.e. $\hat X_t(u) = \tilde X_t(\theta_0(u))$, where $\tilde X_t(\theta)$, $\theta \in \Theta$ is some family of parametric stationary processes. In this paper we will formulate the assumptions in terms of $\tilde X_t(\theta)$ instead of $\hat X_t(u)$ leading to a clear separation between the properties of the model class and the smoothness assumptions on $\theta_0$.
We formalize this by

\begin{assumption}[Locally stationary time series model]\label{ass1}
Let $q \ge 1$ and $\|W\|_q := (\IE |W|^{q})^{1/q}$. Let $X_{t,n}$, $t = 1,...,n$ be a triangular array of observations. Suppose that for each $\theta \in \Theta$, there exists a stationary process $\tilde X_t(\theta)$, $t\in\IZ$ such that for all $q \ge 1$,
	uniformly in $\theta,\theta' \in \Theta$,
	\begin{equation}
		\| \tilde X_t(\theta) - \tilde X_t(\theta')\|_{q} \le C_A |\theta-\theta'|_1, \quad\quad \sum_{t=1}^{n}\big\|X_{t,n} - \tilde X_t\big(\theta_0\big(\frac{t}{n}\big)\big)\big\|_{q} \le C_B,\label{ass1_s1_eq1}
	\end{equation}
	with some $C_A, C_B \ge 0$, and
	\[
		D_{q} := \max\{\sup_{\theta \in \Theta}\|\tilde X_0(\theta)\|_q, \sup_{n\in\IN}\sup_{t=1,...,n}\|X_{t,n}\|_{q}\} < \infty.
	\]
\end{assumption}

\begin{remark}
\begin{enumerate}
	\item[(i)] We conjecture that the assumption on the existence of all moments of $X_{t,n}$ and $\tilde X_t(\theta)$ can be dropped - but the calculations would be very tedious without much additional insight. The number of moments needed for the proofs increases if the Hoelder exponent of the unknown parameter curve decreases.
	\item[(ii)] In many models, the second condition in \reff{ass1_s1_eq1} basically means that the unknown parameter curve $\theta_0$ has bounded variation, see also Assumption \ref{ass3}.
\end{enumerate}
\end{remark}

We first give some examples which are covered by our results. These include in particular several classical parametric time series models where the constant parameters have been replaced by time-dependent parameter curves.

\begin{example}\label{example_model}
	\begin{enumerate}
		\item[(i)] the tvARMA($r,s$) process: Given parameter curves $a_i, b_j, \sigma:[0,1] \to \IR$ $(i = 0,...,r$, $j = 0,...,s$) with $a_0(\cdot), b(\cdot) = 1$,
		\[
			\sum_{i=0}^{r}a_i\big(\frac{t}{n}\big)X_{t,n} = \sum_{j=0}^{s}b_j\big(\frac{t}{n}\big)\sigma\big(\frac{t}{n}\big)\varepsilon_t.
		\]
		\item[(ii)] the tvARCH($r$) process (cf. \cite{suhasini2006}): Given parameter curves $a_i:[0,1]\to \IR$ ($i = 0,...,r$),
		\[
			X_{t,n} = \big( a_0\big(\frac{t}{n}\big) + a_{1}\big(\frac{t}{n}\big)X_{t-1,n}^2 + ... + a_r\big(\frac{t}{n}\big) X_{t-r,n}^2\big)^{1/2}\varepsilon_t.
		\]
		\item[(iii)] the tvTAR($1$) process (cf. \cite{zhouwu2009}): Given parameter curves $a_1,a_2:[0,1]\to\IR$, define
		\[
			X_{t,n} = a_1\big(\frac{t}{n}\big)X_{t-1,n}^{+} +a_2\big(\frac{t}{n}\big)X_{t-1,n}^{-} + \varepsilon_t,
		\]
		where $x^{+} := \max\{x,0\}$ and $x^{-} := \max\{-x,0\}$.
	\end{enumerate}
\end{example}

As an estimator of $\theta_0 (\cdot)$ we consider local likelihood (or local M-) estimators weighted by kernels, that is
\begin{equation}
	\hat \theta_h(u) := \argmin_{\theta \in \Theta} L_{n,h}(u,\theta).\label{def_mle}
\end{equation}
where
\begin{equation}
	L_{n,h}(u,\theta) := \frac{1}{n}\sum_{t=1}^{n}K_h\Big(\frac{t}{n}-u\Big) \ell_{t,n}(\theta)\label{h4}
\end{equation}
and $\ell_{t,n}(\theta) := \ell(X_{t,n}, Y_{t-1,n}^c,\theta)$ with $Y_{t-1,n}^c := (X_{t-1,n},...,X_{1,n},0,0,...)$ consisting of the observed past, where $\ell$ is a given objective function (localized in $L_{n,h}(u,\theta)$ by the kernel $K$). $K: \IR \to \IR$ is nonnegative with $\int K = 1$,  and $h \in (0,\infty)$ is the bandwidth. For shortening the notation, we used $K_h(\cdot) := \frac{1}{h}K\big(\frac{\cdot}{h}\big)$. In practice, $\ell$ is often chosen to be the negative logarithm of the infinite past likelihood of $X_{t,n}$ given $Y_{t-1,n} := (X_{s,n}: s \le t-1)$,
\begin{equation}
	\ell(x,y,\theta) = -\log p_{\theta}(X_{t,n} = x| Y_{t-1,n} = y),\label{infinitepastlikelihood}
\end{equation}
assuming that $\theta_0(\cdot) = \theta \in \Theta$. In this paper, we allow for general objective functions $\ell$ which have to obey some smoothness conditions (see Assumption \ref{ass3}).

%

\subsection{Distance measures}
Define $\tilde Y_t(\theta) := (\tilde X_s(\theta): s \le t)$. In the following, we will use $\nabla$ to denote the derivative with respect to $\theta \in \Theta$, and $x'$ denotes the transpose of a vector or matrix $x$. As global distance measures we use the averaged and the integrated squared error (ASE/ISE) weighted by the Fisher information
\begin{equation}
	I(\theta) := \IE\big[\nabla \l(\tilde Y_{0}(a),\theta)\cdot \nabla \l(\tilde Y_{0}(a),\theta)'\big]\big|_{a=\theta}.\label{fisherinfo}
\end{equation}
and the misspecified Fisher information $V(\theta) := \IE \nabla^2 \l(\tilde Y_0(a),\theta) \big|_{a=\theta}$ of the corresponding stationary approximation.
In addition the weight function $w_{n,h}(\cdot):= \Ii_{[\frac{h}{2},1-\frac{h}{2}]}(\cdot)$ is needed to exclude boundary effects. Since the proof is the same for other weights $w_{n,h}$ we allow in Assumption \ref{ass4} for more general weights.

More precisely we set (with $|x|_{A}^2 := x'A x$ for $x \in \IR^p$ and $A \in \IR^{p\times p}$)
\begin{equation}
	d_A(\hat \theta_h, \theta_0) := \frac{1}{n}\sum_{t=1}^{n} \Big| \hat \theta_h\big(\frac{t}{n}\big) - \theta_0\big(\frac{t}{n}\big)\Big|_{V(\theta_0(t/n))}^2 w_{n,h}\big(\frac{t}{n}\big) \label{da_distance}
\end{equation}
and
\begin{equation}
	d_I(\hat \theta_h, \theta_0) := \int_{0}^{1}\big| \hat \theta_h(u) - \theta_0(u)\big|_{V(\theta_0(u))}^2 w_{n,h}(u)\dif u.
\end{equation}

It can be shown that $2 d_{A}$ and $2 d_{I}$ are for $w \!\equiv \! 1$ an approximation of the Kullback-Leibler divergence between models with parameter curves $\hat \theta_h (\cdot)$ and $\theta_0 (\cdot)$.

In Theorem \ref{theorem3} we will prove that under suitable conditions, $d_{A}(\hat \theta_h, \theta_0)$ can be approximated uniformly in $h$ by a deterministic distance measure $d_{M}^{**}(h)$, which has a unique minimizer $h_0 = h_{0,n}\sim n^{-1/5}$. $h_0$ can be seen as the (deterministic) optimal bandwidth.

\subsection{The crossvalidation method}
We now choose the bandwidth $h$ by a generalized cross validation method. We define a 'quasi-leave-one-out' local likelihood
\begin{equation}
	L_{n,h,-s}(u,\theta) := \frac{1}{n}\sum_{t=1, t \not= s}^{n}K_h\Big(\frac{t}{n}-u\Big) \ell_{t,n}(\theta)
\end{equation}
and a 'quasi-leave-one-out' estimator of $\theta_0$ by
\begin{equation}
	\hat \theta_{h,-s}(u) := \argmin_{\theta \in \Theta} L_{n,h,-s}(u,\theta).
\end{equation}
Here, 'leave-one-out' does not mean that we ignore the $s$-th observation of the process $(X_{t,n})_{t=1,...,n}$, but that we ignore the term which is contributed by the likelihood $\ell_{t,n}$ at time step $s$. Because of that, we refer to the estimator as a quasi-leave-one-out method.

We then choose $\hat h$ via minimizing the cross validation functional
\begin{equation}
	CV(h) := \frac{1}{n}\sum_{s=1}^{n}\ell_{s,n}\big(\hat \theta_{h,-s}\big(\frac{s}{n}\big)\big) w_{n,h}\big(\frac{s}{n}\big).
\end{equation}
It is important to note that such a minimizer $\hat h$ of $CV(h)$ does not need to exist, because $CV(h)$ can not shown to be continuous. When $h$ varies it is possible that the location of the minimum of $L_{n,h,-s}(u,\theta)$ changes and therefore $\hat \theta_{h,-s}(u)$ makes a jump. For the mathematical considerations we therefore choose some $\hat h$ such that
\begin{equation}
	CV(\hat h) - \inf_{h \in H_n}CV(h) \le \frac{1}{n},
\end{equation}
where $H_n$ is a suitable subinterval of $(0, 1)$, see Assumption \ref{ass4}, which covers all relevant values of  $h$.\\

\section{Main results}
\label{sec3}

In this chapter we present our main results concerning the bandwidth $\hat h$ chosen by cross validation. We prove in Theorem \ref{theorem1} that $\hat h$ is asymptotically optimal with respect to $d_{A}$, i.e.
\[
		\lim_{n\to\infty} \frac{d_A(\hat \theta_{\hat h}, \theta_0)}{\inf_{h \in H_n} d_A(\hat \theta_h, \theta_0)} = 1,
\]
and in Theorem \ref{theorem2} that $\hat h$ is consistent in the sense that ${\hat h}/{h_0} \to 1$ a.s., where $h_0$ is the deterministic optimal bandwidth defined in \reff{h02optimal}. Recall that $d_A(\hat \theta_h,\theta_0)$ can be interpreted as a Kullback-Leibler-type distance between the two time series models associated to $\hat \theta_h$ and $\theta_0$. Thus, the cross validation procedure yields an estimator $\hat \theta_{\hat h}$ of $\theta_0$ such that the distributions of the associated time series coincide best.


To prove asymptotic results, we have to state some mixing type conditions on the underlying process $X_{t,n}$. For this, we use the functional dependence measure introduced in \cite{wu2005}. Let $\varepsilon_t$, $t\in\IZ$ be a sequence of i.i.d. random variables. For $t \ge 0 \;$ let $\sF_t := (\varepsilon_t,\varepsilon_{t-1},...)$ be the shift process and $\sF_t^{*} := (\varepsilon_t,...,\varepsilon_{1},\varepsilon_0^{*},\varepsilon_{-1},...)$, where $\varepsilon_{0}^{*}$ is a random variable which has the same distribution as $\varepsilon_{0}$ and is independent of all $\varepsilon_t$, $t\in\IZ$. For a stationary process $W_t = H(\sF_t) \in L^q$ with deterministic $H:\IR^{\infty}\to\IR$ define $W_t^{*} := H_t(\sF_t^{*})$ and the functional dependence measure
\begin{equation}
	\delta^{W}_q(k) := \|W_t - W_t^{*}\|_q.\label{def_dependence}
\end{equation}

\begin{assumption}[Dependence assumption]\label{ass2}
	Suppose that for each $\theta \in \Theta$, there exists a representation
		 $\tilde X_t(\theta) = H(\theta,\sF_t)$ with some measurable $H(\theta,\cdot)$ and $\delta_q(k) := \sup_{\theta \in \Theta}\delta_q^{\tilde X(\theta)}(k) = O(k^{-(3+\eta)})$ for some $\eta > 0$.
\end{assumption}

Note that we only need dependence conditions on the stationary approximations $\tilde X_t(\theta)$ and no further assumption on $X_{t,n}$.

To state smoothness conditions on the objective function $\ell$ in a concise way, we introduce the class of Lipschitz-continuous functions from $\IR^{\infty}$ to $\IR$ where we allow the Lipschitz constant to depend on the location at most polynomially.

\begin{definition}[The class $\sL(M,\chi, C)$]\label{DP_hoelder_def} We say that a function $g:\IR^{\infty} \times \Theta \to \IR$ is in the class $\sL(M,\chi,C)$ if $C = (C_1,C_2)$, $M \ge 1$, $\chi = (\chi_i)_{i=1,2,3,...} \in \IR_{\ge 0}^{\infty}$ and for all $z\in \IR^{\infty}$, $\theta \in \Theta$:
	\begin{equation}
		\sup_{z\not= z'}\frac{|g(z,\theta) - g(z',\theta)|}{|z-z'|_{\chi,1} (1+|z|_{\chi,1}^{M-1} + |z'|_{\chi,1}^{M-1})} \le C_{1},\quad \sup_{\theta\not=\theta'}\frac{|g(z,\theta) - g(z,\theta')|}{|\theta-\theta'|_1(1+|z|_{\chi,1}^{M})} \le C_2\label{DP_hoelder_prop_polynom}
	\end{equation}
	where $|z|_{\chi,1} := \sum_{i=1}^{\infty}\chi_i \cdot |z_i|$ and $\sum_{i=1}^{\infty}\chi_i < \infty$.
\end{definition}

In Assumption \ref{ass3}, we pose some standard conditions on the likelihood function $\ell$ which ensure the validity of basic results (such as Taylor expansions) from maximum likelihood theory. Again all conditions are formulated in terms of the stationary process $\tilde X_t(\theta)$ and therefore easily verifiable due to known results on stationary time series.

\begin{assumption}\label{ass3}
	Suppose that $\ell$ is three times differentiable with respect to $\theta$, and
	\begin{enumerate}[label=(\arabic*),ref=(\arabic*)]
		\item\label{ass3_m1} $\Theta \subset \IR^d$ is compact. For all $u\in [0,1]$, $\theta_0(u)$ lies in the interior of $\Theta$ and $\theta_0$ is Hoelder continuous with exponent $\beta > 0$ and has component-wise bounded variation $B_{\theta_0}$.
		\item\label{ass3_m2} $\theta_0(u)$ is the unique minimizer of $L(u,\theta) := \IE \l(\tilde Y_0(\theta_0(u)),\theta)$.
		\item\label{ass3_m3} the minimal eigenvalue of $V(\theta) := \IE[\nabla^2 \l(\tilde Y_t(\theta),\theta)]$ is bounded from below by some constant $\lambda_0$ uniformly in $\theta\in\Theta$.
		\item\label{ass3_m4} $\nabla \l(\tilde Y_t(\theta'),\theta)\big|_{\theta' = \theta}$ is a martingale difference sequence with respect to $\sF_t$ in each component.
		\item\label{ass3_m5} each component of $g \in \{\l, \nabla \l, \nabla^2,\nabla^3 \l\}$ lies in $\sL(M,\chi,C)$ for some $\chi = (\chi_j)_{j = 1,2,...}$, where $\chi_j = O(j^{-(3+\eta)})$ for some $\eta > 0$. 
	\end{enumerate}
\end{assumption}

Finally, let us formalize the conditions on the set of bandwidths $H_n$, the localizing kernel $K$ appearing in the estimation procedure and the weight function $w_{n,h}$ which arises in the cross validation functional and the distance measures. 

\begin{assumption}\label{ass4} For $n \in \IN$ let $H_n = [\underline{h}, \overline{h}]$, where $\underline{h} \ge c_0 n^{\delta-1}$, $\overline{h} \le c_1 n^{-\delta}$ for some constants $c_0,c_1,\delta > 0$. Suppose that
	\begin{enumerate}[label=(\arabic*),ref=(\arabic*)]
		\item\label{ass4_p1} the kernel $K:\IR \to \IR$ has compact support $\subset [-\frac{1}{2},\frac{1}{2}]$, fulfills $\int K(x) \dif x = 1$ and is Lipschitz continuous with Lipschitz constant $L_K$.
		\item\label{ass4_p2} the weight function $w_{n,h}:[0,1]\to \IR_{\ge 0}$ is bounded by $|w|_{\infty}$, has bounded variation $B_{w}$ uniformly in $n,h$ and support $\subset [\frac{h}{2}, 1- \frac{h}{2}]$.\\
		For some $w:[0,1]\to \IR_{\ge 0}$ with support of Lebesgue measure greater than zero, assume that $\sup_{h\in H_n}\int |w_{n,h}(u) - w(u)| \dif u \to 0$.\\
		Furthermore, suppose that there exists some $C_w > 0$ such that
		\begin{eqnarray}
			\textstyle \frac{1}{n}\sum_{t=1}^{n}|w_{n,h}(t/n) - w_{n,h'}(t/n)| &\le& C_w|h-h'|,\nonumber\\
			\textstyle\int_{0}^{1}|w_{n,h}(u) - w_{n,h'}(u)| \dif u &\le& C_w |h-h'|.\label{ass4_eq1}
		\end{eqnarray}
	\end{enumerate}
\end{assumption}

\begin{remark}
	Note that all conditions in \ref{ass4_p2} are fulfilled by the indicator $w_{n,h}(\cdot) = \Ii_{[\frac{h}{2},1-\frac{h}{2}]}(\cdot)$ or $w_{n,h}(\cdot) = \Ii_{[\nu,1-\nu]}(\cdot)$ with some fixed $\nu > 0$.
\end{remark}

We now show that the cross validation bandwidth $\hat h$ is asymptotically optimal.

\begin{theorem}[Asymptotic optimality of cross validation]\label{theorem1}
	Under assumptions \ref{ass1}, \ref{ass2}, \ref{ass3} and \ref{ass4} the bandwidth $\hat h$ chosen by cross validation is asymptotically optimal in the sense that
	\[
		\lim_{n\to\infty} \frac{d(\hat \theta_{\hat h}, \theta_0)}{\inf_{h \in H_n} d(\hat \theta_h, \theta_0)} = 1,
	\]
	where $d$ is $d_A$ or $d_I$.
\end{theorem}

Under stronger smoothness assumptions which allow a typical bias expansion up to the second derivative, we will show (in Theorem~\ref{theorem2} below) that $\hat h$ is asymptotically equivalent to the asymptotically optimal theoretical bandwidth $h_0$ (ao-bandwidth for short). The additional smoothness assumptions are natural specifications of Assumption \ref{ass1} and \ref{ass3}.

\begin{assumption}[Bias expansion conditions]\label{ass5} Suppose that
\begin{enumerate}[label=(\arabic*),ref=(\arabic*)]
	\item\label{ass5_b1} $K$ is symmetric and $\theta_0$ is twice continuously differentiable,
	\item\label{ass5_b2} for all $\theta \in \Theta$, $z\in\IR^{\infty}$, $z \mapsto \nabla \l(z,\theta)$ is twice partially differentiable and $\partial_{z_i} \partial_{z_j} \nabla \l(\cdot,\theta) \in \sL(\max\{M-2,1\},\tilde \chi,\tilde \psi_1(i)\tilde \psi_2(j))$ for all $i,j \ge 1$ with absolutely summable sequences $\tilde \psi_1, \tilde\psi_2$.
	\item\label{ass5_b3} $\theta \mapsto \tilde X_t(\theta)$ is twice continuously differentiable almost surely. For all $i,j=1,...,d$,  $\|\sup_{\theta \in \Theta}|\nabla_i \tilde X_0(\theta)| \|_M$ and $\|\sup_{\theta \in \Theta}|\nabla^2_{ij} \tilde X_0(\theta)| \|_M$ are finite.
\end{enumerate}
\end{assumption}
We know from standard asymptotics that
\begin{eqnarray*}
\hat \theta_h(u) - \theta_0 (u) &\approx& - \nabla ^{2} L_{n,h}(u,\bar \theta (u))^{-1} \nabla L_{n,h}(u,\theta_0(u))\\
&\approx& - V(\theta_0(u))^{-1} \nabla L_{n,h}(u,\theta_0(u)),
\end{eqnarray*}
which motivates the following approximations to $d_A(\hat \theta_h, \theta_0)$ and $d_I(\hat \theta_h, \theta_0)$:
\begin{eqnarray}
	d_A^{*}(\hat \theta_h, \theta_0) &:=& \frac{1}{n}\sum_{t=1}^{n}\Big| \nabla L_{n,h}\big(\frac{t}{n},\theta_0\big(\frac{t}{n}\big)\big)\Big|_{V(\theta_0(\frac{t}{n}))^{-1}}^2 w_{n,h}\big(\frac{t}{n}\big),\\
	d_I^{*}(\hat \theta_h, \theta_0) &:=& \int_{0}^{1}\big| \nabla L_{n,h}(u,\theta_0(u))\big|_{V(\theta_0(u))^{-1}}^2 w_{n,h}(u) \dif u.
\end{eqnarray}
We now set
\begin{equation*} \label{}
	d_{M}^{*}(\hat \theta_h,\theta_0) := \IE d_{I}^{*}(\hat \theta_h,\theta_0) .
\end{equation*}
If $\theta_0$ is twice continuously differentiable and some additional smoothness assumptions on the approximating stationary process (see Assumption \ref{ass5}), Proposition~\ref{misedecomposition} together with Assumption \ref{ass4}\ref{ass4_p2} implies the usual bias-variance decomposition for $d_{M}^{*}$:
\begin{equation}
	d_{M}^{*}(\hat \theta_h, \theta_0) = \frac{\mu_K V_0}{nh} + \frac{h^{4}}{4}d_K^2 B_0 + o( (nh)^{-1}) + o(h^4)\label{misedecompositionoben}
\end{equation}
uniformly in $h \in H_n$, where $\mu_K  := \int K(x)^2 \dif x$, $d_K := \int K(x) x^2 \dif x$ and
\begin{align}
		V_0 &:= \int_{0}^{1}\tr\{V(\theta_0(u))^{-1}I(\theta_0(u))\} w(u)\dif u > 0,\label{asv}\\
		B_0 &:= \int_{0}^{1}\big|\IE\big[\partial_u^2 \nabla \l(\tilde Y_t(\theta_0(u)),\theta)\big]\big|_{\theta = \theta_0(u)}\big|_{V(\theta_0(u))^{-1}}^2 w(u) \dif u \ge 0,\label{asb}
\end{align}
leading to the definition of the deterministic bias-variance decomposition $d_M^{**}(h)$ and the resulting asymptotically optimal bandwidth in the following two theorems.

\begin{theorem}[Approximation of distance measures]\label{theorem3} Let Assumptions \ref{ass1}, \ref{ass2}, \ref{ass3}, \ref{ass4} and \ref{ass5} hold. Define
	\begin{equation}
		d_M^{**}(h) := \frac{\mu_K V_0}{nh} + \frac{h^4}{4}d_K^2 B_0  \label{dmsternstern}
	\end{equation}
If the bias $B_0$ is not degenerated, i.e. $B_0 > 0$, then it holds that
	\[
		\sup_{h\in H_n}\left|\frac{d(\hat \theta_h, \theta_0) - d_{M}^{**}(h)}{d_{M}^{**}(h)}\right| \to 0 \qquad a.s.
	\]
	where $d$ is $d_A$ or $d_I$.
\end{theorem}

\bigskip

\begin{theorem}[Consistency of the cross validation bandwidth]\label{theorem2}
	Let Assumptions \ref{ass1}, \ref{ass2}, \ref{ass3}, \ref{ass4} and \ref{ass5} hold. Then the bandwidth $\hat h$ chosen by cross validation fulfils
	\[
		\frac{\hat h}{h_0} \to 1 \qquad a.s.
	\]
	where
	\begin{equation}
		h_0 = \left(\frac{V_0 \mu_K}{B_0 d_K^2}\right)^{1/5} n^{-1/5}\label{h02optimal}.
	\end{equation}
	is the unique minimizer of $d_M^{**}(h)$.
\end{theorem}


\subsection{Proofs}
Here we present the structure of the proofs of Theorems \ref{theorem1}, \ref{theorem2} and \ref{theorem3}. The technical details including the proofs of the lemmata are postponed to the appendix. The main tool for the proofs is a general bound for moments on quadratic and cubic forms of functions of locally stationary processes (cf. Proposition \ref{hilfslemma1}). From now on, we assume that Assumptions \ref{ass1}, \ref{ass2}, \ref{ass3} and \ref{ass4} hold.
The following Lemma shows that the approximated distances $d_I^{*}$, $d_{A}^{*}$ are close to  $d_M^{*}$.

\begin{lemma}\label{da_approx_0}
	We have almost surely
	\[
		\sup_{h \in H_n}\Big|\frac{d_{I}^{*}(\hat \theta_h, \theta_0) - d_{M}^{*}(\hat \theta_h,\theta_0)}{d_{M}^{*}(\hat \theta_h, \theta_0)} \Big| \to 0,\quad\quad
		\sup_{h \in H_n}\Big|\frac{d_{A}^{*}(\hat \theta_h, \theta_0) - d_{M}^{*}(\hat \theta_h,\theta_0)}{d_{M}^{*}(\hat \theta_h, \theta_0)} \Big| \to 0.
	\]
\end{lemma}

As a consequence of Lemma \ref{da_approx_0} also the distances $d_{I}, d_A$ are close to  $d_M^{*}$:

\medskip

\begin{corollary}\label{da_approx}
	We have almost surely
	\[
		\sup_{h \in H_n}\Big|\frac{d_{I}(\hat \theta_h, \theta_0) - d_{M}^{*}(\hat \theta_h,\theta_0)}{d_{M}^{*}(\hat \theta_h, \theta_0)} \Big| \to 0, \quad\quad
		\sup_{h \in H_n}\Big|\frac{d_{A}(\hat \theta_h, \theta_0) - d_{M}^{*}(\hat \theta_h,\theta_0)}{d_{M}^{*}(\hat \theta_h, \theta_0)} \Big| \to 0.
	\]
\end{corollary}

To get a connection between the distance measure $d_M^{*}$ and the cross validation functional $CV(h)$, we
 define
\[
	\overline{d}_{A}(\hat \theta_h, \theta_0) := \frac{1}{n}\sum_{t=1}^{n}\Big| \hat \theta_{h,-t}\big(\frac{t}{n}\big) - \theta_0\big(\frac{t}{n}\big)\Big|^2_{V(\theta_0(t/n))}.
\]
The next two lemmata show that $\overline{d}_A$ is close both to $d_M^{*}$ and $CV(h)$. Lemma \ref{da_cv_approx} can be viewed as the core of the proof since there the main assumptions come into play, such as the martingale property of $\nabla \ell(\tilde Y_t(\theta'),\theta)\big|_{\theta' = \theta}$ which is used for normalization and the differentiability properties of $\ell$ which are used for Taylor expansions of third order.

\begin{lemma}\label{da_strich_approx}
	We have almost surely
\[
	\sup_{h\in H_n}\left| \frac{\overline{d}_{A}(\hat \theta_h, \theta_0) - d_{M}^{*}(\hat \theta_h, \theta_0)}{d_{M}^{*}(\hat \theta_h, \theta_0)}\right| \to 0.
\]
\end{lemma}

\begin{lemma}\label{da_cv_approx}
	We have almost surely
	\begin{equation}
		\sup_{h \in H_n}\Big|\frac{CV(h) - \frac{1}{n}\sum_{t=1}^{n}\l_{t,n}(\theta_0(\frac{t}{n}))w_{n,h}(t/n) -\overline{d}_{A}(\hat \theta_h,\theta_0)}{d_{M}^{*}(\hat \theta_h,\theta_0)}\Big| \to 0.\label{hauptterm}
	\end{equation}
\end{lemma}
With the help of these results, we can now prove Theorems \ref{theorem1}, \ref{theorem3}, \ref{theorem2}:
\begin{proof}[Proof of Theorem \ref{theorem1}]
An immediate consequence of Lemma \ref{da_cv_approx} is (use $\frac{x_1+x_2}{y_1+y_2} \le \frac{x_1}{y_1} + \frac{x_2}{y_2}$ for positive numbers $x_1,x_2,y_1,y_2 > 0$)
	\[
		\sup_{h,h' \in H_n}\left|\frac{\overline{d}_A(\hat \theta_h, \theta_0) - \overline{d}_A(\hat \theta_{h'}, \theta_0) - (CV(h) - CV(h'))}{d_{M}^{*}(\hat \theta_h, \theta_0) + d_{M}^{*}(\hat \theta_{h'},\theta_0)}\right| \to 0 \quad a.s.
	\]
almost surely. Now, using Corollary \ref{da_approx} and Lemma \ref{da_strich_approx} it is easy to see that
\[
	\sup_{h,h' \in H_n}\left|\frac{d_A(\hat \theta_h, \theta_0) - d_{A}(\hat \theta_{h'},\theta_0) - (CV(h) - CV(h'))}{d_{A}(\hat \theta_h, \theta_0) + d_{A}(\hat \theta_h, \theta_0)}\right| \to 0\quad a.s.
\]
Choosing $h = \hat h$ and $h'$ such that
\[
	d_{A}(\hat \theta_{h'}, \theta_0) - \inf_{h \in H_n} d_{A}(\hat \theta_h, \theta_0) \le n^{-1}
\]
yields
\begin{eqnarray*}
	0 &\leftarrow& \frac{d_{A}(\hat \theta_{\hat h},\theta_0) - d_{A}(\hat \theta_{h'}, \theta_0) - (CV(\hat h) - CV(h'))}{d_{A}(\hat \theta_{\hat h}, \theta_0) + d_{A}(\hat \theta_{h'}, \theta_0)}\\
	&\ge& \frac{d_{A}(\hat \theta_{\hat h},\theta_0) - \inf_{h\in H_n} d_{A}(\hat \theta_{h}, \theta_0) - ( \inf_{h \in H_n} CV(h) - CV(h'))}{d_{A}(\hat \theta_{\hat h}, \theta_0) + \inf_{h \in H_n} d_{A}(\hat \theta_{h}, \theta_0) + n^{-1}}\\
	&&\quad\quad\quad+ \frac{2n^{-1}}{d_{A}(\hat \theta_{\hat h}, \theta_0) + d_{A}(\hat \theta_{h'}, \theta_0)}
\end{eqnarray*}
almost surely. Because of Corollary \ref{da_approx} and \reff{misedecompositionoben} we have $\sup_{h \in H_n}\frac{n^{-1}}{d_{A}(\theta_h, \theta_0)} \to 0$ a.s. Thus,
\[
	\frac{d_{A}(\hat \theta_{\hat h}, \theta_0) - \inf_{h \in H_n}d_{A}(\hat \theta_{h}, \theta_0)}{d_{A}(\hat \theta_{\hat h}, \theta_0) + \inf_{h \in H_n} d_{A}(\hat \theta_{h}, \theta_0)} \to 0\quad a.s.,
\]
from which
\[
	\frac{d_{A}(\hat \theta_{\hat h}, \theta_0)}{\inf_{h \in H_n}d_{A}(\hat \theta_h, \theta_0)} \to 1 \quad a.s.
\]
follows. The same can be done for $d_{I}$.
\end{proof}

\begin{proof}[Proof of Theorem \ref{theorem3}]
	Because of $B_0 > 0$ and \reff{misedecompositionoben}, we have
	\begin{equation}
		\sup_{h \in H_n}\left| \frac{d_{M}^{*}(\hat \theta_h, \theta_0) - d_{M}^{**}(h)}{d_{M}^{**}(h)}\right| \to 0\quad a.s.\label{theorem3beweis}
	\end{equation}
	Application of Corollary \ref{da_approx} finishes the proof.
\end{proof}

\begin{proof}[Proof of Theorem \ref{theorem2}]
	As in the proof of Theorem \ref{theorem3}, we show \reff{theorem3beweis}. This result in combination with  Lemma \ref{da_strich_approx} and Lemma \ref{da_cv_approx} gives almost surely
	\[
		\sup_{h\in H_n}\left| \frac{CV(h) - \frac{1}{n}\sum_{t=1}^{n}l_{t,n}(\theta_0(t/n))w(t/n) - d_M^{**}(h)}{d_{M}^{**}(h)}\right| \to 0.
	\]
	Using the same methods as in the proof of Theorem \ref{theorem1}, we have almost surely
	\[
		\frac{d_M^{**}(\hat h)}{d_M^{**}(h_0)} = \frac{d_M^{**}(\hat h)}{\inf_{h \in H_n}d_{M}^{**}(h)} \to 1
	\]
	The structure of $d_M^{**}(h)$ implies $\hat h / h_0 \to 1$ a.s.
\end{proof}

\section{Examples and Simulations}
\label{sec4}

\subsection{Examples} Assumptions \ref{ass1}, \ref{ass2}, \ref{ass3} and \ref{ass5} are fulfilled for a large class of locally stationary time series models. Here, we discuss how the conditions transform in the case of some special linear and recursively defined time series. More general statements can be found in the technical supplement, see Proposition \ref{example_linearmodel} and \ref{example_recursion} therein.

Recall that $\varepsilon_t$, $t\in\IZ$ is a sequence of i.i.d. real random variables. We will use a Gaussian likelihood for $\ell$ defined in \reff{infinitepastlikelihood}, but allow for a non Gaussian distribution of $\varepsilon_t$.

An important special case of locally stationary linear processes is given by tvARMA processes, see also Proposition 2.4. in \cite{dahlhauspolonik2009}. Since in this case, the linear filter $A_{\theta}(\lambda) = \sigma\cdot \frac{\beta(e^{i\lambda})}{\alpha(e^{i\lambda})}$ and the spectral density $f_{\theta}(\lambda) = \frac{\sigma^2}{2\pi}\cdot \big|\frac{\beta(e^{i\lambda})}{\alpha(e^{i\lambda})}\big|^2$ have a simple form, the conditions in Proposition \ref{example_linearmodel} are obviously fulfilled. The likelihood \reff{infinitepastlikelihood} takes the form
\begin{equation}
	\l(z,\theta) = \frac{1}{2}\log\big(\frac{2\pi}{\gamma_{\theta}(0)^2}\big) + \frac{1}{2}\Big(\sum_{k=0}^{\infty}\gamma_{\theta}(k) z_{k+1}\Big)^2\label{example_linearmodel_eq2}.
\end{equation}

\begin{example}[tvARMA($r,s$) process]\label{example_tvarma}
	Assume that $\varepsilon_t, t \in\IZ$ are i.i.d. with existing moments of all order. Suppose that $\IE \varepsilon_0 = 0$ and $\IE \varepsilon_0^2 = 1$. Let Assumption \ref{ass3}\ref{ass3_m1} hold. Assume that $X_{t,n}$ obeys
	\[
		X_{t,n} + \sum_{j=1}^{r}\alpha_j\big(\frac{t}{n}\big) X_{t-j,n} = \sigma\big(\frac{t}{n}\big)\varepsilon_t + \sum_{k=1}^{s}\beta_k\big(\frac{t}{n}\big)\sigma\big(\frac{t-k}{n}\big)\varepsilon_{t-k}, \quad t = 1,...,n,
	\]
	where $\theta_0 = (\alpha_1,...,\alpha_r,\beta_1,...,\beta_s,\sigma)'$. Define $\beta(z) := 1 + \sum_{k=0}^{s}\beta_k z^k$, $\alpha(z) := 1 + \sum_{k=0}^{r}\alpha_k z^k$, and let $\Theta$ be an arbitrary compact subset of
	\begin{eqnarray*}
		&& \{\theta = (\alpha_1,...,\alpha_r,\beta_1,...,\beta_s,\sigma)' \in \IR^{r+s+1}: \sigma > 0,\\
		&&\quad\quad \alpha(z), \beta(z) \text{ have no zeros in common and}\\
		&&\quad\quad\quad\quad\quad\quad\quad\quad \text{only zeros outside the unit circle}\}.
	\end{eqnarray*}
	Then Assumptions \ref{ass1}, \ref{ass2}, \ref{ass3} are fulfilled for $\ell$ chosen as in \reff{example_linearmodel_eq2}. If additionally Assumption \ref{ass5}\ref{ass5_b1} is fulfilled, then Assumption \ref{ass5} is fulfilled. It holds that $V(\theta) = \frac{1}{4\pi}\int \nabla \log f_{\theta}(\lambda) \cdot \nabla \log f_{\theta}(\lambda)' \dif \lambda$.
\end{example}

\begin{remark}[tvAR($r$) processes]
	In the special case of $tvAR(r)$ processes, closed forms for the estimators based on $\l(z,\theta) = \frac{1}{2}\log(2\pi \sigma^2) + \frac{1}{2\sigma^2}\big(z_1 + \sum_{j=1}^{r}\alpha_j z_{j+1}\big)^2$ are available: $\hat \alpha_h(u) = -\hat \Gamma_h(u)^{-1}\hat \gamma_h(u)$ and $\hat \sigma_h(u)^2 = \frac{1}{n}\sum_{t=r+1}^{n}\big(X_{t,n} + \sum_{j=1}^{r}\hat \alpha_j(u) X_{t-j,n}\big)^2$, where $Z_{t-1,n} = (X_{t-1,n},...,X_{t-r,n})'$ and
	\begin{eqnarray*}
		\hat \Gamma_h(u) &:=& \frac{1}{n}\sum_{t=r+1}^{n}K_h\big(\frac{t}{n}-u\big)Z_{t-1,n}Z_{t-1,n}',\\
		\hat \gamma_h(u) &:=& \frac{1}{n}\sum_{t=r+1}^{n}K_h\big(\frac{t}{n}-u\big)X_{t,n}Z_{t-1,n}.
	\end{eqnarray*}
\end{remark}

We now discuss recursively defined nonlinear time series models with additive innovations $\varepsilon_t$. Let us fix some $r > 0$ and define the vectors of the last $r$ lags $Z_{t-1,n} = (X_{t-1,n},...,X_{t-r,n})'$, $\tilde Z_{t-1}(\theta) = (\tilde X_{t-1}(\theta),...,\tilde X_{t-r}(\theta))'$ as the vector of the $r$ past values of the locally stationary and the stationary time series, respectively. Many popular locally stationary models assume that the conditional mean and / or variance is a linear combination of unknown parameter curves and functions of $Z_{t-1,n}$, i.e.
\[
	X_{t,n} = \tilde\mu(Z_{t-1,n},\theta_0(t/n)) + \tilde\sigma(Z_{t-1,n},\theta_0(t/n)) \varepsilon_t, \quad t = 1,...,n,
\]
with some measurable $\tilde \mu$, $\tilde \sigma$. In this case, the likelihood \reff{infinitepastlikelihood} takes the form
\begin{equation}
	\l(x,y,\theta) := \frac{1}{2}\log\big(2\pi \tilde \sigma(y,\theta)^2\big)+ \frac{1}{2}\Big(\frac{x-\tilde \mu(y,\theta)}{\tilde \sigma(y,\theta)}\Big)^2.\label{example_recursion_eq2}
\end{equation}
In the first example we discuss conditional mean processes. This class covers the tvAR- as well as the tvTAR case.

\begin{example}[Conditional mean processes]\label{example_ar_mean}  Assume that $\varepsilon_t$, $t\in\IZ$ are i.i.d. and have all moments with $\IE \varepsilon_0 = 0$ and $\IE \varepsilon_0^2 = 1$. Suppose that Assumption \ref{ass3}\ref{ass3_m1} is fulfilled. Assume that $X_{t,n}$ obeys
	\[
		X_{t,n} = \alpha_1\big(\frac{t}{n})\mu_1(Z_{t-1,n}) + ... + \alpha_{p-1}\big(\frac{t}{n})\mu_{p-1}(Z_{t-1,n})+ \sigma\big(\frac{t}{n}\big)\varepsilon_t, \quad t = 1,...,n.
	\]
	Here, $\theta_0 = (\alpha_1,...,\alpha_{p-1},\sigma)':[0,1] \to \Theta$ and $\mu = (\mu_1,...,\mu_{p-1}):\IR^r\to \IR^{p-1}$ is a function which fulfills
	\begin{enumerate}
		\item[(a)] $\sup_{y\not=y'}\frac{|\mu_i(y) - \mu_i(y')|}{|y-y'|_{\chi_i,1}} \le 1$ with some $\chi_i \in \IR_{\ge 0}^r$ ($i = 1,...,p-1$),
		\item[(b)] $\mu_1(\tilde Y_0(\theta)),...,\mu_{p-1}(\tilde Y_0(\theta))$ are linearly independent in $L^2$ for all $\theta \in \Theta$.
	\end{enumerate}
	Define $\Theta := \{\theta = (\alpha_1,...,\alpha_d,\sigma)\in\IR^{p}:\sum_{i=1}^{p-1}\sum_{j=1}^{r}|\alpha_i|\chi_{i,j} \le \rho, \sigma_{min} \le \sigma \le \sigma_{max}\}$ with some $0 < \rho < 1$ and $0 < \sigma_{min} < \sigma_{max}$.
		
	Then Assumptions \ref{ass1}, \ref{ass2}, \ref{ass3} are fulfilled for $\ell$ chosen as in \reff{example_recursion_eq2}. Furthermore, with $W(\theta) := \IE[\mu(\tilde Z_0(\theta))\mu(\tilde Z_0(\theta))']$ it holds that
	\[
		V(\theta) = \frac{1}{\sigma^2}\begin{pmatrix}W(\theta) & 0\\
		0 & 2 \end{pmatrix}, \quad\quad I(\theta) = \frac{1}{\sigma^2}\begin{pmatrix}W(\theta) & \IE \varepsilon_0^3 \IE \mu(\tilde Z_0(\theta))\\
		\IE \varepsilon_0^3 \IE \mu(\tilde Z_0(\theta))' & \IE \varepsilon_0^4-1 \end{pmatrix}.
	\]
\end{example}

The next example discusses conditional variance processes, who cover, for instance, the tvARCH process.

\begin{example}[Conditional variance processes]\label{example_ar_var}  Assume that $\varepsilon_t$, $t\in\IZ$ are i.i.d. with $\IE \varepsilon_0 = 0$ and $\IE \varepsilon_0^2 = 1$ and are almost surely bounded by some $C_{\varepsilon} > 0$. Suppose that Assumption \ref{ass3}\ref{ass3_m1} is fulfilled. Assume that $X_{t,n}$ obeys
	\[
		X_{t,n} = \Big(\alpha_1\big(\frac{t}{n})\mu_1(Z_{t-1,n}) + ... + \alpha_{p}\big(\frac{t}{n})\mu_{p}(Z_{t-1,n})\Big)^{1/2}\varepsilon_t, \quad t = 1,...,n.
	\]
	Here, $\theta_0 = (\alpha_1,...,\alpha_{p})':[0,1] \to \Theta$ and $\mu = (\mu_1,...,\mu_{p-1}):\IR^r\to \IR^{p}_{\ge 0}$ is a function which fulfills
	\begin{enumerate}
		\item[(a)] $\sup_{y\not=y'}\frac{|\sqrt{\mu_i(y)} - \sqrt{\mu_i(y')}|}{|y-y'|_{\chi_i,1}} \le 1$ with some $\chi_i \in \IR_{\ge 0}^r$ ($i = 1,...,p$). There exists $\mu_0 > 0$ such that $\mu_1(y) \ge \mu_0$ for all $y\in\IR^{r}$.
		\item[(b)] $\mu_1(\tilde Y_0(\theta)),...,\mu_{p}(\tilde Y_0(\theta))$ are linearly independent in $L^2$ for all $\theta \in \Theta$.
	\end{enumerate}
	Define $\Theta := \{\theta = (\alpha_1,...,\alpha_p)\in\IR^{p}:\sum_{i=1}^{p}\sum_{j=1}^{r}\sqrt{\alpha_i}\chi_{i,j} \le \rho_{max}C_{\varepsilon}^{-1}, \theta_i \ge \rho_{min}\}$ with some $0 < \rho_{max} < 1$ and $\rho_{min} > 0$.
	
	Then Assumptions \ref{ass1}, \ref{ass2}, \ref{ass3} are fulfilled for $\ell$ chosen as in \reff{example_recursion_eq2}. It holds that
	\[
		I(\theta) = \frac{\IE \varepsilon_0^4 - 1}{2}V(\theta) = \frac{\IE \varepsilon_0^4 - 1}{4}\cdot \IE\Big[\frac{\mu(\tilde Z_0(\theta)) \mu(\tilde Z_0(\theta))'}{\langle \theta, \mu(\tilde Z_0(\theta))\rangle^2}\Big],
	\]
\end{example}

\textbf{A simulation study. } Here, we study the behavior of the presented cross validation algorithm for different time series models. We assume that $\varepsilon_t$ is standard Gaussian distributed, and consider
\begin{enumerate}
	\item[(a)] tvAR(1) processes $X_{t,n} = \alpha(\frac{t}{n}) X_{t-1,n} + \sigma(\frac{t}{n}) \varepsilon_t$, with $\alpha(u) = 0.9 \sin(2\pi u)$ and $\sigma(u) = 0.3 \sin(2\pi u) + 0.5$.
	\item[(b)] tvMA(1) processes $X_{t,n} = \sigma(\frac{t}{n}) \varepsilon_t + \alpha(\frac{t}{n}) \sigma(\frac{t-1}{n})\varepsilon_{t-1}$, with $\alpha(u) = 0.9 \sin(2\pi u)$ and $\sigma(u) = 0.3 \sin(2\pi u) + 0.5$.
	\item[(c)] tvARCH(1) processes $X_{t,n} = \sqrt{\alpha_1(\frac{t}{n}) + \alpha_2(\frac{t}{n})X_{t-1,n}^2}\cdot\varepsilon_{t-1}$, with $\alpha_1(u) = 0.2 \sin (2\pi u) + 0.4$ and $\alpha_2(u) = 0.1 \sin(2\pi u) + 0.2$.
	\item[(d)] tvTAR(1) processes  $X_{t,n} = \alpha_1(\frac{t}{n}) X_{t-1,n}^{+} + \alpha_2(\frac{t}{n}) X_{t-1,n}^{-} + \varepsilon_t$, with $\alpha_1(u) = 0.4 \sin(2\pi u)$ and $\alpha_2(u) = 0.5 \cos(2\pi u)$ and $y^{+} := \max\{y,0\}$, $y^{-} := \max\{-y,0\}$ for real numbers $y$.
\end{enumerate}

We performed a Monte Carlo study by generating in each case $N = 2000$ realizations of time series with length $n \in \{200,500\}$. For estimation, we used the weight function $w_{n,h}(\cdot) = \Ii_{[0.05,0.95]}(\cdot)$ which excludes most of the boundary effects and the Epanechnikov kernel $K(x) = \frac{3}{2}(1-(2x)^2)\Ii_{[-\frac{1}{2},\frac{1}{2}]}(x)$. We do not use $w_{n,h}(\cdot) = \Ii_{[\frac{h}{2},1-\frac{h}{2}]}(\cdot)$ since this weight function has poor finite sample properties for large $h$.

We chose $H_n = [0.01,1.0]$ and calculated the cross-validation bandwidth $\hat h$, the ao-bandwidth $h_0$  from Theorem \ref{theorem2} (for models (a)-(c), model (d) does not satisfy the smoothness conditions) and the optimal theoretical bandwidth
\[
	h^{*} = \argmin_{h\in H_n}d_{A}(\hat \theta_h,\theta_0),
\]
Note that $\hat  h, h^{*}$ depend on the current realization while $h_0$ is deterministic and fixed. $h^{*}$ and $h_0$ depend on the unknown true curve $\theta_0(\cdot)$ and are unavailable in practice.

Figure \ref{simulation_plot} shows the results $\hat h, h^{*}$ for the four models respectively. The histograms show the chosen cross validation bandwidths $\hat h$, the bandwidth $h_0$ is marked via a black vertical line. The boxplots show the achieved values of $d_{A}(\hat \theta_h,\theta_0)$ for the different selectors $h \in \{\hat h, h_0, h^{*}\}$ (labeled as 'CV', 'Plugin' and 'Optimal'). Each box contains $50\%$ while the whiskers contain $90\%$ of the values of $d_{A}(\hat \theta_h,\theta_0)$. It can be seen that the cross validation procedure works well even for the case of a time series length of only $n = 200$ if the model is recursively defined (i.e., (a), (b), (d)) while the method needs a larger sample size such that the bandwidths accumulate around $h_0$ in the tvMA case (c). For the models (a),(d) we observe that the distances $d_A$ attained by the cross validation approach are nearly as good as the distances obtained by the optimal selector $h^{*}$ which is remarkable. For the models (b) and (c) the values of $d_A$ associated to $\hat h$ have a higher variance. This can be explained by the higher variance of the maximum likelihood estimators $\hat \theta_h$ in these models. In all cases, the distances produced by the estimator based on the cross validation procedure are of course greater in average, but they still look quite satisfying in our opinion.\\

\begin{figure}[h!]
	\centering
	\begin{tabular}{cc}
		(a), $n = 200$ & (a), $n = 500$\\
		\includegraphics[width=5.5cm]{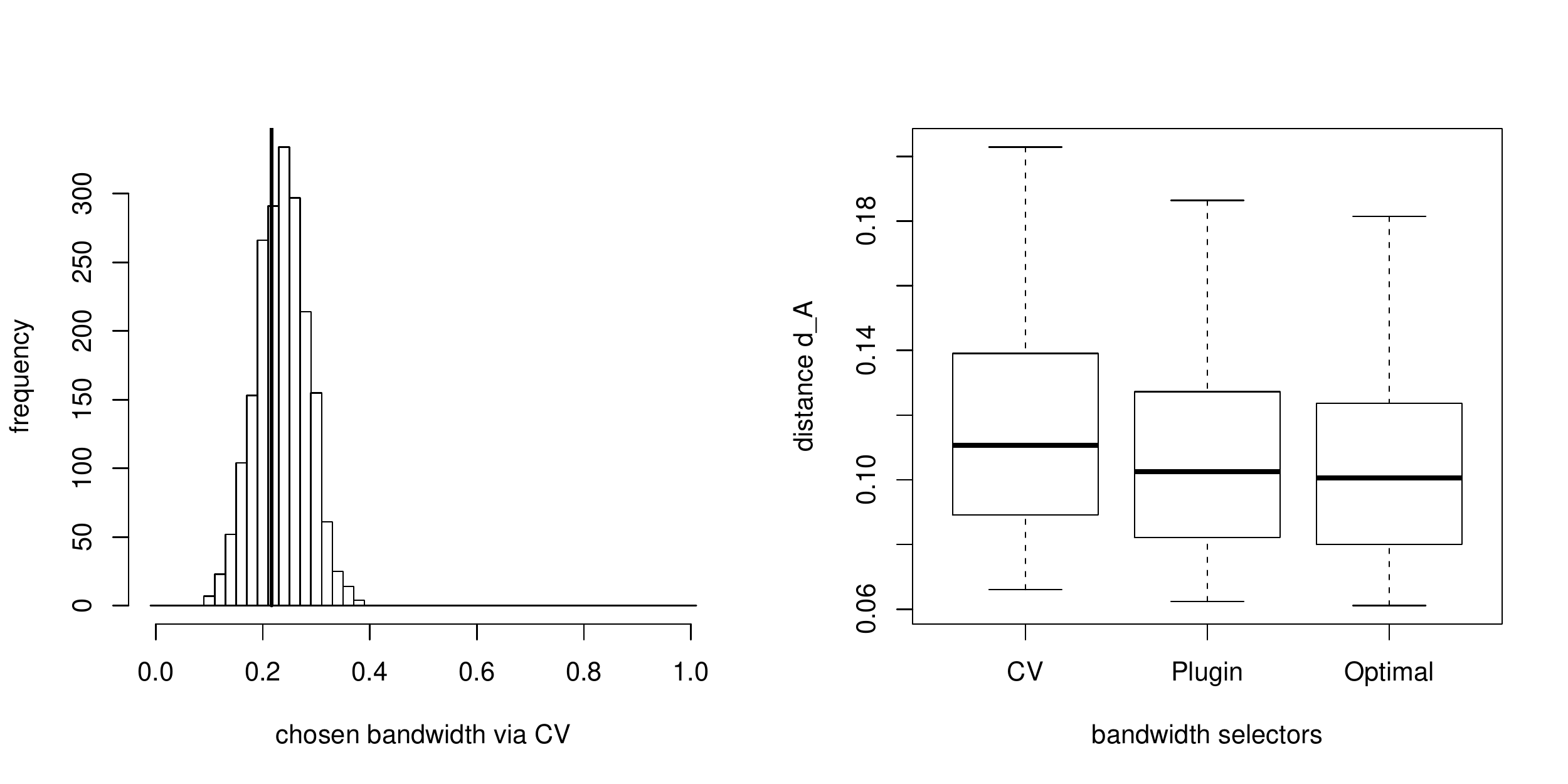} & \includegraphics[width=5.5cm]{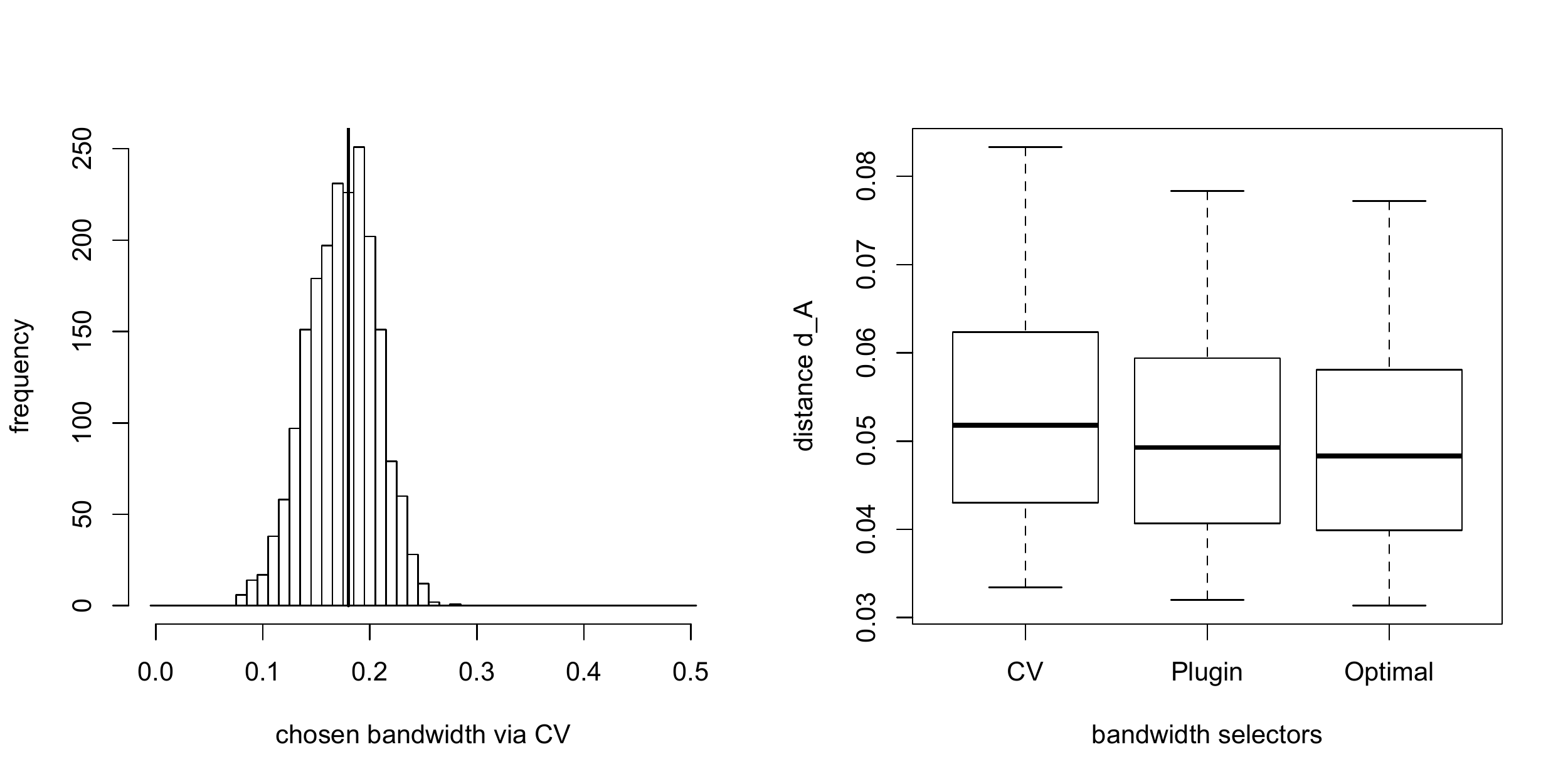}\\
		(b), $n = 200$ & (b), $n = 500$\\
		\includegraphics[width=5.5cm]{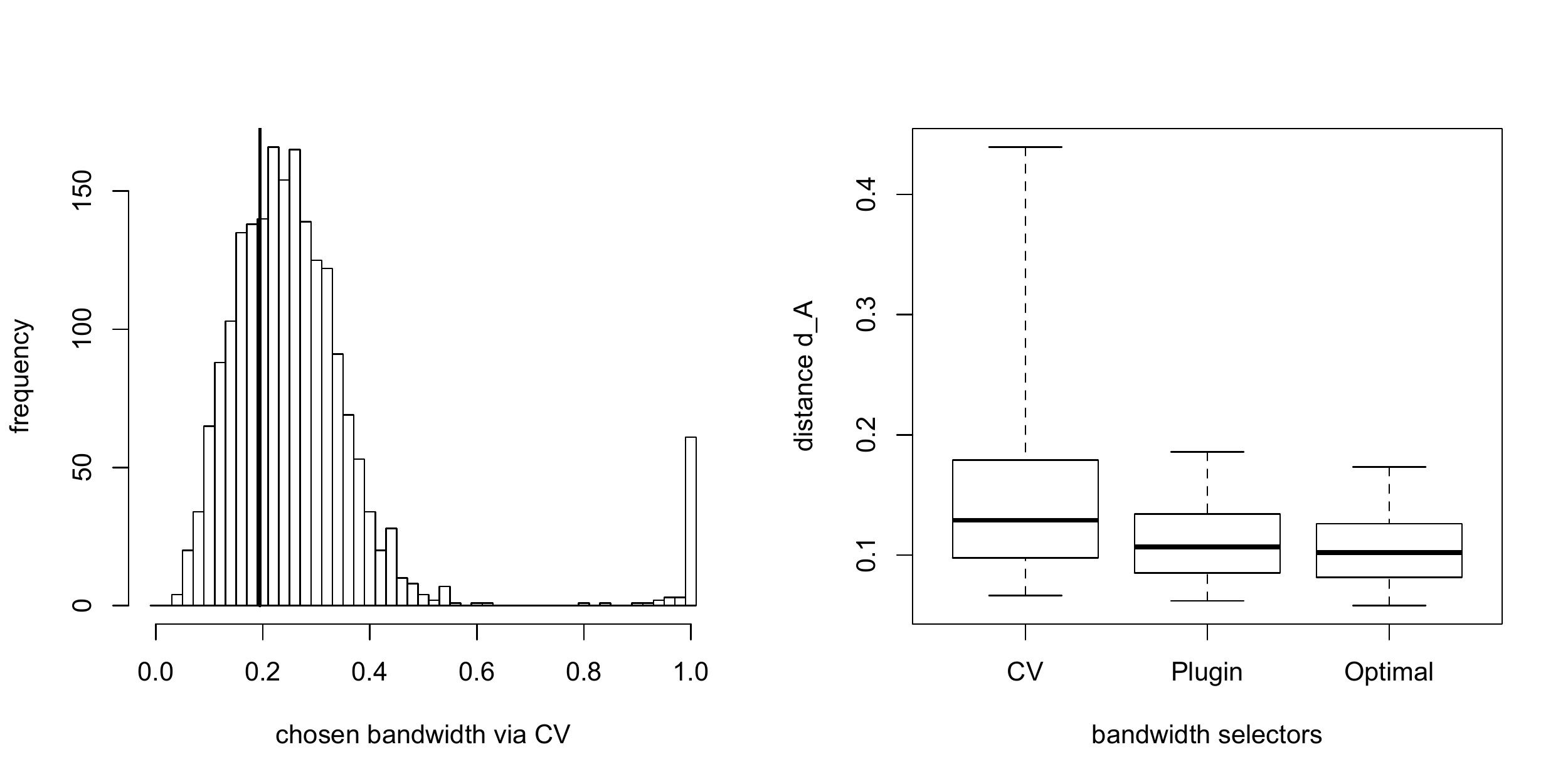} & \includegraphics[width=5.5cm]{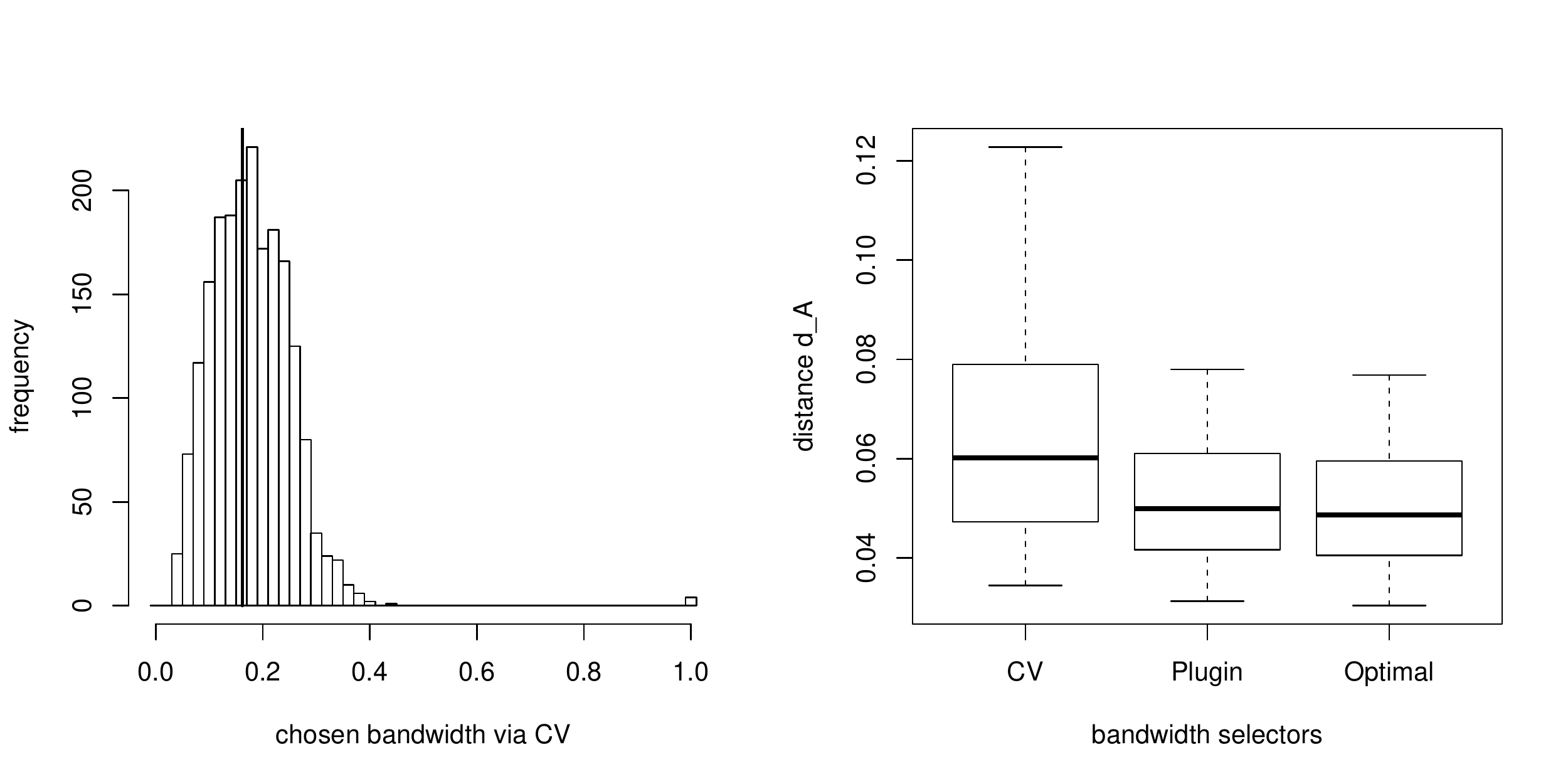}\\
		(c), $n = 200$ & (c), $n = 500$\\
		\includegraphics[width=5.5cm]{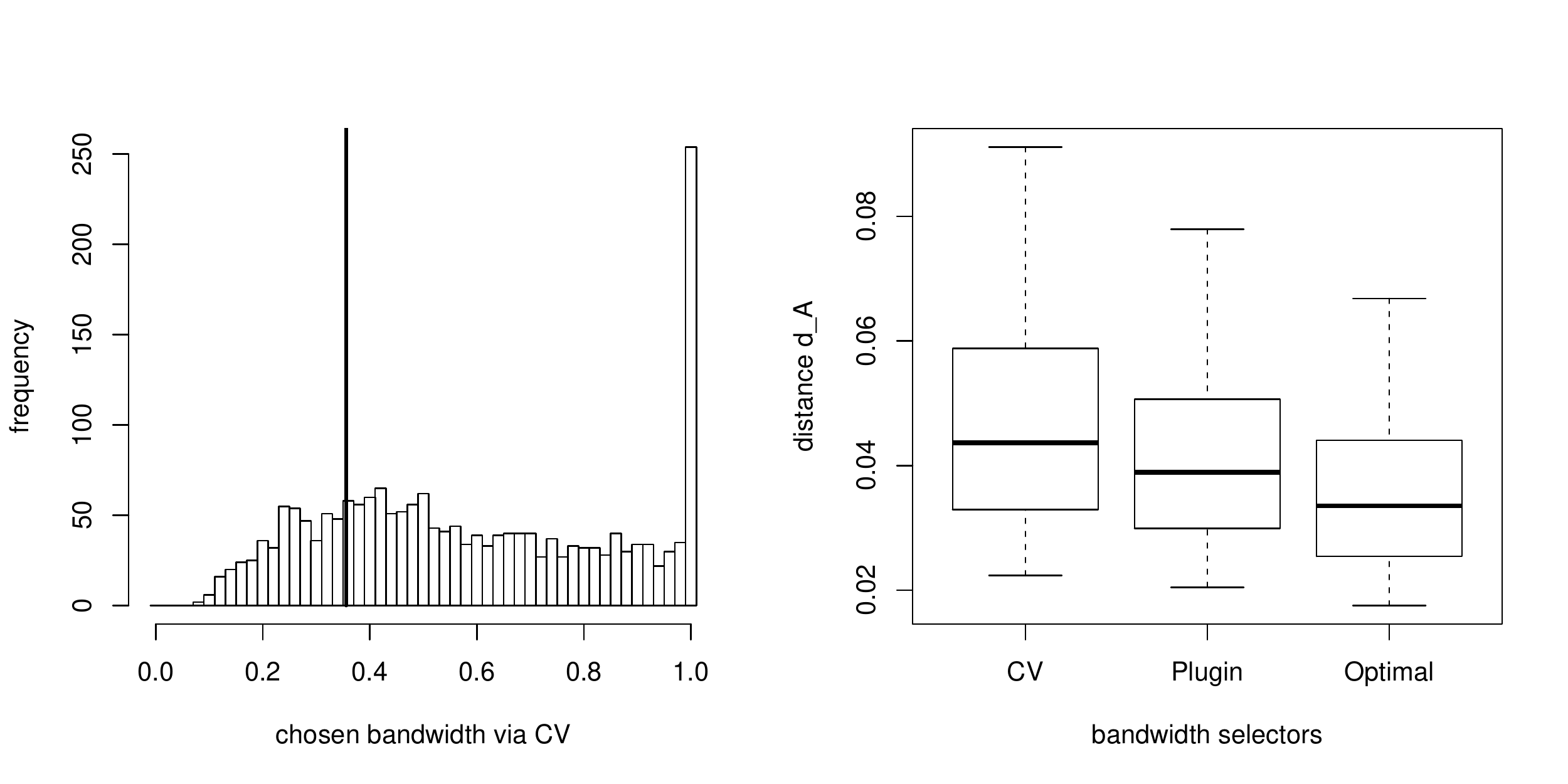} & \includegraphics[width=5.5cm]{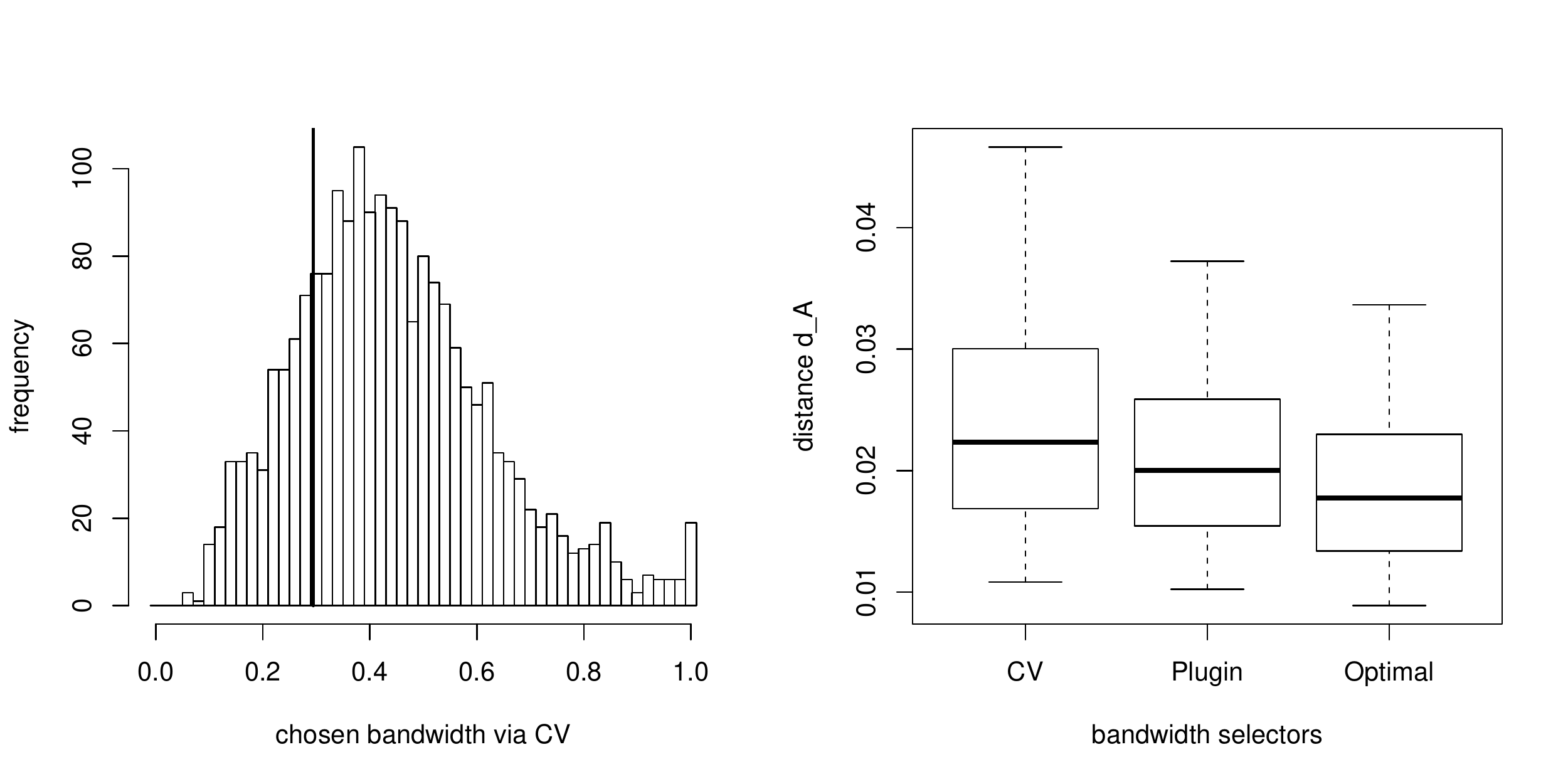}\\
		(d), $n = 200$ & (d), $n = 500$\\
		\includegraphics[width=5.5cm]{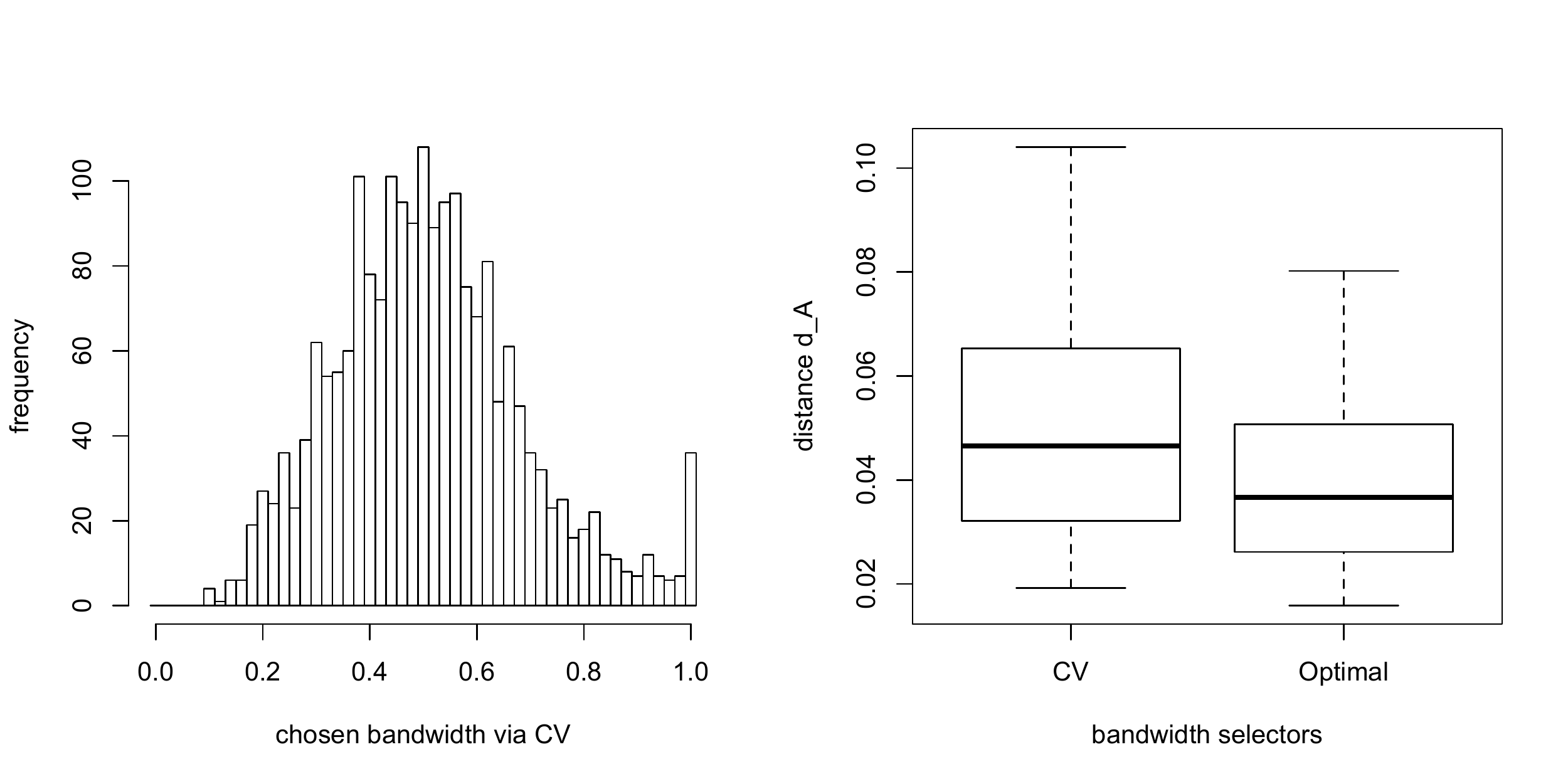} & \includegraphics[width=5.5cm]{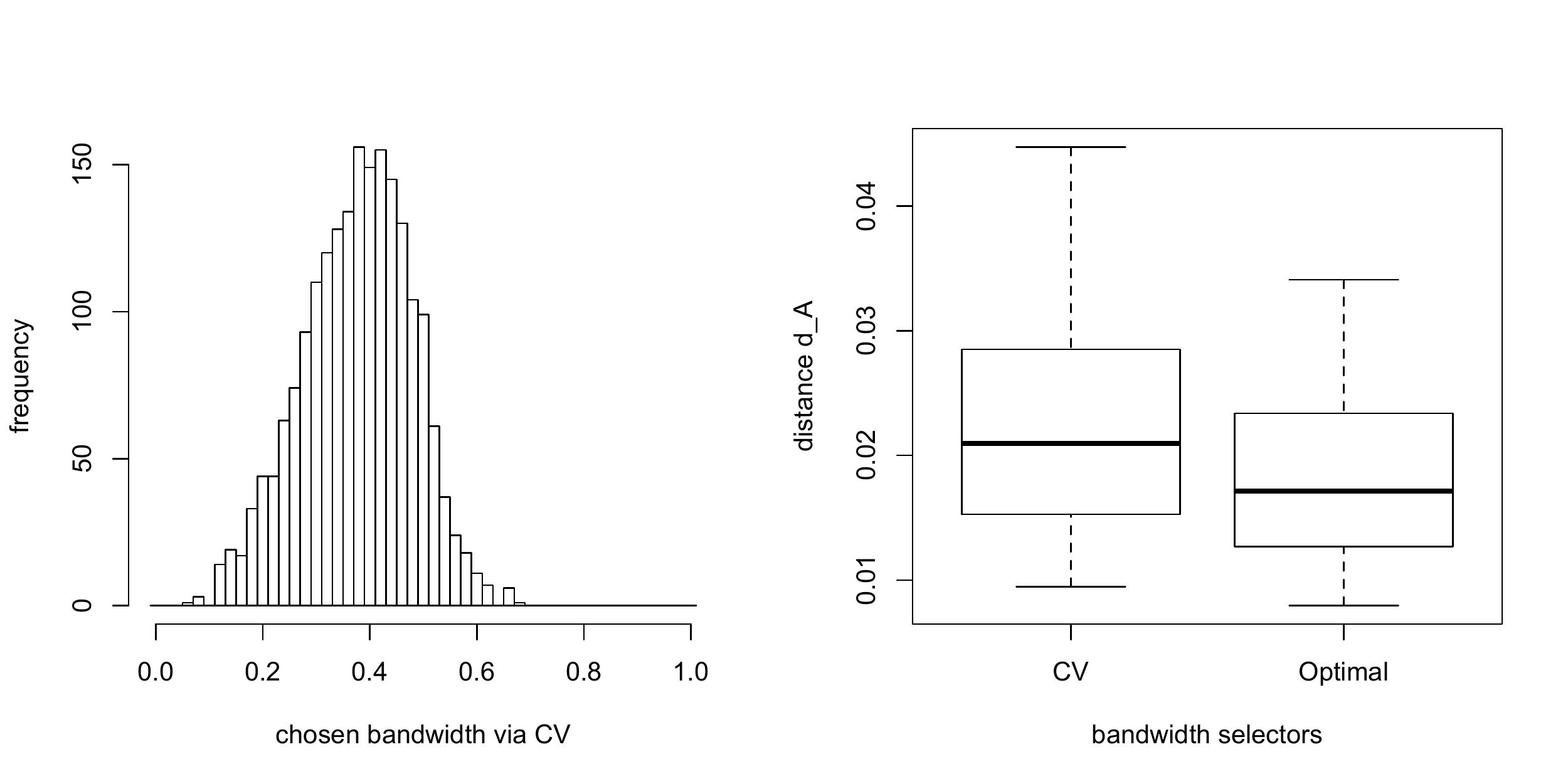}\\
	\end{tabular}
	\caption{Simulation results for the models (a),(b),(c),(d) for time series lengths $n = 200$ (left) and $n = 500$ (right) and $N = 2000$ replications. The left plot shows a histogram of the chosen cross validation bandwidths $\hat h$, the vertical line therein represents the asymptotically optimal bandwidth $h_0$. The right box plots show the values of $d_A(\hat \theta_h,\theta_0)$ achieved for $h \in \{\hat h, h_0, h^{*}\}$.}
	\label{simulation_plot}
\end{figure}

\textbf{Model misspecifications: } We observed in simulations that the performance of the cross validation procedure is robust against the distribution of $\varepsilon_t$, leading to similar results even if $\varepsilon_t$ is uniformly, exponentially or Pareto distributed (meaning that the moment conditions from Assumption \ref{ass1} are violated).

Due to the fact that our cross validation method is a natural generalization of the version for iid regression it works even well if the underlying model itself is misspecified. In the following we estimate parameters with a Gaussian likelihood which assumes that the time series model follows a tvAR(1) model $X_{t,n} = \alpha^{ms}(t/n) X_{t-1,n} + \sigma^{ms}(t/n) \varepsilon_t$, but in fact the underlying model is either tvMA (b) or tvARCH (c). The cross validation method then tries to estimate the minimizer $\theta_0^{ms}(u) = (\alpha^{ms}(u), \sigma^{ms}(u))'$ of $\theta \mapsto L(u,\theta)$, i.e. $\alpha^{ms}(u) = \frac{c(1,u)}{c(0,u)}$ and $\sigma^{ms}(u) = \big(\frac{c(0)^2 - c(1)^2}{c(0)}\big)^{1/2}$ with the covariances $c(k,u) := \IE[\tilde X_0(\theta_0(u)) \tilde X_k(\theta_0(u))]$:
\begin{center}
	\begin{tabular}{|c|c|c|}
\hline
	model & $\alpha^{ms}(u)$ & $\sigma^{ms}(u)$\\
	\hline
	tvMA & $\frac{\alpha(u)}{1+\alpha(u)^2}$ & $\big(\frac{1+\alpha(u)^2 + \alpha(u)^4}{1+\alpha(u)^2}\big)^{1/2}\cdot \sigma(u)$ \\
	\hline
	tvARCH & 0 & $\big(\frac{\alpha_1(u)}{1-\alpha_2(u)}\big)^{1/2}$\\
	\hline
\end{tabular}
\end{center}
To compare the distances, we use $d_A(\hat \theta_h(u),\theta_0^{ms}(u))$ with $V$ from the tvAR(1) model. The simulations are performed in the same way as for the correctly specified case above. In Figure \ref{simulation_plot_misspec} it is seen that even in the misspecified case the bandwidth selector $\hat h$ produces reasonable estimators which are comparable with the optimal bandwidth choice $h^{*}$ in the case of tvMA estimators and still satisfying in the tvARCH case (note that a lot of information is lost due to the fact that $\alpha^{ms}(u) \equiv 0$ in this case).
\begin{figure}[h!]
	\centering
	\begin{tabular}{cc}
		(b), $n = 200$ & (b), $n = 500$\\
		\includegraphics[width=5.5cm]{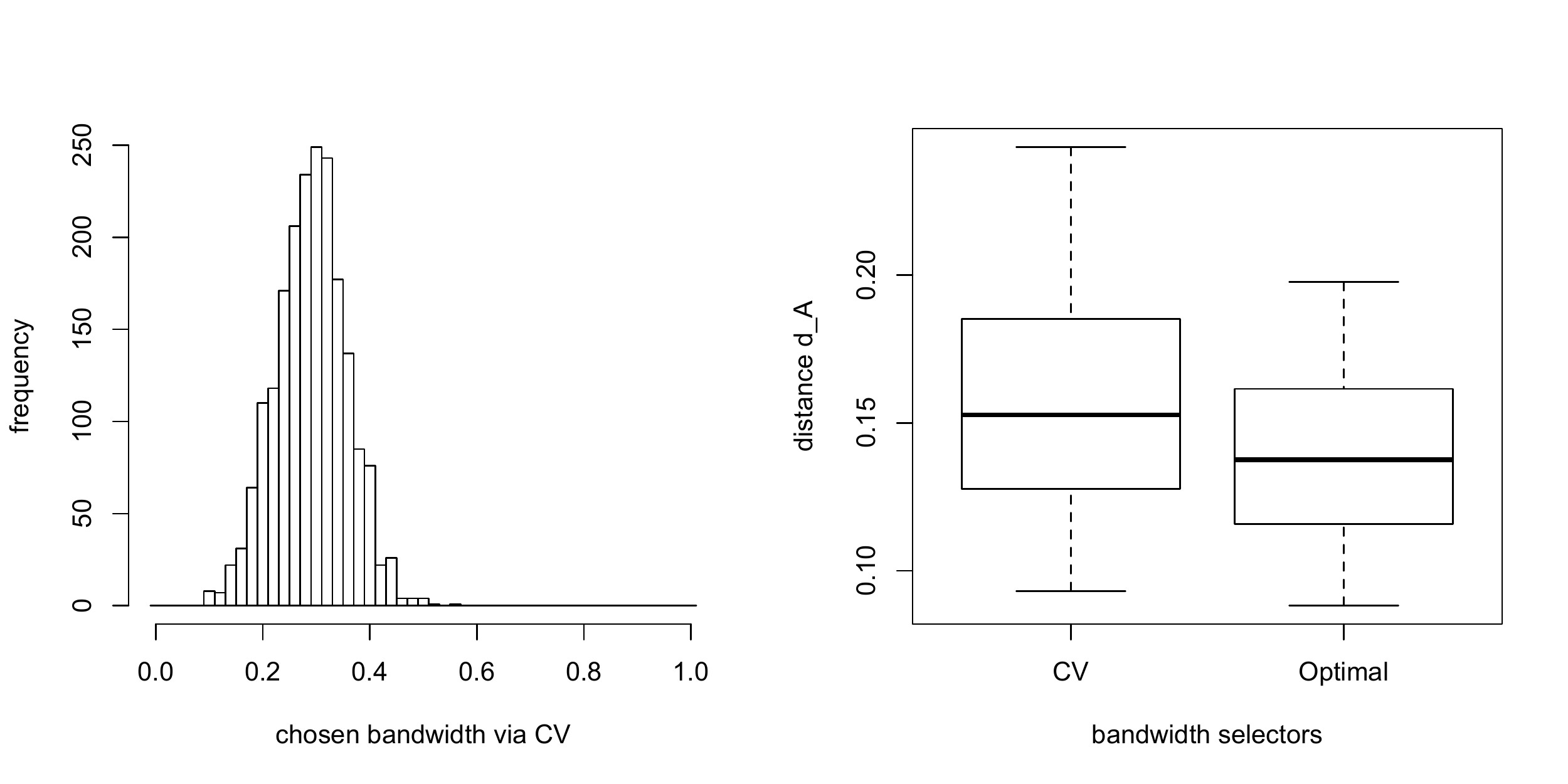} & \includegraphics[width=5.5cm]{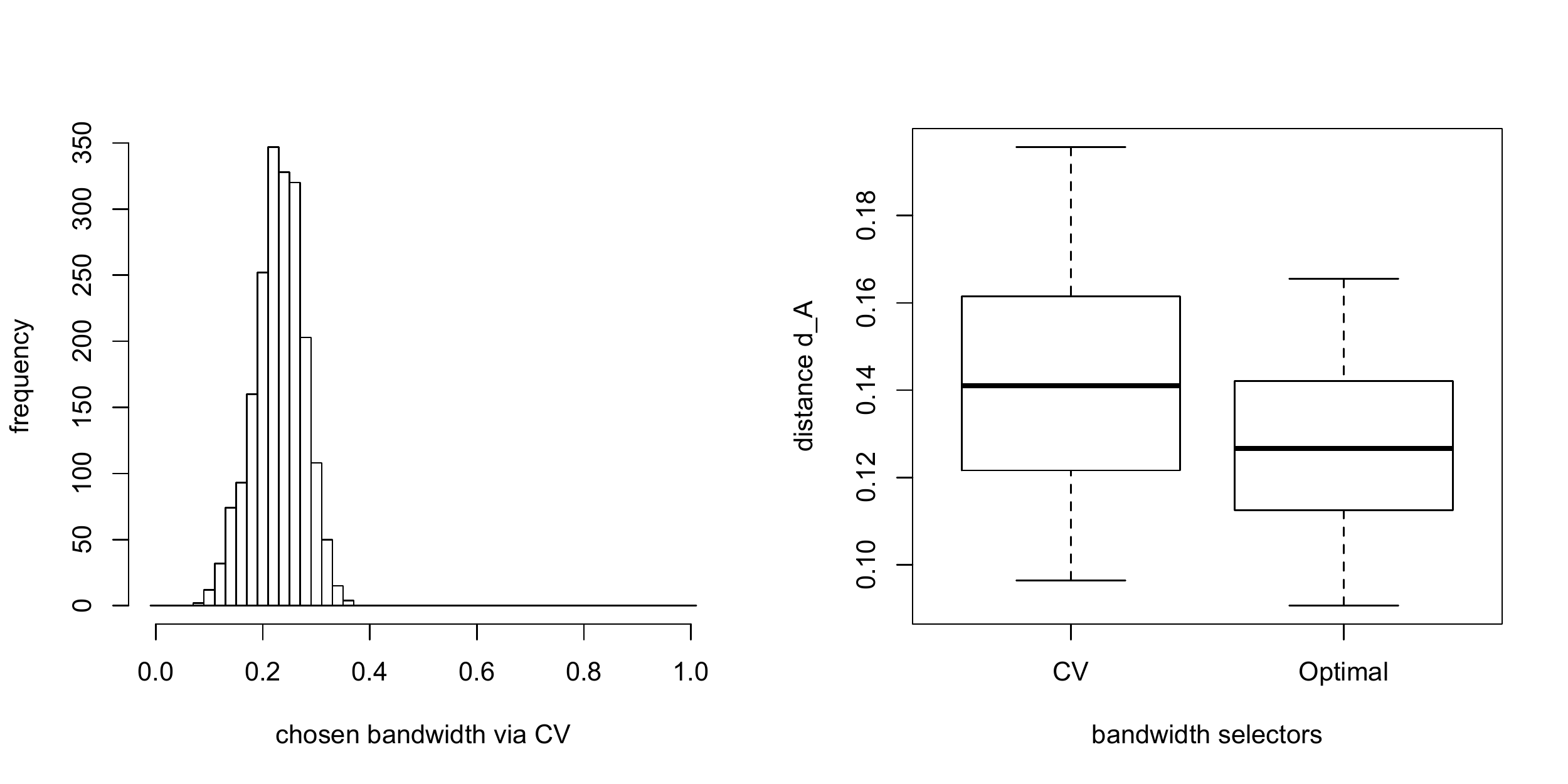}\\
		(c), $n = 200$ & (c), $n = 500$\\
		\includegraphics[width=5.5cm]{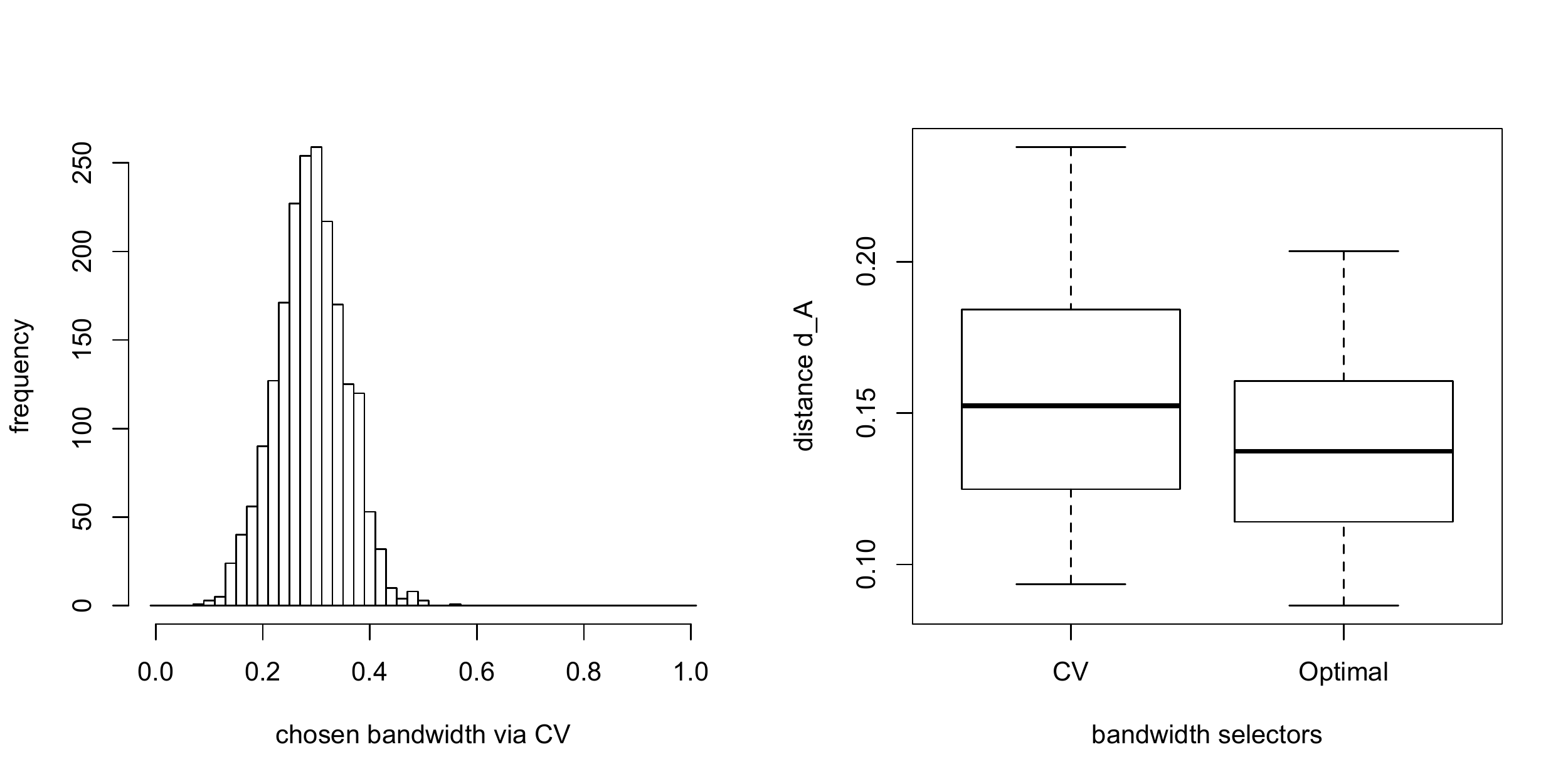} & \includegraphics[width=5.5cm]{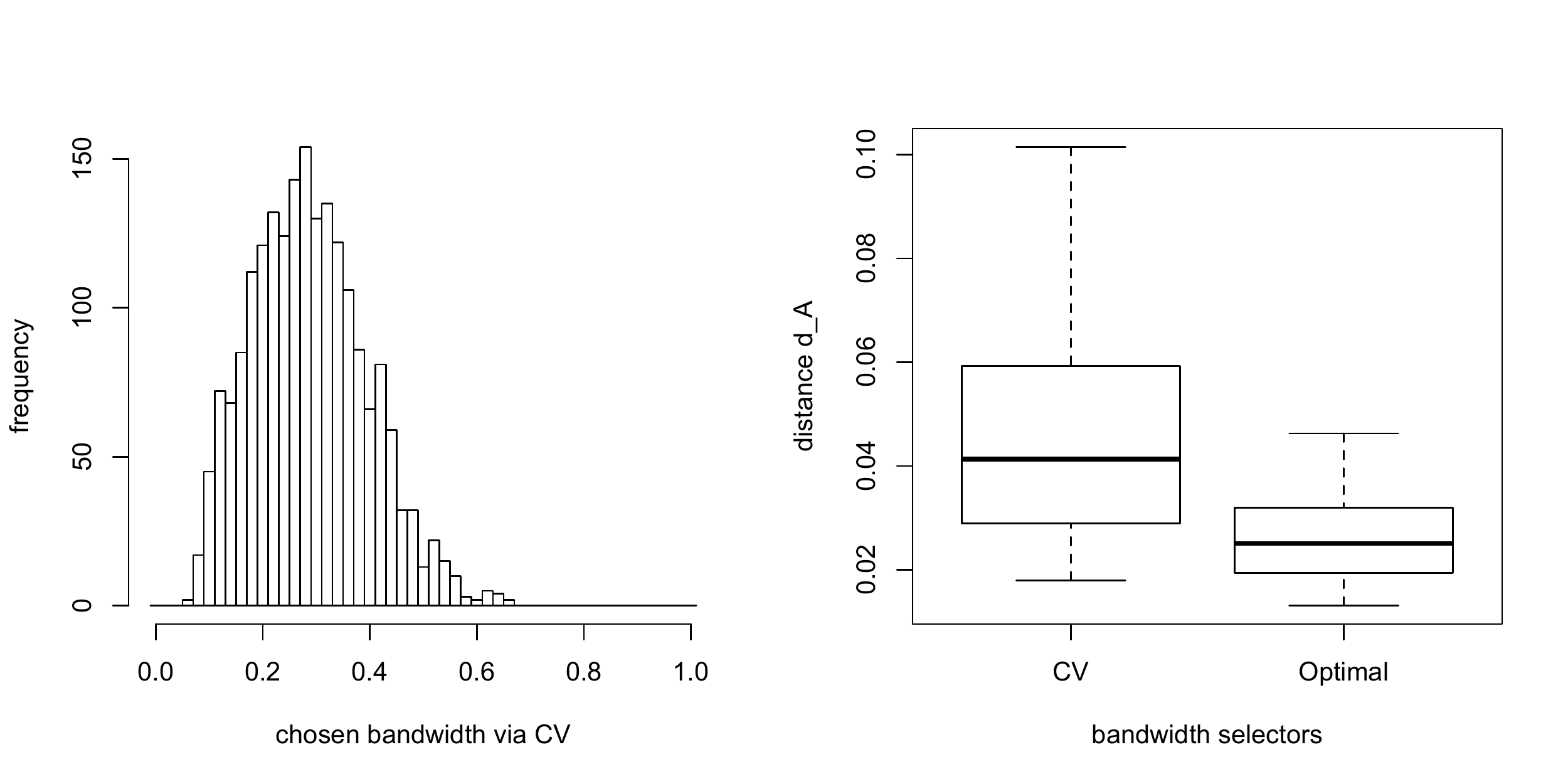}\\
	\end{tabular}
	\caption{Simulation results in the misspecified case: The underlying processes are either (b),(c) but for the estimation it is assumed that a tvAR(1) process is present.}
	\label{simulation_plot_misspec}
\end{figure}

\section{Concluding remarks}
\label{sec5}

In this paper we have introduced a data adaptive bandwidth selector via cross validation which is applicable for a large class of locally stationary processes. An important property of the method is the fact that it does not involve any tuning parameters.

In simulations we have seen that the proposed cross validation method yields nearly optimal bandwidth choices with respect to an Kullback-Leibler type distance measure in the case of correctly specified models and still leads to satisfying results in the case of model misspecification. It remains an open question if a similar cross validation procedure can be defined which is asymptotically optimal with respect to a simple quadratic distance measure (i.e. without a weighting matrix) which would then lead to estimates of $\theta_0$ which do not optimize the prediction properties of the associated model but the estimation quality of the parameter curve $\theta_0$ itself.

We mention that it is not hard to generalize the proposed method and the proofs to multidimensional time series which may be of interest in many practical applications.

An interesting open problem is the adaptive estimation in time series models with several parameter curves coming from different smoothness class, in particular since these curves are not observed separately but via a single time series.

Let us point out the fact that cross validation procedures in general are not stable if applied locally. Thus it remains an open question to find an local adaptive bandwidth selector.

In nonparametric regression there exist also several results on the rate of convergence of $\hat h - h_0$.  Based on the simulations, we conjecture that $n^{1/10}\cdot(\hat h-h_0)$ (with $h_0$ from \reff{h02optimal}) is asymptotically normal if $\theta_0$ is twice continuously differentiable, like \cite{haerdle1988} showed in the iid regression case. This raises the question if there are improved crossvalidation methods like \cite{chiu1991} (via Fourier transform) or \cite{hall1992} (via presmoothing) proved in the iid kernel density estimation case that attain the optimal rate of $n^{1/2}$ if further smoothness assumptions on $\theta_0$ are supposed.


\section{Proofs}
\label{sec6}

In the following, we will use the abbreviation $\IE_0 Z := Z - \IE Z$ for real-valued random variables (or random vectors) $Z$. Note that by Assumption \ref{ass3}, we have that the minimal eigenvalue of $V(\theta)$ is bounded from below by some $\lambda_0$, leading to invertibility of $V(\theta)$ and by equivalence of the norms on $\IR^{p \times p}$ to a uniform upper bound
\[
	V_{-1,max} := \sup_{\theta\in \Theta}\sup_{i,j=1,...,p}|(V(\theta)^{-1})_{ij}|.
\]
Additionally define $V_{1,max} := \sup_{\theta \in \Theta}\sup_{i,j=1,...,p}|V(\theta)_{ij}|$.
In the following, we will use the abbreviation $\tilde X_t(u) := \tilde X_t(\theta_0(u))$ and $\tilde Y_t(u) = (\tilde X_s(u): s \le t)$. Recall that $Y_{t,n}^c = (X_{t,n},...,X_{1,n},0,0,...)$.  Define
\[
	\textstyle\hat L_{n,h}(u,\theta) := \frac{1}{nh}\sum_{s=1}^{n}K\big(\frac{s/n-u}{h}\big)\cdot \l(\tilde Y_s(s/n),\theta)
\]
and by omitting the $t$-th summand, $\hat L_{n,h,-t}(u,\theta)$.

\subsection{Standard approximations}


\begin{lemma}[The stationary crop approximation]\label{lemma_stat}
 Let $g \in \sL(M,\chi,C)$. Put
	\[
		\textstyle S_{n,h}(g(\cdot,\theta),u) := \frac{1}{n}\sum_{t=1}^{n}K_h\Big(\frac{t}{n}-u\Big)\cdot \big\{g(Y_{t,n}^c, \theta) - g(\tilde Y_{t}(t/n),\theta)\big\}.
	\]
Suppose that Assumption \ref{ass1} holds. Assume that $\chi_j = O(j^{-(2+\delta)})$ for some $\delta > 0$.
Then for all $q \ge 1$ there exists a constant $C_{S,q} > 0$ not depending on $n,h$ such that 
\[
	\textstyle\big\|\sup_{\theta \in \Theta}|S_{n,h}(g(\cdot,\theta),u)|\big\|_q \le C_{S,q}\cdot (nh)^{-1}.
\]
\end{lemma}
The proof follows from Hoelder's inequality. Details can be found in the appendix.

\begin{lemma}[Weak bias approximation]\label{lemma_bias_weak} Let $g \in \sL(M,\chi,C)$. Define
	\[
		\textstyle B_{n,h}(g(\cdot,\theta),u) := \frac{1}{n}\sum_{t=1}^{n}K_h\Big(\frac{t}{n}-u\Big)\cdot \big\{g(\tilde Y_{t}(t/n),\theta) - g(\tilde Y_{t}(u),\theta)\big\}.
	\]
Suppose that Assumption \ref{ass1} holds. Assume that $\theta_0$ is Hoelder continuous with exponent $\beta$ in each component. Then there exist constants $C_{bias} > 0$ such that for all $u \in  [0,1]$:
\[
	\textstyle\big\|\sup_{\theta \in \Theta}|B_{n,h}(g(\cdot,\theta),u)| \big\|_2 \le C_{bias} h^{\beta \wedge 1}.
\]
\end{lemma}
\begin{proof}[Proof of Lemma \ref{lemma_bias_weak}:]
	Since $\theta_0$ is Hoelder continuous in each component, there exists $\tilde L > 0$ such that $|\theta_{0,i}(u) - \theta_{0,i}(v)| \le \tilde L|u-v|^{\beta}$ for all $i = 1,...,p$ and $u,v\in [0,1]$. Thus $\|\tilde X_t(t/n) - \tilde X_t(u)\|_{2M} \le C_A |\theta_0(t/n) - \theta_0(u)|_1 \le C_A d\cdot \tilde L |t/n - u|^{\beta \wedge 1}$. By Hoelder's inequality,
	\begin{align}
		& \big\|\sup_{\theta \in \Theta}|g(\tilde Y_t(t/n),\theta) - g(\tilde Y_t(u),\theta)|\big\|_2\nonumber\\
		&\le C_1 \big(1 + 2(D_{2M} |\chi|_1)^{M-1}\big) \cdot \| \tilde X_0(t/n) - \tilde X_0(u)\|_{2M}.\label{lemma_bias_weak_eq1}
\end{align}
	Since $\frac{1}{n}\sum_{t=1}^{n}\big|K_h\big(\frac{t}{n}-u\big)\big|\cdot \| \tilde X_0(t/n) - \tilde X_0(u)\|_{2M} \le |K|_{\infty}C_A p L \cdot h^{\beta \wedge 1}$, we obtain the result.
\end{proof}

\subsection{The bias-variance decomposition of $\nabla L_{n,h}(u,\theta_0(u))$}

The proof of the next Lemma \ref{lemma_expectation_expansion} is purely analytical and is deferred to the appendix.

\begin{lemma}[Expansion of expectations]\label{lemma_expectation_expansion} Let Assumption \ref{ass1}, \ref{ass3} and \ref{ass5} hold. Assume that $g: \IR^{\infty} \to \IR$ is twice continuously partially differentiable and $\partial_{i}\partial_{j} g \in \sL_{\infty}(M-2,\tilde \chi,\tilde\psi_1(i)\tilde\psi_2(j))$ for each component of $\partial^{2}g$, where $\tilde \psi_1, \tilde \psi_2$ are absolutely summable sequences.\\
Furthermore assume that $|\partial_{i} g(0)| \le \tilde \psi_1(i)$ and $|\partial_i \partial_j g(0)| \le \tilde \psi_1(i) \tilde \psi_2(j)$ for all $i,j \ge 1$. Then it holds for $u\in [0,1]$ and $\xi \in \IR$ that
\begin{equation}
	\IE g(\tilde Y_0(u+\xi)) = \IE g(\tilde Y_0(u)) + \xi \cdot \IE \partial_u g(\tilde Y_0(u)) + \frac{\xi^2}{2}\cdot \IE \partial_u^2 g(\tilde Y_0(u)) + \frac{\xi^2}{2}\cdot R(u,\xi),\label{expectation_expansion_eq1}
\end{equation}
where $R(u,\xi) := \int_{0}^{1}\big\{\IE\partial_u^2 g(\tilde Y_0(u+\xi s))  - \IE \partial_u^2 g(\tilde Y_0(u))\big\} \dif s = o(1)$ ($\xi \to 0$) uniformly in $u \in [0,1]$, and all expressions exist.
\end{lemma}

We now summarize the results about the bias-variance decomposition of $\nabla L_{n,h}(u,\theta_0(u))$. The following Proposition is obtained as a corollary from Lemma \ref{lemma_stat}, Lemma \ref{lemma_bias_weak} and Lemma \ref{lemma_expectation_expansion}. Details of the proof are deferred to appendix.

\begin{proposition}[Bias-Variance decomposition of $\nabla L_{n,h}(u,\theta_0(u))$]\label{misedecomposition}
	Let Assumptions \ref{ass1}, \ref{ass2}, \ref{ass3} and \ref{ass4} hold.
	\begin{enumerate}
		\item[(i)] Decomposition: Let $\tr\{\cdot\}$ denote the trace of a matrix, $\mu_K := \int K(x)^2 \dif x$,
			\begin{eqnarray}
		v_h(u) &:=& \frac{\mu_K}{nh}\tr\{V^{-1}(\theta_0(u)) I(\theta_0(u))\},\label{var_form}\\
		b_h(u) &:=& \frac{1}{h}\int_0^{1}K\big(\frac{v-u}{h}\big)\cdot \IE  \nabla \ell(\tilde Y_t(v),\theta_0(u)) \dif v.\label{bias_form}
	\end{eqnarray}
	Set $B_h := \int_{0}^{1}|b_h(u)|_{V(\theta_0(u))^{-1}}^2 w_{n,h}(u) \dif u$, and $V_h := \int_{0}^{1} v_h(u) w_{n,h}(u) \dif u$. Then it holds that
	\[
		\sup_{h\in H_n}\ (nh)\cdot\Big|d_M^{*}(\hat \theta_h,\theta_0) - \big\{ V_h + B_h\big\}\Big| \to 0.
	\]
	\item[(ii)] Put $\hat B_h := \int_{0}^{1}|\IE \hat L_{n,h}(u,\theta_0(u))|_{V(\theta_0(u))^{-1}}^2 w_{n,h}(u) \dif u$ and define the discrete bias terms $\hat B_h^{dis} := \frac{1}{n}\sum_{t=1}^{n}|\IE \hat L_{n,h}(t/n,\theta_0(t/n))|^2_{V(\theta_0(t/n))^{-1}}w_{n,h}(t/n)$ and $B_h^{dis} := \frac{1}{n}\sum_{t=1}^{n}| \IE L_{n,h}(t/n,\theta_0(t/n))|^2_{V(\theta_0(t/n))^{-1}}w_{n,h}(t/n)$. Then it holds for $\tilde B_h \in \{\hat B_h, \hat B_h^{dis}, B_{h}^{dis}\}$ that
	\begin{equation}
		\sup_{h\in H_n}\ (nh)\cdot |B_h - \tilde B_h| \to 0.\label{bias_discrete_approx}
	\end{equation}
	\item[(iii)] Bias expansion: Suppose additionally that Assumption \ref{ass5} holds. Then it holds uniformly in $u \in [\frac{h}{2},1-\frac{h}{2}]$ that
	\[
		b_h(u) = \frac{h^2}{2} \cdot \int y^2 K(y) \dif y\cdot \IE\big[\partial_u^2 \nabla \l(\tilde Y_t(\theta_0(u)),\theta)\big]\big|_{\theta = \theta_0(u)} + o(h^2).
	\]
	\end{enumerate}
\end{proposition}

\subsection{Uniform convergence results and moment inequalities for the local likelihood $L_{n,h}(u,\theta)$}
\label{section_uniformconvergence}
In this section we show the uniform convergence of empirical processes of $X_{t,n}$ towards their expectations. We give convergence rates and prove the uniform consistency (w.r.t. $u$ and $h$) of the maximum likelihood estimator $\hat \theta_h(u)$ towards $\theta_0(u)$. For some $\phi:[0,1]\to \IR$ and $g \in \sL(M,\chi,C)$, define
\begin{eqnarray*}
	 \textstyle E_n(\phi,g,\theta) &:=& \textstyle\frac{1}{n}\sum_{t=1}^{n}\phi\big(\frac{t}{n}\big) \cdot g(Y_{t,n}^c,\theta),\\
	\textstyle  \hat E_n(\phi,g,\theta) &:=& \textstyle\frac{1}{n}\sum_{t=1}^{n}\phi\big(\frac{t}{n}\big) \cdot g(\tilde Y_{t}(t/n),\theta).
\end{eqnarray*}

The proof of the next Lemma \ref{pmomente} as well as the proofs of Lemma \ref{da_approx_0}, \ref{lemma_approxmart} use Lemma \ref{hilfslemma1} which is deferred to the appendix due to its complexity. It allows to bound moments of linear, quadratic and cubic forms of functions of locally stationary processes. For instance we obtain bounds $\| \sum_{t=1}^{n}a_t V_{t,n}^{(1)}\|_q \le \tilde C_q \big(\sum_{t=1}^{n}a_t^2\big)^{1/2}$ and $\| \sum_{s,t=1}^{n}a_{s,t} V_{t,n}^{(1)}(s) V_{s,n}^{(2)}(t)\|_q \le \tilde C_q \big(\sum_{t=1}^{n}a_{s,t}^2\big)^{1/2}$ for deterministic numbers $a_t$ or $a_{s,t}$ and processes $V_{t,n}^{(1)}$, $V_{t,n}^{(1)}(s)$ and $V_{s,n}^{(2)}(t)$ which fulfill dependence conditions and have bounded variation with respect to the indices $(s)$ and $(t)$.

\begin{lemma}[Moment inequality]\label{pmomente}

	Let Assumption \ref{ass1}, \ref{ass2} hold. Let $g \in \sL(M,\chi,C)$. Then, for all $\theta \in \theta$,
	\[
		\textstyle\big\|  \IE_0 \hat E_n(\phi,g,\theta)\big\|_q \le \rho_{1,\psi C_{\delta,\cdot},q}\cdot n^{-1/2} \big(\frac{1}{n}\sum_{t=1}^{n}\phi\big(\frac{t}{n}\big)^2 \big)^{1/2},
	\]
	where $\psi_q(t) = \sum_{j=0}^{t-1}\chi_j \delta_{qM}(t-j)$, $C_{\delta,q} := C_1(1+2(D_{qM}|\chi|_1)^{M-1})$ and $\rho_{1,\psi C_{\delta,\cdot},q}$ is defined in Lemma \ref{hilfslemma1}.
\end{lemma}
\begin{proof}[Proof of Lemma \ref{pmomente}]
	Note that for $g \in \sL(M,\chi,C)$, we have by Hoelder's inequality for all $\theta \in \Theta$, $u\in [0,1]$:
	\begin{eqnarray}
		\delta_{q}^{g(\tilde Y(u),\theta)}(t) &:=& \| g(\tilde Y_{t}(u),\theta) - g(\tilde Y_{t}(u)^{*},\theta)\|_q\nonumber\\
		&\le& C_1\sum_{j=0}^{t-1}\chi_j \delta_{qM}(t-j) \cdot \big(1 + 2 \big(D_{qM} |\chi|_1\big)^{M-1}\big).\label{dependence_measure_std}
	\end{eqnarray}
	Since $\sum_{t=1}^{\infty}\Big(\sum_{j=0}^{t-1}\chi_j \delta_{qM}(t-j)\Big) \le \sum_{j=0}^{\infty}\chi_j \cdot \sum_{t=1}^{\infty}\delta_{qM}(t) < \infty$, Lemma \ref{hilfslemma1} is applicable and we obtain the assertion.
\end{proof}

\begin{lemma}[Continuity properties of localized sums]\label{lemma_stetigkeit} Let Assumption \ref{ass1}, \ref{ass2} and \ref{ass4} hold. Let $g \in \sL(M,\chi,C)$. Then it holds for arbitrary $\theta,\theta' \in \Theta$, $u,u' \in [0,1]$:
\begin{eqnarray*}
	\big\|E_n(K_h(\cdot-u),g,\theta) - E_n(K_h(\cdot-u'),g,\theta')\big\|_q &\le& C_{-,q}\big\{\frac{|u-u'|}{h} + |\theta-\theta'|_1\big\},\\
	\|E_n(K_h(\cdot-u),g,\theta)\|_q &\le& C_{\infty,q},
\end{eqnarray*}
where $C_{\infty,q} = C_{\infty,q}(g)$, $C_{-,q} = C_{-,q}(g)$ depend solely on $|g|_{\infty} := \sup_{\theta \in \Theta}|g(0,\theta)| < \infty$ and $M,\chi,C$.
\end{lemma}
\begin{proof}

Define $|g|_{\infty} := \sup_{\theta \in \Theta}|g(0,\theta)| <\infty$. Since $g \in \sL(M,\chi,C)$, it holds that for all $\theta \in \Theta$, $|g(y, \theta)| \le |g|_{\infty} + C_1|y|_{\chi,1}\cdot \big(1 + |y|_{\chi,1}^{M-1}\big)$. By Young's inequality, it holds that $a \le \frac{1}{M}\big((M-1) + a^{M}\big)$ for $M \ge 1$ and nonnegative real numbers $a$, which shows that there exists a constant $C_{\infty} = C_{\infty}(g) > 0$ such that for $y \in \IR^{\IN}$,
	\begin{equation}
		|g(y,\theta)| \le C_{\infty}\big(1 + |y|_{\chi,1}^M\big).\label{proof_uniform_eq1}
	\end{equation}
	We can use the bound \reff{proof_uniform_eq1} to see that uniformly in $u,\theta,t$ it holds that $\|E_{n}(K_h(\cdot-u),g,\theta)\|_q \le |K|_{\infty}C_{\infty}(1+(|\chi|_1 D_{qM})^{M})=: C_{\infty,q}$.
	It holds that
	\begin{eqnarray*}
		&& \textstyle\|E_n(K_h(\cdot-u),g,\theta) - E_n(K_h(\cdot-u'),g,\theta')\|_q\\
		&\le& \textstyle\frac{1}{nh}\big\|\sum_{t=1}^{n}\big|K\big(\frac{t/n-u}{h}\big) - K\big(\frac{t/n-u'}{h}\big)\big|\cdot |g(Y_{t,n}^c,\theta)|\big\|_q\\
		&&\textstyle\quad\quad + \frac{1}{nh}\big\|\sum_{t=1}^{n}\big|K\big(\frac{t/n-u'}{h}\big)\big|\cdot |g(Y_{t,n}^c,\theta) - g(Y_{t,n}^c,\theta')|\big\|_q\\
		&\le&\textstyle \big(2 L_K C_{\infty}\frac{|u-u'|}{h} +  |K|_{\infty}C_2 |\theta - \theta'|_1\big)\cdot \big(1 + (|\chi|_1 D_{qM})^M\big).
	\end{eqnarray*}
\end{proof}

\begin{lemma}[Uniform convergence and weak bias expansion]\label{lemma_uniformconvergence}
	Let Assumption \ref{ass1}, \ref{ass2} and \ref{ass4} hold. Let $g \in \sL(M,\chi,C)$. Then for all $0 < \alpha < \frac{1}{2}$, it holds almost surely that
	\begin{equation}
		\sup_{h\in H_n}\sup_{u \in \supp(w_{n,h})}\sup_{\theta \in \theta}\ (nh)^{\frac{1}{2}-\alpha}\cdot \big|\IE_0 E_n(K_h(\cdot - u), g, \theta)\big| \to 0.\label{lemma_uniformconvergence_eq1}
	\end{equation}
	and there exists a constant $C_{S} > 0$ independent of $u,h,\theta$ such that for all $u \in \supp(w_{n,h}) \subset [\frac{h}{2},1-\frac{h}{2}]$:
	\begin{equation}
		|\IE E_n(K_h(\cdot-u),g,\theta) - \IE g(\tilde Y_0(u),\theta)| \le C_S \big\{ (nh)^{-1} + h^{\beta \wedge 1}\big\}.\label{lemma_uniformconvergence_eq2}
	\end{equation}
\end{lemma}
\begin{proof}[Proof of Lemma \ref{lemma_uniformconvergence}]
Define
\begin{eqnarray*}
	f(\xi) &:=& E_n(K_h(\cdot - u), g, \theta) - \IE E_n(K_h(\cdot - u), g, \theta),\\
	\hat f(\xi) &:=& \hat E_n(K_h(\cdot - u), g, \theta) - \IE \hat E_n(K_h(\cdot - u), g, \theta),
\end{eqnarray*}
where $\xi = (h,u,\theta) \in \Xi_{n} := \{(h,u,\theta): h \in H_n, u\in \supp(w_{n,h}), \theta \in \Theta\}$. For each $r > 0$, we can find a space $\Xi_n'$ with $\# \Xi_n' < c_{\gamma} n^{\gamma}$ such that the compact space $\Xi_n$ is approximated in the following way: for each $\xi = (h,u,\theta) \in \Xi_{n}$ there is a $\xi' = (h',u',\theta') \in \Xi_n'$ such that $|\xi - \xi'|_1 \le c_r n^{-r}$. Now fix some $\delta > 0$. For $0 < \alpha < \frac{1}{2}$, we obtain
\begin{eqnarray*}
	&& \textstyle\IP\big( \sup_{\xi \in \Xi_n} (nh)^{\frac{1}{2}-\alpha}|f(\xi)| > \delta\big)\\
	&\le& \textstyle\IP\big( \sup_{\xi' \in \Xi_n'} (nh)^{\frac{1}{2}-\alpha}| f(\xi')| > \frac{\delta}{2}\big)\\
	&&\textstyle\quad\quad + \IP\big(\sup_{\xi \in \Xi_n, \xi' \in \Xi_n', |\xi - \xi'|_1 \le c_r n^{-r}} (nh)^{\frac{1}{2}-\alpha}|f(\xi) - f(\xi')| > \frac{\delta}{2}\big)\\
	&=:& W_{1} + W_2.
\end{eqnarray*}
Our goal is to bound $W_1$, $W_2$ by absolutely summable sequences in $n$. Then the assertion follows from Borel-Cantelli's lemma. From Lemma \ref{pmomente} we obtain:
	\begin{eqnarray}
		\|\hat f(\xi)\|_q &\le& \rho_{1,\psi C_{\delta,\cdot},q} (nh)^{-1/2} \Big(\frac{1}{nh}\sum_{t=1}^{n}K\Big(\frac{t/n-u}{h}\Big)^2\Big)^{1/2}\nonumber\\
		&\le& \rho_{1,\psi C_{\delta,\cdot},q} |K|_{\infty} (nh)^{-1/2}.\label{likelihood_rate}
	\end{eqnarray}
	Furthermore we have for $q \ge 1$ by Lemma \ref{lemma_stat} that
	\begin{eqnarray*}
		\|f(\xi) - \hat f(\xi)\|_q &\le& 2\|E_{n}(K_h(\cdot-u),g,\theta) - \hat E_n(K_h(\cdot-u),g,\theta)\|_q\\
		&\le& C_{S,q}\cdot (nh)^{-1}.
	\end{eqnarray*}
	By Markov's inequality, it follows that
	\begin{eqnarray*}
		\sup_{\xi \in \Xi_n}\IP\left( (nh)^{\frac{1}{2}-\alpha}|f(\xi)| > \delta/2\right)
		\le \Big(\frac{\rho_{1,\psi C_{\delta,\cdot},q}|K|_{\infty} + C_{S,q}}{\delta/2}\Big)^q  \cdot \sup_{h \in H_n} (nh)^{-\alpha q},
	\end{eqnarray*}
	and thus for $q$ large enough, $W_1 \le \# \Xi_n' \cdot \sup_{\xi\in \Xi_n}\IP\big( (nh)^{\frac{1}{2}-\alpha}|f(\xi)| > \delta/2\big) \le \big(\frac{\rho_{1,\psi C_{\delta,\cdot},q}|K|_{\infty} + C_{S,q}}{\delta/2}\big)^q \cdot n^{q-\alpha \delta q}$ is bounded by an absolutely summable sequence in $n$.
	
	We now discuss $W_2$. Define $Z_n := 1 + \frac{1}{n}\sum_{t=1}^{n}|Y_{t,n}^c|_{\chi,1}^{M}$. Using the inequality \reff{proof_uniform_eq1} and the Lipschitz property of $K$, we obtain
	\begin{eqnarray*}
		|f(\xi) - f(\xi')| 
		&\le& \frac{2}{\underline{h}^3}\big[L_K + |K|_{\infty}\big] \max\{C_2,C_{\infty}\}\cdot |\xi - \xi'|_1 \cdot Z_n.
	\end{eqnarray*}
	Since $\underline{h} \ge c_0 n^{\delta-1}$ and $\overline{h} \le c_1 n^{-\delta}$ (cf. Assumption \ref{ass4}), we have shown that $(nh)^{\frac{1}{2}-\alpha}|f(\xi) - f(\xi')| \le C(n) \cdot \i\xi - \xi'\i_1 \cdot Z_n$, where the deterministic $C(n)$ grows only polynomially fast in $n$. Choose $r$ large enough and some constants $\gamma_r, C_r > 0$ such that $ C(n) c_r n^{-r} \le C_r n^{-(1+\gamma_r)}$ for all $n\in\IN$, then we have
\[
	W_2 \le \IP\left( C_r n^{-(1+\gamma)} Z_n > \frac{\delta}{2}\right) \le \frac{C_r(1 + |\chi|_1 \max\{D_M, \tilde D_M\})}{\delta/2}\cdot n^{-(1+\gamma)}
\]
which is absolutely summable.

The proof of \reff{lemma_uniformconvergence_eq2} is immediate from the bounds \reff{lemma_stat_eq1}, \reff{lemma_stat_eq2} and \reff{lemma_bias_weak_eq1} applied to each summand of $\IE E_n(K_h(\cdot-u),g,\theta)$ and the fact that $K$ has bounded variation which gives
\[
	\textstyle\IE g(\tilde Y_0(u),\theta) - \frac{1}{n}\sum_{t=1}^{n}K_h\big(\frac{t}{n}-u\big)\cdot \IE g(\tilde Y_t(u),\theta) = O((nh)^{-1})
\]
as long as $u \in [\frac{h}{2},1-\frac{h}{2}]$.
\end{proof}

\begin{lemma}\label{lemma_minust}
	Let Assumption \ref{ass1}, \ref{ass2} and \ref{ass4} hold. Let $g\in \sL(M,\chi,C)$. Define $E_{n,-s}(\phi,g,\theta) := \frac{1}{n}\sum_{t=1, t \not= s}^{n}\phi\big(\frac{t}{n}\big)\cdot g(Y_{t,n}^c,\theta)$. Then for all $0 < \alpha < 1$, we have
	\begin{eqnarray}
		F_n &:=& \sup_{s = 1,...,n}\sup_{h \in H_n}\sup_{u \in \supp(w_{n,h})} \sup_{\theta \in \Theta}\nonumber\\
		&&\quad (nh)^{1- \alpha}\big| E_{n,-s}(K_h(\cdot -u),g,\theta) - E_{n}(K_h(\cdot-u),g,\theta) \big| \to 0 \quad \text{a.s.}\label{lemma_minust_eq1}
	\end{eqnarray}
\end{lemma}
\begin{proof}[Proof of Lemma \ref{lemma_minust}:]
	Fix $\delta > 0$. Since $F_n \le n^{-\delta \alpha} \cdot |K|_{\infty}C_{\infty}\cdot \sup_{s=1,...,n}\big(1 + |Y_{s,n}|_{\chi,1}^M\big)$, we obtain by Markov's inequality
	\begin{eqnarray*}
		&& \IP(F_n > \delta) \le \Big(\frac{|K|_{\infty}C_{\infty}(1+|\chi|_1^M D_{qM}^M)}{\delta}\Big)^q\cdot n^{1-\delta \alpha q}.
	\end{eqnarray*}
	If $q$ is chosen large enough, we obtain the assertion by Borel-Cantelli's lemma.
\end{proof}

The following corollary is immediate from Lemma \ref{lemma_uniformconvergence} and \ref{lemma_minust}.
\begin{corollary}[Uniform convergence of likelihoods]\label{lemma_likelihood_uniform} Let Assumption \ref{ass1}, \ref{ass2}, \ref{ass3} and \ref{ass4} hold. Then for all $k = 0,1,2,3$ and all $0 < \alpha \le \frac{1}{2}$ it holds component-wise that
\[
	\sup_{h\in H_n} \sup_{u \in \supp(w_{n,h})} \sup_{\theta \in \Theta}\ (nh)^{\frac{1}{2}-\alpha}\cdot \big| \IE_0 \nabla^k L_{n,h}(u,\theta)\big| \to 0 \quad \text{a.s.}
\]
and there exists a constant $C_S > 0$ independent of $u,h,\theta$ such that (component-wise):
\[
	|\IE \nabla^k L_{n,h}(u,\theta) - \IE \nabla^k \l(\tilde Y_t(u),\theta)| \le C_S \big\{ (nh)^{-1} + h^{\beta \wedge 1}\big\}.
\]
and for all $0 < \alpha' \le 1$, it holds that
\[
	\sup_{s = 1,...,n}\sup_{h \in H_n}\sup_{u \in \supp(w)} \sup_{\theta \in \Theta}\  (nh)^{1- \alpha'}\Big| \nabla_{\theta}^{k} L_{n,h,-s}(u,\theta) - \nabla_{\theta}^{k} L_{n,h}(u,\theta) \Big| \to 0.\quad \text{a.s.}
\]
\end{corollary}

The following Theorem is a consequence of Corollary \ref{lemma_likelihood_uniform}. The proof uses standard arguments from maximum likelihood theory and is postponed to the appendix.

\begin{theorem}[Uniform strong consistency of the maximum likelihood estimator]\label{mle_konsistenz} Let Assumptions \ref{ass1}, \ref{ass2}, \ref{ass3} and \ref{ass4} hold. Then for each $0 < \alpha \le \frac{1}{2}$, it holds almost surely in each component that
	 \begin{equation}
	 	\sup_{h \in H_n}\sup_{u \in \supp(w_{n,h})} \big| \hat \theta_h(u) - \theta_0(u)\big| \to 0 \quad \text{a.s.}\label{mle_konsistenz_gleichung}
	 \end{equation}
Furthermore for $n$ large enough, we have uniformly in $h\in H_n$, $u \in \supp(w_{n,h})$ for each component that
\begin{eqnarray}
	&& \big|\hat \theta_h(u) - \theta_0(u)\big|\nonumber\\
	&\le& \frac{V_{-1,max}}{2}\sup_{i = 1,...,p}\big\{|\IE_0 \nabla_i L_{n,h}(u,\theta_0(u))| + |\IE \nabla_i L_{n,h}(u,\theta_0(u))|\big\}\label{mle_konsistenz_gleichung2}
\end{eqnarray}
The results still hold if $\hat \theta_h$, $L_{n,h}$ are replaced by $\hat \theta_{h,-t}$, $L_{n,h,-t}$ accordingly.
\end{theorem}

\subsection{Proofs of the results of Chapter \ref{sec3}}


\begin{proof}[Proof of Lemma \ref{da_approx_0}]
We have for arbitrary $q > 2$:
\begin{eqnarray}
	&& \|d_{A}^{*}(\hat \theta_h,\theta_0) - d_{I}^{*}(\hat \theta_h,\theta_0)\|_q\nonumber\\
	&\le& 2p^2V_{-1,max} |w|_{\infty}\sup_{i=1,...,d}\sup_{u\in[h/2,1-h/2]}\|\nabla_i L_{n,h}(u,\theta_0(u))\|_{2q}\nonumber\\
	&&\quad \times \sup_{j=1,...,d}\sum_{t=1}^{n}\int_{(t-1)/n}^{t/n}\|\nabla_j L_{n,h}(t/n,\theta_0(t/n)) - \nabla_j L_{n,h}(u,\theta_0(u)\|_{2q} \dif u\nonumber\\
	&& + p^2 \sup_{i=1,...,d}\sup_{u\in[h/2,1-h/2]}\|\nabla_i L_{n,h}(u,\theta_0(u))\|_{2q}^2\nonumber\\
	&&\quad\times p^4 V_{-1,max}^2 \sum_{t=1}^{n}\int_{(t-1)/n}^{t/n}\Big\{|w|_{\infty} \sup_{i,j=1,...,p}|V(\theta_0(t/n))_{ij} - V(\theta_0(u))_{ij}|\nonumber\\
	&&\quad\quad\quad\quad\quad\quad\quad\quad\quad + V_{-1,max} |w_{n,h}(t/n) - w_{n,h}(u)|\Big\} \dif u.\label{proof_da_approx_0_eq1}
\end{eqnarray}
Since $\theta_0$ has bounded variation $B_K$, $w_{n,h}$ has bounded variation $B_w$ and $\|\nabla_i L_{n,h}(u,\theta)\|_{2q} \le C_{\infty,2q}$, the second summand of \reff{proof_da_approx_0_eq1} is of order $O(n^{-1})$. Furthermore, by Lemma \ref{lemma_stat},  Lemma \ref{pmomente} and Lemma \ref{lemma_uniformconvergence} we obtain for each component
\[
	\|\nabla_i L_{n,h}(u,\theta_0(u))\|_{2q} \le \rho_{1,\psi C_{\delta,\cdot},2q}(nh)^{-1/2} + C_S\{ (nh)^{-1} + h^{\beta \wedge 1}\}.
\]
By Lemma \ref{lemma_stetigkeit}, we have
\begin{eqnarray*}
	&&\|\nabla_j L_{n,h}(t/n,\theta_0(t/n)) - \nabla_j L_{n,h}(u,\theta_0(u))\|_{2q}\\
	&\le& C_{-,q}\big\{\frac{|t/n-u|}{h} + |\theta_0(t/n) - \theta_0(u)|_1\big\}.
\end{eqnarray*}
Since $\theta_0$ has bounded variation $B_{\theta_0}$, we obtain that the first summand in \reff{proof_da_approx_0_eq1} is $O((nh)^{-1}\big\{h^{\beta\wedge 1} + (nh)^{-1/2}\big\})$. So in view of Proposition \ref{misedecomposition}, we have shown that there exists $C, \gamma > 0$ such that 
\begin{equation}
	\frac{\|d_A^{*}(\hat \theta_h,\theta_0) - d_I^{*}(\hat \theta_h,\theta_0)\|_q}{d_{M}^{*}(\hat \theta_h,\theta_0)} \le C n^{-\gamma}.\label{proof_da_approx_0_eq2}
\end{equation}
Define $Z_n := 1 + \frac{1}{n}\sum_{t=1}^{n}|Y_{t,n}^c|_{\chi,1}^{M}$. Using the inequality \reff{proof_uniform_eq1} and the notation therein, we obtain for each component that $|\nabla L_{n,h}(u,\theta)| \le \frac{1}{\underline{h}}|K|_{\infty}C_{\infty}Z_n$ and $|\nabla L_{n,h}(u,\theta) - \nabla L_{n,h'}(u,\theta)| \le \frac{1}{\underline{h}^3}L_K C_{\infty}Z_n |h-h'|$. These results together with \reff{ass4_eq1} imply
\begin{eqnarray*}
	&&\big|d_{A}^{*}(\hat \theta_h,\theta_0) - d_A^{*}(\hat \theta_{h'},\theta_0)\big|\\
	&\le& p^2 V_{-1,max}C_{\infty}^2 \cdot \Big[\frac{2}{\underline{h}^4}L_K |K|_{\infty}|w|_{\infty} + \frac{C_w}{\underline{h}^2}|K|_{\infty^2}\Big]\cdot Z_n^2 \cdot |h-h'|.
\end{eqnarray*}
A similar argumentation is valid for $|d_{I}^{*}(\hat \theta_h,\theta_0) - d_I^{*}(\hat \theta_{h'},\theta_0)|$. Since $\underline{h} \ge c_0 n^{\delta-1}$ by assumption, we have shown that there exists $C(n)$ which grows at most polynomially in $n$ such that
\begin{eqnarray}
	&& \big|\big(d_{A}^{*}(\hat \theta_h,\theta_0) - d_I^{*}(\hat \theta_h,\theta_0)\big) - \big(d_A^{*}(\hat \theta_{h'},\theta_0) - d_{I}^{*}(\hat \theta_{h'},\theta_0)\big)\big|\nonumber\\
	&\le& C(n)\cdot Z_n^2 \cdot |h-h'|.\label{proof_da_approx_0_eq3}
\end{eqnarray}
As in the proof of Lemma \ref{da_approx_0} it can be shown that \reff{proof_da_approx_0_eq2} and \reff{proof_da_approx_0_eq3} together imply that $\sup_{h\in H_n}\frac{|d_A^{*}(\hat \theta_h,\theta_0) - d_{I}^{*}(\hat \theta_h,\theta_0)|}{d_M^{*}(\hat \theta_h,\theta_0)} \to 0$ a.s. 

With the definitions $A_1(u) := |\nabla \hat L_{n,h}(u,\theta_0(u)) - \IE \nabla \hat L_{n,h}(u,\theta_0(u))|_{V(\theta_0(u))}^2$ and $A_2(u) := \langle \nabla \hat L_{n,h}(u,\theta_0(u)) -   \IE\nabla \hat L_{n,h}(u,\theta_0(u)), V(\theta_0(u)) \IE \nabla \hat L_{n,h}(u,\theta_0(u))\rangle$, we decompose
\begin{eqnarray*}
	&& d_{I}^{*}(\hat \theta_h,\theta_0) - \IE d_{I}^{*}(\hat \theta_h,\theta_0)\\
	&=& \big[d_{I}^{*}(\hat \theta_h,\theta_0) - \int_{0}^{1}|\hat \nabla L_{n,h}(u,\theta_0(u))|_{V(\theta_0(u))}^2 w_{n,h}(u) \dif u\big]\\
	&& \quad\quad + \big[\int_{0}^{1}\big\{A_1(u) - \IE A_1(u)\big\} w_{n,h}(u)\dif u\big] + \big[\int_{0}^{1}A_2(u) w_{n,h}(u) \dif u\big]\\
	&=:& R_{n,h,1} + R_{n,h,2} + R_{n,h,3}.
\end{eqnarray*}
With Lemma \ref{hilfslemma1} we obtain $\|R_{n,h,1}\|_q = O(n^{-1})$, $\|R_{n,h,2}\|_q = O(n^{-1/2}(nh)^{-1/2})$ and $\|R_{n,h,3}\|_q = O(n^{-1/2}\hat B_h + n^{-1/2}(nh)^{-1/2})$ (see Proposition \ref{misedecomposition}(ii) for $\hat B_h$), i.e. $\frac{\|R_{n,h,i}\|_q}{d_{M}^{*}(\hat \theta_h,\theta_0)} = O(n^{-\tilde\gamma})$ with some $\tilde \gamma > 0$. Details are deferred to the appendix. It is straightforward to see that $|R_{n,h,i} - R_{n,h',i}|$ fulfills a similar condition as in \reff{proof_da_approx_0_eq2} for $i = 1,2,3$. The technique from the proof of Lemma \ref{da_approx_0} implies $\sup_{h\in H_n}\big|\frac{d_{I}^{*}(\hat \theta_h,\theta_0) - d_M^{*}(\hat \theta_h,\theta_0)}{d_M^{*}(\hat \theta_h,\theta_0)}\big| \to 0$ a.s.
\end{proof}

\begin{proof}[Proof of Corollary \ref{da_approx}] The convergence $\sup_{h\in H_n}\big|\frac{d_I(\hat \theta_h,\theta_0) - d_I^{*}(\hat \theta_h,\theta_0)}{d_I^{*}(\hat \theta_h,\theta_0)}\big|\to 0$ a.s. follows from the decomposition
\[
	\hat \theta_{h}(u) - \theta_0(u) = -(\nabla^2 L_{n,h}(u, \bar \theta_h(u)))^{-1} \cdot \nabla L_{n,h}(u, \theta_0(u))
\]
and the results from Corollary \ref{lemma_likelihood_uniform} and Theorem \ref{mle_konsistenz}. Together with Lemma \ref{da_approx_0}, the assertion of the Corollary follows. Details can be found in the appendix.
\end{proof}

\begin{proof}[Proof of Lemma \ref{da_strich_approx}]
Put
\[
	\overline{d}_{A}^{*}(\hat \theta_h,\theta_0) := \frac{1}{n}\sum_{t=1}^{n}\big| \nabla L_{n,h,-t}(t/n, \theta_0(t/n))\big|_{V(\theta_0(t/n))^{-1}}^2.
\]
We have to show that
\begin{equation}
	\sup_{h \in H_n}\Big| \frac{\overline{d}_{A}^{*}(\hat \theta_h, \theta_0) - d_{A}^{*}(\hat \theta_h, \theta_0)}{d_M^{*}(\hat \theta_h, \theta_0)}\Big| \to 0\quad\text{a.s.},\label{dazuzeigen}
\end{equation}
then it follows immediately from Lemma \ref{da_approx_0}:
\begin{equation}
	\sup_{h\in H_n}\Big|\frac{\overline{d}_{A}^{*}(\hat \theta_h, \theta_0) - d_{M}^{*}(\hat \theta_h, \theta_0)}{d_M^{*}(\hat \theta_h,\theta_0)}\Big| \to 0\quad \text{a.s.}\label{da_abstand_1}
\end{equation}
Using the same techniques as in the proof of Corollary \ref{da_approx} (and additionally the uniform convergence results of Corollary \ref{lemma_likelihood_uniform}), it can be shown that
\begin{equation}
	\sup_{h\in H_n}\Big|\frac{\overline{d}_{A}(\hat \theta_h, \theta_0) - \overline{d}_{A}^{*}(\hat \theta_h, \theta_0)}{\overline{d}_{A}^{*}(\hat \theta_h, \theta_0)}\Big| \to 0\quad \text{a.s.},\label{da_abstand_2}
\end{equation}
and we can conclude from \reff{da_abstand_1}, \reff{da_abstand_2} (by a similar expansion as in \reff{proof_cor_da_cv_approx_eq1}) that $\sup_{h\in H_n}\big|\frac{\overline{d}_{A}(\hat \theta_h, \theta_0) - d_{M}^{*}(\hat \theta_h, \theta_0)}{d_M^{*}(\hat \theta_h,\theta_0)}\big| \to 0$ a.s.
It remains to prove \reff{dazuzeigen}.
By applying the Cauchy Schwarz inequality, we obtain
\begin{eqnarray}
	&& |d_{A}^{*}(\hat \theta_h,\theta_0) - \overline{d}_{A}^{*}(\hat \theta_h,\theta_0)|\nonumber\\
	&\le& \frac{2 K(0)}{nh} \Big|\frac{1}{n}\sum_{t=1}^{n}\langle \nabla \l(Y_{t,n}^c,\theta_0(t/n)), V(\theta_0(t/n))^{-1} \nabla L_{n,h}(t/n, \theta_0(t/n))\rangle\Big|\nonumber\\
	&& \quad + \frac{K(0)^2}{(nh)^2} \cdot \frac{1}{n}\sum_{t=1}^{n}\big|\nabla \l(Y_{t,n}^c,\theta_0(t/n))\big|_{V(\theta_0(t/n))^{-1}}^2\nonumber\\
	&\le& \frac{2K(0)}{nh}W_{n}^{1/2}\cdot d_{A}^{*}(\hat \theta_h, \theta_0)^{1/2} +  \frac{K(0)^2}{(nh)^2}\cdot W_{n}\label{daapprox10},
\end{eqnarray}
where $W_{n} := \frac{1}{n}\sum_{t=1}^{n}\big|\nabla \l(Y_{t,n}^c,\theta_0(t/n))\big|^2_{V(\theta_0(t/n))^{-1}}$ is bounded a.s. (details are deferred to the appendix). The assertion now follows from \reff{daapprox10}, Lemma \ref{da_approx_0} and the expansion $d_{A}^{*} = d_M^{*}\cdot\big(1 + \frac{d_{A}^{*}-d_{M}^{*}}{d_M^{*}}\big)$.
\end{proof}

For some $A \in \IR^{d\times d \times d}$ and vectors $y,z\in\IR^d$, define $x := A[y,z] \in \IR^d$ as $x_i := \sum_{j,k=1}^{d}A_{ijk}y_j z_k$.
\begin{proof}[Proof of Lemma \ref{da_cv_approx}] Define $\l_{t,n}(\theta) := \l(Y_{t,n}^c,\theta)$. By Taylor's expansion, it holds that
\begin{align}
	& \l_{t,n}(\hat \theta_{h,-t}(u)) - \l_{t,n}(\theta_0(u))\nonumber\\
	&= \nabla \l_{t,n}(\theta_0(u))' \big\{\hat \theta_{h,-t}(u) - \theta_0(u)\big\} + |\hat \theta_{h,-t}(u) - \theta_0(u)|^2_{\nabla^2 \l_{t,n}(\bar \theta_{h,-t}(u))}\label{proof_da_cv_approx_eq1}
\end{align}
with some intermediate value $\bar \theta_{h,-t}(u))$ which satisfies $|\bar \theta_{h,-t}(u)) - \theta_0(u)|_1 \le |\hat \theta_{h,-t}(u) - \theta_0(u)|_1$. By Theorem \ref{mle_konsistenz}, $\hat \theta_{h,-t}(u)$ converges to $\theta_0(u)$ uniformly in $u,h,t$ and thus lies in the interior of $\Theta$ for $n$ large enough. Using a third-order Taylor expansion, we obtain 
\begin{eqnarray}
	&& \hat \theta_{h,-t}(u) - \theta_0(u)\nonumber\\
	&=& -\big[\nabla^2 L_{n,h,-t}(u,\theta_0(u))\big]^{-1}\big\{ \nabla L_{n,h,-t}(u,\theta_0(u))\nonumber\\
	&&\quad\quad\quad\quad+ \nabla^3 L_{n,h,-t}(\tilde \theta_{h,-t}(u))[\hat \theta_{h,-t}(u) - \theta_0(u),\hat \theta_{h,-t}(u)-\theta_0(u)]\big\} \label{proof_da_cv_approx_eq2}
\end{eqnarray}
with some intermediate value $\tilde \theta_{h,-t}(u))$ which satisfies $|\tilde \theta_{h,-t}(u)) - \theta_0(u)|_1 \le |\hat \theta_{h,-t}(u) - \theta_0(u)|_1$. Put $V_{t,n} := \nabla^2 L_{n,h,-t}(t/n,\theta_0(t/n))$. For $n$ large enough and $h \in H_n$, we have by Corollary \ref{lemma_likelihood_uniform} that $|V_{t,n} - V(\theta_0(t/n))|_2, |\IE V_{t,n} - V(\theta_0(t/n))|_2 \le \frac{\lambda_0}{4}$. Then it follows that the minimal eigenvalue of $V_{t,n}$ and $\IE V_{t,n}$ are bounded from below by $\frac{\lambda_0}{2}$. So for $n$ large enough, $h\in H_n$, we have
\begin{equation}
	\sup_{i,j=1,...,p}|(V_{t,n}^{-1})_{ij}|,  \sup_{i,j=1,...,d}|((\IE V_{t,n})^{-1})_{ij}| \le \frac{V_{-1,max}}{2}.
\end{equation}
With \reff{proof_da_cv_approx_eq1} and \reff{proof_da_cv_approx_eq2}, we obtain the decomposition
\begin{align*}
	& \textstyle CV(h) - \frac{1}{n}\sum_{t=1}^{n}\l(\theta_0(t/n))w_{n,h}(t/n) - \overline{d_A}(\hat \theta_h,\theta_0)\nonumber\\
	&\textstyle= -\frac{1}{n}\sum_{t=1}^{n}\nabla \l_{t,n}(\theta_0(t/n))'\cdot V_{t,n}^{-1}\cdot \nabla L_{n,h,-t}(t/n,\theta_0(t/n)) w_{n,h}(t/n)\\
	&\textstyle\quad - \frac{1}{n}\sum_{t=1}^{n}\nabla \l_{t,n}(\theta_0(t/n))'\cdot V_{t,n}^{-1}\cdot \nabla^3 L_{n,h,-t}(t/n,\tilde \theta_{h,-t}(t/n))\\
	&\quad\quad\quad\quad\quad\quad \big[\hat \theta_{h,-t}(t/n) - \theta_0(t/n)),\hat \theta_{h,-t}(t/n) - \theta_0(t/n))\big]w_{n,h}(t/n)\\
	&\textstyle\quad + \frac{1}{n}\sum_{t=1}^{n}\big|\hat \theta_{h,-t}(t/n) - \theta_0(t/n)\big|^2_{\nabla^2 \l_{t,n}(\bar \theta_{h,-t}(t/n)) - V(\theta_0(t/n))} w_{n,h}(t/n)\\
	&\textstyle=: R_{n,h,1} + R_{n,h,2} + R_{n,h,3}. 
\end{align*}
The remainders $R_{n,h,i}$ ($i = 1,2,3$) now have to be discussed separately. To get rid of $V_{t,n}^{-1}$ and the intermediate values $\tilde \theta_{h,-t}$ and $\bar \theta_{h,-t}$ we replace them by $(\IE V_{t,n})^{-1}$ and $\theta_0$, respectively. To replace $V_{t,n}^{-1}$, we use the decompositions
\begin{eqnarray}
	V_{t,n}^{-1} &=& (\IE V_{t,n})^{-1} + (V_{t,n})^{-1}\big\{ \IE V_{t,n} - V_{t,n}\big\} (\IE V_{t,n})^{-1}\label{proof_da_cv_approx_eq25}\\
	&=& (\IE V_{t,n})^{-1} + (\IE V_{t,n})^{-1}\big\{ \IE V_{t,n} - V_{t,n}\big\} (\IE V_{t,n})^{-1}\nonumber\\
	&&\quad+ V_{t,n}^{-1}\big\{ \IE V_{t,n} - V_{t,n}\big\} (\IE V_{t,n})^{-1} \big\{ \IE V_{t,n} - V_{t,n}\big\} (\IE V_{t,n})^{-1}. \label{proof_da_cv_approx_eq3}
\end{eqnarray}
It can be shown these replacements are of order $(nh)^{-\frac{3}{2}(1-\alpha)} + (B_h^{dis})^{3/2}$ uniformly in $h\in H_n$ with arbitrary small $\alpha > 0$ by using the uniform results of Corollary \ref{lemma_likelihood_uniform} and Theorem \ref{mle_konsistenz}, \reff{mle_konsistenz_gleichung2} (see Proposition \ref{misedecomposition} for $B_h^{dis}$). By the decomposition of $d_M^{*}(\hat \theta_h,\theta_0)$ in Proposition \ref{misedecomposition}(i),(ii)  the replacements are of smaller order than $d_M^{*}(\hat \theta_h,\theta_0)$. The remaining terms are listed in Lemma \ref{lemma_approxmart}, where also their convergence is proven. Details are in the appendix.
\end{proof}

In the proof of the following Lemma \ref{lemma_approxmart} we use a similar technique as in Lemma \ref{lemma_uniformconvergence} (or Lemma \ref{da_approx_0}). The main part therefore is to calculate the norm $\|\cdot \|_q$ of the quantities which is done via the results of Lemma \ref{hilfslemma1}. Details can be found in the appendix. 

\begin{lemma}\label{lemma_approxmart}
	Let Assumption \ref{ass1}, \ref{ass2}, \ref{ass3} and \ref{ass4} hold. Put $V_{t,n} := \nabla^2 L_{n,h,-t}(t/n,\theta_0(t/n))$. Then it holds almost surely that
	\begin{align}
		\sup_{h\in H_n}\frac{1}{d_{M}^{*}(\hat \theta_h,\theta_0)}\Big|\frac{1}{n}\sum_{t=1}^{n}\nabla \l(Y_{t,n}^c,\theta_0(t/n))' (\IE V_{t,n})^{-1}\quad\quad\quad\quad & \nonumber\\
		\times\nabla L_{n,h,-t}(t/n, \theta_0(t/n))w_{n,h}(t/n)\Big| &\to 0,\label{lemma_approxmart_eq1}\\
		\sup_{h\in H_n}\frac{1}{d_{M}^{*}(\hat \theta_h,\theta_0)}\Big|\frac{1}{n}\sum_{t=1}^{n}w_{n,h}(t/n)\quad\quad\quad\quad\quad\quad\quad\quad\quad\quad\quad\quad\quad\quad\quad\quad &\nonumber\\
		\times |(\IE V_{t,n})^{-1}\nabla L_{n,h,-t}(t/n,\theta_0(t/n))|^2_{\nabla^2 \l(Y_{t,n}^c,\theta_0(t/n)) - V(\theta_0(t/n))}\Big| &\to 0,\label{lemma_approxmart_eq2}\\
		\sup_{h\in H_n}\frac{1}{d_M^{*}(\hat \theta_h,\theta_0)}\Big|\frac{1}{n}\sum_{t=1}^{n}\nabla \l(Y_{t,n}^c,\theta_0(t/n))' (\IE V_{t,n})^{-1}\quad\quad\quad\quad\quad\quad\quad\quad&\nonumber \\
		\times (V_{t,n} - \IE V_{t,n}) (\IE V_{t,n})^{-1}\nabla L_{n,h,-t}(t/n, \theta_0(t/n))w_{n,h}(t/n)\Big| &\to 0,\label{lemma_approxmart_eq3}\\
		\sup_{h\in H_n}\frac{1}{d_M^{*}(\hat \theta_h,\theta_0)}\big|\frac{1}{n}\sum_{t=1}^{n}w_{n,h}(t/n)\cdot \nabla \l_{t,n}(\theta_0(t/n))'(\IE V_{t,n})^{-1}\quad\quad\nonumber\\
		\times  \big(\IE \nabla^3 L_{n,h,-t}(\theta_0(t/n))\big)\big[(\IE V_{t,n})^{-1} \nabla L_{n,h,-t}(t/n,\theta_0(t/n))\}^2,\quad\quad\quad\quad&\nonumber\\
	\quad\quad\quad\quad  (\IE V_{t,n})^{-1} \nabla L_{n,h,-t}(t/n,\theta_0(t/n))\}^2\big]\big| &\to 0,\label{lemma_approxmart_eq4}\\
	\intertext{}
		\sup_{h\in H_n}\frac{\frac{1}{n}\sum_{t=1}^{n}\IE_0[ |Y_{t,n}^c|_{\chi,1}^{M}] w_{n,h}(t/n) \sup_{i=1,...,p}|\IE \nabla_i \hat L_{n,h}(t/n, \theta_0(t/n))|}{d_M^{*}(\hat \theta_h,\theta_0)} &\to 0.\label{lemma_approxmart_eq5}
	\end{align}
\end{lemma}

\textbf{Acknowledgements.}

We gratefully acknowledge support by Deutsche Forschungsgemeinschaft
through the Research Training Group RTG 1653.

\newpage

\appendix


\section{More detailed proofs of section 3}

\begin{proof}[Proof of Corollary \ref{da_approx}, more detailed] 
	We have
	\[
		\hat \theta_{h}(u) - \theta_0(u) = -(\nabla^2 L_{n,h}(u, \bar \theta_h(u)))^{-1} \cdot \nabla L_{n,h}(u, \theta_0(u)) = \big\{ I_d + R_h(u)\big\} \cdot D_h(u),
	\]
	where $I_d \in \IR^{d\times d}$ is the identity matrix and
	\begin{eqnarray*}
		D_h(u) &:=& -V(\theta_0(u))^{-1}\cdot \nabla L_{n,h}(u, \theta_0(u)),\\
		R_h(u) &:=& (\nabla^2 L_{n,h}(u, \bar \theta_h(u)))^{-1}V(\theta_0(u)) - I_{d}
	\end{eqnarray*}
	and $\bar \theta_h(u)$ is some intermediate value with $|\bar \theta_h(u) - \theta_0(u)|_2 \le |\hat \theta_h(u) - \theta_0(u)|_2$. Using the bound $|\langle x,Ax \rangle| \le |x|_2^2 |A|_{spec}$, we have
	\begin{eqnarray}
		&& \big| | (I_d + R_h(u))D_h(u)|^2_{I(\theta_0(u))}  - | D_h(u)|^2_{V(\theta_0(u))}\big|\nonumber\\
		&\le& 2 \cdot \big|\langle D_h(u), I(\theta_0(u)) R_h(u) D_h(u)\rangle\big| + |R_h(u) D_h(u)|^2_{I(\theta_0(u))}\nonumber\\
		&\le& \big\{ 2 | I(\theta_0(u)) R_h(u)|_2 + | R_h(u)' V(\theta_0(u)) R_h(u)|_2\big\} \cdot | D_h(u)|^2. \label{abstand2}
	\end{eqnarray}
	By Corollary \ref{lemma_likelihood_uniform} and Theorem \ref{mle_konsistenz}, we have
	\[
		\sup_{h \in H_n}\sup_{u \in \supp(w_{n,h})}|\nabla^2 L_{n,h}(u, \bar \theta(u)) - V(\theta_0(u))| \to 0 \quad \text{a.s.},
	\]
	thus
	\begin{equation}
		\sup_{h \in H_n} \sup_{u \in \supp(w_{n,h})} |R_h(u)| \to 0.\label{abstand3}
	\end{equation}
	According to Assumption \ref{ass3}, let $\lambda_0 > 0$ be the value which bounds all eigenvalues from $V(\theta)$ from below. Using the representations $d_{I}(\hat \theta_h, \theta_0) = \int_{0}^{1}| (I_d + R_h(u)) \cdot D_h(u)|^2_{V(\theta_0(u))} w_{n,h}(u) \dif u$, $d_{I}^{*}(\hat \theta_h, \theta_0) = \int_{0}^{1}| D_h(u)|^2_{V(\theta_0(u))} w_{n,h}(u) \dif u$, we conclude with \reff{abstand2}, \reff{abstand3}:
	\begin{eqnarray*}
		&& \sup_{h \in H_n}\Big| \frac{d_{I}(\hat \theta_h, \theta_0) - d_{I}^{*}(\hat \theta_h, \theta_0)}{d_{I}^{*}(\hat \theta_h, \theta_0)}\Big|\\
		&\le& \frac{1}{\lambda_0}\sup_{h\in H_n}\Big\{ \big\{\int_{0}^{1}\big| |(I_d + R_h(u))D_h(u)|^2_{V(\theta_0(u))}\\
		&&\quad\quad\quad\quad\quad\quad\quad\quad\quad - |D_h(u)|^2_{V(\theta_0(u))}\big| \cdot w_{n,h}(u) \dif u\big\}\\
		&&\quad\quad\quad\quad\quad \times \big\{\int_{0}^{1} | D_h(u)|^2 \cdot w_{n,h}(u) \dif u\big\}^{-1}\Big\} \to 0 \quad (n\to\infty).
	\end{eqnarray*}
	Using the shortcuts $d_I = d_I(\hat \theta_h, \theta_0)$ and so on, we have
	\begin{equation}
		 \frac{d_{I} - d_M^{*}}{d_{M}^{*}}= \frac{d_{I} - d_{I}^{*}}{d_{I}^{*}}\cdot \left(\frac{d_{I}^{*} - d_{M}^{*}}{d_{M}^{*}} +1\right)+ \frac{d_I^{*}- d_{M}^{*}}{d_{M}^{*}},\label{proof_cor_da_cv_approx_eq1}
	\end{equation}
	hence, the assertion follows from Lemma \ref{da_approx_0}. The proof for $d_A$ is the same by using sums instead of integrals.
\end{proof}

\section{Additional proofs of section \ref{sec6}}

\begin{proof}[Proof of Lemma \ref{lemma_stat}] Since $g \in \sL(M,\chi,C)$, we have by Hoelder's inequality:
\begin{align}
	& \big\|\sup_{\theta \in \Theta}|g(Y_{t,n}^c,\theta) - g(\tilde Y_t(t/n)^c,\theta)|\big\|_q\nonumber\\
	&\le C_1\textstyle\sum_{j=1}^{t}\chi_j \big\|X_{t-j+1,n} - \tilde X_{t-j+1}\big(\frac{t}{n}\big)\big\|_{qM}\nonumber\\
	&\quad\quad\quad \times \big(1 + \big(\textstyle\sum_{j=1}^{t}\chi_j\|X_{t-j+1,n}\|_{qM}\big)^{M-1}\nonumber\\
	&\quad\quad\quad\quad\quad+ \big(\textstyle\sum_{j=1}^{t}\chi_j\big\|\tilde X_{t-j+1}\big(\frac{t}{n}\big)\big\|_{qM}\big)^{M-1}\big)\nonumber\\
	&\le C_1\big(1 + 2(D_{qM} |\chi|_1)^{M-1}\big)\cdot \textstyle\sum_{j=1}^{t}\chi_j \big\|X_{t-j+1,n} - \tilde X_{t-j+1}\big(\frac{t}{n}\big)\big\|_{qM}.\label{lemma_stat_eq1}
\end{align}
Furthermore, we have
\begin{eqnarray*}
	&& \textstyle\sum_{t=1}^{n}\sum_{j=1}^{t}\chi_j \|X_{t-j+1,n} - \tilde X_{t-j+1}(\frac{t}{n})\|_{qM}\\
	&\le& \textstyle\sum_{j=1}^{\infty}\chi_j\sum_{t=1}^{n}\|X_{t,n} - \tilde X_{t}(\frac{t}{n})\|_{qM}\\
	&&\textstyle\quad\quad\quad\quad + \sum_{j=1}^{\infty}\chi_j \sum_{t=1}^{n}\|X_{t-j+1}\big(\frac{t-j+1}{n}\big) - \tilde X_{t-j+1}(\frac{t}{n})\|_{qM}\\
	&\le& \textstyle C_B |\chi|_1 + p L_K C_A \sum_{j=1}^{\infty}j\chi_j.
\end{eqnarray*}
Define $ \tilde C_{S,q,1} := C_1(1 + 2(D_{qM} |\chi|_1)^{M-1})\cdot (C_B|\chi|_1 + pL_KC_A \sum_{j=1}^{\infty}j\chi_j)$, then
\begin{equation}
	\textstyle\big\|\sup_{\theta\in \Theta}\big|\frac{1}{n}\sum_{t=1}^{n}K_h\big(\frac{t}{n}-u\big)\cdot \big\{g(Y_{t,n}^c,\theta) - g(\tilde Y_t(t/n)^c,\theta)\big\}\big\|_2 \le \frac{|K|_{\infty} \tilde C_{S,q,1}}{nh}.\label{proof_lemma_stat_eq3}
\end{equation}
By Hoelder's inequality, we have for arbitrary $u \in [0,1]$:
\begin{equation}
	\big\|\sup_{\theta \in \Theta}|g(\tilde Y_t(u)^c,\theta) - g(\tilde Y_t(u),\theta)|\big\|_q \le C_1 \big(1 + 2(D_{qM} |\chi|_1)^{M-1}\big) D_{qM} \cdot \textstyle\sum_{j=t+1}^{\infty}\chi_j .\label{lemma_stat_eq2}
\end{equation}
Defining $\tilde C_{S,q,2} := C_1 (1 + 2( D_{qM} |\chi|_1)^{M-1}) D_{qM} \cdot \sum_{t=1}^{\infty}\sum_{j=t+1}^{\infty}\chi_j$, we obtain
\begin{equation}
	\textstyle\big\|\sup_{\theta\in \Theta}\big|\frac{1}{n}\sum_{t=1}^{n}K_h\big(\frac{t}{n}-u\big)\cdot \big\{g(\tilde Y_{t}(t/n)^c,\theta) - g(\tilde Y_t(t/n),\theta)\big\}\big\|_q \le \frac{|K|_{\infty} \tilde C_{S,q,2}}{nh}.\label{proof_lemma_stat_eq4}
\end{equation}
\reff{proof_lemma_stat_eq3} and \reff{proof_lemma_stat_eq4} give the result.
\end{proof}

\begin{proof}[Proof of Lemma \ref{lemma_expectation_expansion}:]
	Put $\tilde M := M-2$. Define $f_j := \partial_j g$ and $D_i := \max\{\tilde\chi_i, \psi_1(i)\psi_2(j)\}$. We show that $f_j:(\IR^{\infty}, |\cdot|_{D,1}) \to (\IR,|\cdot|)$ is Frechet differentiable with derivative $f_j'(y)h := \sum_{i=1}^{\infty}\partial_i \partial_j g(y)\cdot h_i$. Choose $h \in \IR^{\infty}$ with $|h|_{D,1} < \varepsilon$. Let $e_j = (e_{j,i})_{i=1,2,...} \in \IR^{\infty}$ be a sequence of zeros where only at the $j$-th position is a 1. By the mean value theorem in $\IR$, there exists $s_i \in [0,1]$ such that
	\begin{eqnarray*}
		&& \textstyle |f_j(y+h) - f_j(y) - f_j'(y)h|\\
		&\le& \textstyle\sum_{i=1}^{\infty}\big|f_j\big(y + \sum_{k=1}^{i}h_k e_k\big) - f_j\big(y + \sum_{k=1}^{i-1}h_k e_k\big) - \partial_i f_j(y) h_i\big|\\
		&=& \textstyle\sum_{i=1}^{\infty}\big|\partial_i f_j\big(y + \sum_{k=1}^{i-1}h_k e_k + s_i h_i\big) - \partial_i f_j(y)\big| \cdot |h_i|\\
		&\le& \textstyle\sum_{i=1}^{\infty}\psi_1(i)\psi_2(j) |h|_{\tilde\chi,1}\cdot \big(1 + 2^{\tilde M}|y|_{\tilde\chi,1}^{\tilde M-1} + |h|_{\tilde\chi,1}^{\tilde M-1}\big) \cdot |h_i|\\
		&\le& \textstyle\varepsilon\cdot \big(1 + 2^{\tilde M}|y|_{\tilde\chi,1}^{\tilde M-1} + |h|_{\tilde\chi,1}^{\tilde M-1}\big).
	\end{eqnarray*}
	This shows Frechet differentiability of $f_j$. By applying the chain rule, $s \mapsto f_j(y + s\cdot(y'-y))$ is differentiable with derivative $\sum_{i=1}^{\infty}\partial_i f_j(y + s\cdot(y'-y))\cdot (y_i'-y_i)$. By the fundamental theorem of analysis, 
	\begin{eqnarray*}
		&& |f_j(y') - f_j(y)|\\
		&\le& \int_0^{1} \sum_{i=1}^{\infty}|\partial_i f_j(y + s\cdot(y'-y)) - \partial_i f_j(0)|\cdot |y_i'-y_i| \dif s + |y'-y|_{\psi_1(\cdot) \psi_2(j),1}\\
		&\le& \Big\{ \big(|y|_{\tilde\chi,1}+|y'|_{\tilde\chi,1}\big)\cdot\big(1 + 2^{\tilde M}|y|_{\tilde\chi,1}^{\tilde M-1} + 2^{\tilde M}|y'|_{\tilde\chi,1}^{\tilde M-1}\big) + 1\Big\} |y'-y|_{\psi_1(\cdot) \psi_2(j),1}\\
		&\le& \tilde \psi_1 \psi_2(j) |y'-y|_{\psi_1,1}\big(1 + |y|_{\tilde\chi,1}^{\tilde M} + |y'|_{\tilde\chi,1}^{\tilde M}\big)
	\end{eqnarray*}
	with some constant $\tilde \psi_1$ dependent on $\tilde M$. This shows that $f_j = \partial_j g \in \sL(\tilde M+1, \tilde\chi+\psi_1, \tilde \psi_1 \psi_2(j))$. In the same manner, we obtain Frechet differentiability of $g$ itself, and $g \in \sL(M,\tilde\chi + \psi_1 + \tilde \psi_1 \psi_2, \tilde \psi_1 \tilde \psi_2)$ with some constant $\tilde \psi_2$ depending on $\tilde M$. Since $u\mapsto \theta_0(u)$ and $\theta \mapsto \tilde Y_t(\theta)$ are twice continuously differentiable, we conclude that the composition $u \mapsto \tilde Y_t(u) = \tilde Y_t(\theta_0(u))$ is twice continuously differentiable and thus
	\begin{eqnarray*}
		\textstyle\partial_u g(\tilde Y_t(u)) &=& \textstyle\sum_{j=1}^{\infty}\partial_j g(\tilde Y_t(u))\cdot \partial_u \tilde X_{t-j+1}(u),\\
		\textstyle\partial_u^2 g(\tilde Y_t(u)) &=& \textstyle\sum_{i,j=1}^{\infty} \partial_i \partial_j g(\tilde Y_t(u)) \cdot  \partial_u \tilde X_{t-i+1}(u) \partial_u \tilde X_{t-j+1}(u)\\
		&&\quad\quad\quad\textstyle+ \sum_{j=1}^{\infty}\partial_j g(\tilde Y_t(u))\cdot \partial_u^2 \tilde X_{t-j+1}(u),
	\end{eqnarray*}
	by the chain rule. We obtain with Hoelder's inequality
	\begin{eqnarray}
		&& \|\partial_i \partial_j g(\tilde Y_0(u))\|_{M/\tilde M}\nonumber\\
		&\le& \|\partial_i \partial_j g(\tilde Y_0(u)) - \partial_i \partial_j g(0)\|_{M/\tilde M} + |\partial_i \partial_j g(0)|\nonumber\\
		&\le& \psi_1(i)\psi_2(j) \big[\big\| |\tilde Y_0(u)|_{\tilde\chi,1}\cdot\big(1 + |\tilde Y_0(u)|_{\tilde\chi,1}^{\tilde M-1}\big)\big\|_{M/\tilde M} + 1\big]\nonumber\\
		&\le& \psi_1(i)\psi_2(j) \big[ \| |\tilde Y_0(u)|_{\tilde\chi,1}\|_M \cdot \big(1 + \| |\tilde Y_0(u)|_{\tilde\chi,1}\|_M^{\tilde M-1}\big) + 1\big]\nonumber\\
		&\le& \tilde D_2 \psi_1(i)\psi_2(j),\label{expectation_expansion_eq2}
	\end{eqnarray}
	where $\tilde D_2 := |\tilde\chi|_1 D_M\cdot (1 + (|\tilde\chi|_1 D_M)^{\tilde M-1}) + 1$. With similar arguments, we can show that $\|\partial_j g(\tilde Y_0(u))\|_{M/(\tilde M+1)} \le \tilde D_1 \psi_2(j)$ and $\|g(\tilde Y_0(u))\|_1 \le \tilde D_0$ with some constants $\tilde D_0, \tilde D_1 > 0$. Hoelder's inequality yields 
	\begin{align}
		\textstyle\|\partial_u g(\tilde Y_0(u))\|_1 &\le \textstyle\tilde D_1\sum_{j=1}^{\infty}\psi_2(j) \|\partial_u \tilde X_0(u)\|_M < \infty,\nonumber\\
		\textstyle\|\partial_u^2 g(\tilde Y_0(u))\|_1 &\le \textstyle\tilde D_2 \sum_{i,j=1}^{\infty}\psi_1(i) \psi_2(j) \|\partial_u \tilde X_0(u)\|_M^2\nonumber\\
		&\quad\quad\quad\textstyle+ \tilde D_1\sum_{j=1}^{\infty}\psi_2(j) \|\partial_u^2 \tilde X_0(u)\|_M < \infty,\label{expectation_expansion_eq3}
	\end{align}
	so we have proven the existence of all terms in the expansion \reff{expectation_expansion_eq1}. It remains to analyze the residual term $R(u,\xi)$. Since $u \mapsto \partial_u^2 g(\tilde Y_0(u))$ is (uniformly) continuous a.s., we have
	\begin{eqnarray}
		&& \lim_{\xi \to 0}\sup_{u\in[0,1]} \sup_{s \in [0,1]}|\partial_u^2 g(\tilde Y_0(u+\xi s)) - \partial_u^2 g(\tilde Y_0(u))|\nonumber\\
		&\le& \lim_{\xi \to 0}\sup_{|u-v| \le \xi}|\partial_u^2 g(\tilde Y_0(v)) - \partial_u^2 g(\tilde Y_0(u))|= 0.\label{expectation_expansion_eq4}
	\end{eqnarray}
	Furthermore, it holds that $\sup_{u\in[0,1]}|\tilde Y_0(u) - \tilde Y_0(0)| \le \sup_{u\in[0,1]}|\partial_u \tilde Y_0(u)| \le \sup_{u\in[0,1]}|\theta_0'(u)|_2 \cdot \sup_{\theta \in \Theta}|\nabla \tilde Y_t(\theta)|_2$ and thus $\|\sup_{u\in[0,1]}|\tilde Y_0(u)|\|_M$ is finite by Assumption \ref{ass5}. Assumption \ref{ass5} also directly implies that
	\[
		\|\sup_{u\in[0,1]}|\partial_u \tilde Y_0(u)|\|_M, \quad\|\sup_{u\in[0,1]}|\partial_u^2 \tilde Y_0(u)|\|_M < \infty.
	\]
	Using similar techniques as in \reff{expectation_expansion_eq2} and \reff{expectation_expansion_eq3}, $\|\sup_{u\in[0,1]}|\partial_u^2 g(\tilde Y_0(u))| \|_1 < \infty$ follows. The dominated convergence theorem and \reff{expectation_expansion_eq4} yield
	\[
		\lim_{\xi \to 0}\sup_{u\in[0,1]}R(u,\xi) = 0.
	\]
\end{proof}

\begin{proof}[Proof of Proposition \ref{misedecomposition}]
	(i) Let $j \in \{1,...,d\}$. By Lemma \ref{lemma_stat}, there exists some constant $C_S > 0$ such that
	\begin{eqnarray}
		&& \big|\IE \nabla L_{n,h}(u,\theta) - \IE \nabla \hat L_{n,h}(u,\theta)\big|\nonumber\\
		&\le& \Big\|\frac{1}{n}\sum_{t=1}^{n}K_h\Big(\frac{t}{n}-u\Big)\big\{\nabla \l(Y_{t,n}^c,\theta) - \nabla \l(\tilde Y_t(t/n), \theta)\big\}\Big\|_1\nonumber\\
		&\le& C_S(nh)^{-1}.\label{proof_misedecomposition_eq2}
	\end{eqnarray}
	By $\nabla \l \in \sL_{\infty}(M,\chi,C)$ (component-wise) and \reff{proof_uniform_eq1}, it holds that
	\[
		\sup_{\theta,\theta'\in\Theta}\|\nabla \l(\tilde Y_0(\theta'),\theta)\|_1 \le D := C_{\infty}(1 + (|\chi|_1 D_{qM})^M).
	\]
	Furthermore, $v \mapsto \IE \nabla \l(\tilde Y_0(v),\theta)$ has bounded variation $B_{\nabla \l}$ (uniformly in $\theta$) since $\theta_0$ has bounded variation. The same holds for the kernel $K$. We conclude that
	\begin{eqnarray}
		&& \Big|\frac{1}{n}\sum_{t=1}^{n}K_h\Big(\frac{t}{n}-u\Big)\IE \nabla \l(\tilde Y_t(t/n), \theta_0(u))\nonumber\\
		&&\quad\quad\quad\quad - \int_{0}^{1}K_h(v-u) \IE \nabla \l(\tilde Y_t(v),\theta_0(u)) \dif v\Big| \le \frac{D L_K + |K|_{\infty} B_{\nabla \l}}{nh}.\label{proof_misedecomposition_eq11}
	\end{eqnarray}
	This implies, uniformly in $u,h$, 
	\begin{equation}
		|\IE \nabla L_{n,h}(u,\theta_0(u)) - b_h(u)| = O((nh)^{-1}).\label{proof_misedecomposition_eq14}
	\end{equation}
	By Lemma \ref{lemma_stat} and Lemma \ref{lemma_bias_weak}, there exist constants $C_S', C_{bias} > 0$ such that
	\begin{equation}
		 \|\nabla L_{n,h}(u,\theta_0(u)) - \nabla \tilde L_{n,h}(u,\theta_0(u))\|_2 \le C_S' (nh)^{-1} + C_{bias}h^{\beta \wedge 1}.\label{proof_misedecomposition_eq3}
	\end{equation}
	By the (component-wise) martingale difference property of $\nabla \l(\tilde Y_t(u),\theta_0(u))$ and since $K$ has bounded variation, we have
	\begin{eqnarray}
		&&\IE\big| \nabla \tilde L_{n,h}(u,\theta_0(u))\big|_{V(\theta_0(u))^{-1}}^2\nonumber\\
		&=& \Big(\frac{1}{(nh)^2}\sum_{t=1}^{n}K\Big(\frac{t/n-u}{h}\Big)^2\Big)\cdot \IE \big| \nabla \l(\tilde Y_0(u),\theta_0(u))\big|_{V(\theta_0(u))^{-1}}^2\nonumber\\
		&=& \frac{\mu_K}{nh}\cdot \IE \big| \nabla \l(\tilde Y_0(u),\theta_0(u))\big|_{V(\theta_0(u))^{-1}}^2 + O((nh)^{-2}).\label{proof_misedecomposition_eq12}
	\end{eqnarray}
	Let us use the abbreviations $\nabla L := \nabla L_{n,h}(u,\theta_0(u)) - \IE \nabla L_{n,h}(u,\theta_0(u))$ and $\nabla \tilde L :=  \nabla \tilde L_{n,h}(u,\theta_0(u)) := \frac{1}{n}\sum_{t=1}^{n}K_h(t/n-u)\ell(\tilde Y_t(u),\theta_0(u))$. With \reff{proof_misedecomposition_eq3} we conclude
	\begin{eqnarray}
		&& \big| \IE \big| \nabla L\big|_A^2 - \IE\big| \nabla \tilde L\big|_A^2 \big|\nonumber\\
		&\le&  \big| \IE \langle \nabla L - \nabla \tilde L, A (\nabla L - \nabla \tilde L)\rangle\big| + \big|\IE \langle \nabla L - \nabla \tilde L, A \nabla \tilde L\rangle\big|\nonumber\\
		&&\quad\quad\quad + \big|\IE \langle \nabla \tilde L, A(\nabla L - \nabla \tilde L)\rangle\big|\nonumber\\
		&\le& |A|_{spec}\big\{\IE |\nabla L - \nabla \tilde L|_2^2 + \big(\IE |\nabla L - \nabla \tilde L|_2^2 \cdot \IE |\nabla \tilde L|_2^2\big)^{1/2}\big\}\nonumber\\
		&=& O\big( (nh)^{-2} + (nh)^{-1}h^{\beta \wedge 1}\big).\label{proof_misedecomposition_eq13}
	\end{eqnarray}
	Note that $\IE\big|\nabla \l(\tilde Y_t(u),\theta_0(u))\big|_{V(\theta_0(u))^{-1}}^2 = \tr\{V^{-1}(\theta_0(u)) I(\theta_0(u))\}$. By combining \reff{proof_misedecomposition_eq14}, \reff{proof_misedecomposition_eq12} and \reff{proof_misedecomposition_eq13}, we obtain
	\begin{eqnarray*}
		&& \sup_{u\in[0,1]}\sup_{h\in H_n}\ (nh)\cdot\Big|\IE \big|\nabla L_{n,h}(u,\theta_0(u))\big|^2_{V(\theta_0(u))^{-1}}\\
		&&\quad\quad\quad\quad\quad\quad\quad\quad\quad- \big\{ v_h(u) + |b_h(u)|_{V(\theta_0(u))^{-1}}^2\big\}\Big| \to 0,
	\end{eqnarray*}
	from which the stated convergence follows by integration.\\
	(ii) By Lipschitz continuity of $K$, we have for each component $i = 1,...,p$ that $|b_{h,i}(u) - b_{h,i}(u')| \le \frac{2}{h} L_K \sup_{v\in[0,1]}|\nabla_i \ell(\tilde Y_t(v),\theta_0(u))|$. Furthermore, by using the inequality $|V(\theta)^{-1} - V(\theta')^{-1}|_{spec} \le \lambda_0^{-2}|V(\theta) - V(\theta')|_{2}$ (where $|A|_{spec}$ denotes the spectral norm of a matrix $A$), the bounded variation of $u \mapsto V(\theta_0(u))$ and $w_{n,h}$ and Lemma \ref{lemma_bias_weak}, we have 
	\[
		\textstyle\Big|B_h - \frac{1}{n}\sum_{t=1}^{n}|b_h(t/n)|_{V(\theta_0(t/n))^{-1}}^2 w_{n,h}(t/n)\Big| = O( (nh)^{-1}\cdot h^{\beta \wedge 1} + n^{-1}).
	\]
	The uniform results \reff{proof_misedecomposition_eq2} and \reff{proof_misedecomposition_eq11} together with Lemma \ref{lemma_bias_weak} provide \reff{bias_discrete_approx}.
	
	(iii) If $u \in [\frac{h}{2},1-\frac{h}{2}]$, it holds that $b_h(u) = \int_{-1/2}^{1/2}K(y) \cdot \IE \nabla \l(\tilde Y_t(\theta_0(u+yh)),\theta_0(u)) \dif y$. According to Lemma \ref{lemma_expectation_expansion}, we have uniformly in $u\in[0,1]$:
	\begin{align}
		b_h(u) &= \textstyle\int K(y) \dif y\cdot \IE \nabla \l(\tilde Y_t(\theta_0(u)),\theta_0(u))\nonumber\\
		& \textstyle\quad\quad + h \cdot \int K(y) y \dif y\cdot \IE\big[\partial_u \nabla \l(\tilde Y_t(\theta_0(u)),\theta)\big]\big|_{\theta = \theta_0(u)}\nonumber\\
		&\textstyle\quad\quad + \frac{h^2}{2} \cdot \int K(y) y^2 \dif y\cdot \IE\big[\partial_u \nabla \l(\tilde Y_t(\theta_0(u)),\theta)\big]\big|_{\theta = \theta_0(u)}\nonumber\\
		&\textstyle\quad\quad + \frac{h^2}{2}\int K(y)y^2 R(u,yh)\dif y.\label{misedecomposition_eq20}
	\end{align}
	Since $\nabla \l(\tilde Y_t(\theta_0(u)),\theta_0(u))$ is a martingale difference sequence and $K$ is symmetric, the first two terms in \reff{misedecomposition_eq20} vanish. The last term in \reff{misedecomposition_eq20} is $o(h^2)$ uniformly in $u\in[0,1]$ since $K(y) = 0\gdw |y| \le \frac{1}{2}$ and $\lim_{\xi \to 0}\sup_{u\in[0,1]}|R(u,\xi)| = 0$ from Lemma \ref{lemma_expectation_expansion}.
\end{proof}

\begin{proof}[Proof of Theorem \ref{mle_konsistenz}]
	By application of Corollary \ref{lemma_likelihood_uniform} with $k = 0$, we obtain the uniform convergence
	\[
		\sup_{h \in H_n} \sup_{u \in \supp(w_{n,h})} \sup_{\theta \in \Theta} |L_{n,h}(u,\theta) - L(u,\theta)| \to 0 \quad \text{a.s.}
	\]
	The identifiability condition in Assumption \ref{ass3} implies that $L(u,\theta)$ attains its unique minimum at $\theta = \theta_0(u)$. Standard arguments provide the uniform convergence \reff{mle_konsistenz_gleichung} (see also \cite{dahlhaus2017}). Since $\theta_0(u)$ lies in the interior of $\Theta$, we have that $\hat \theta_h(u)$ lies in the interior of $\Theta$ almost surely for $n$ large enough. With some intermediate value $\bar \theta_h(u)$ satisfying $|\bar \theta_h(u) - \theta_0(u)|_1 \le |\hat \theta_h(u) - \theta_0(u)|_1$, it holds that
	 \begin{eqnarray*}
	 	&& \hat \theta_h(u) - \theta_0(u)\\
		&=& -\big[\nabla^2 L_{n,h}(u,\bar \theta_h(u))\big]^{-1}\cdot \big\{\IE_0 \nabla L_{n,h}(u,\theta_0(u)) + \IE \nabla L_{n,h}(u,\theta_0(u))\big\}
	 \end{eqnarray*}
	 By using the continuity of $\theta \mapsto V(\theta)$ and the uniform convergences of $\nabla^2 L_{n,h}$ provided by Corollary \ref{lemma_likelihood_uniform}, we have that
	 \begin{eqnarray*}
	 	&&|\nabla^2 L_{n,h}(u,\bar \theta_h(u)) - V(\theta_0(u))|_2\\
		&\le& |\nabla^2 L_{n,h}(u,\bar \theta_h(u)) - [\IE \nabla^2 L_{n,h}(u,\theta)]\big|_{\theta = \bar \theta_h(u)}|_2\\
		&&\quad + | [\IE \nabla^2 L_{n,h}(u,\theta)]\big|_{\theta = \bar \theta_h(u)} - V(\bar \theta_h(u))|_2  + |V(\bar \theta_h(u)) - V(\theta_0(u))|_2
	 \end{eqnarray*}
	 converges to $0$ uniformly in $h\in H_n$, $u \in \supp(w_{n,h})$. Since the smallest eigenvalue of $V(\theta)$ is bounded from below by $\lambda_0 > 0$ uniformly in $\theta \in \Theta$, we have that for $n$ large enough, the smallest eigenvalue of $\nabla^2 L_{n,h}(u,\bar \theta_h(u))$ is bounded from below by $\frac{\lambda_0}{2}$ uniformly in $h \in H_n$, $u \in \supp(w_{n,h})$ giving the result \reff{mle_konsistenz_gleichung2}. The arguments for $\hat \theta_{h,-t}(u)$ are similar.
\end{proof}

\begin{proof}[Proof of Lemma \ref{da_approx_0} (additional material)]
We now show that \reff{proof_da_approx_0_eq2} and \reff{proof_da_approx_0_eq3} together imply that $\sup_{h\in H_n}\frac{Q_{n,h}}{d_M^{*}(\hat \theta_h,\theta_0)} \to 0$ a.s., where $Q_{n,h} := |d_A^{*}(\hat \theta_h,\theta_0) - d_{I}^{*}(\hat \theta_h,\theta_0)|$. For each $r > 0$, we can find a space $H_n' := \{\frac{k}{n^r}:k = 1,...,n^r\}$ such that for each $h\in H_n$ there exists some $h' \in H_n'$ such that $|h-h'| \le n^{-r}$. Choose $r$ large enough such that $C(n)\cdot n^{-r} \le C_{\xi} n^{-(1+\xi)}$ with some constants $c_{\xi},\xi > 0$. Let $\delta > 0$ be arbitrarily chosen. Then by Markov's inequality, \reff{proof_da_approx_0_eq2} and \reff{proof_da_approx_0_eq3},
\begin{eqnarray}
	&& \textstyle\IP(\sup_{h\in H_n}\frac{Q_{n,h}}{d_M^{*}(\hat \theta_h,\theta_0)} > \delta)\nonumber\\
	&\le& \textstyle\IP\Big(\sup_{h'\in H_n'}\frac{Q_{n,h'}}{d_M^{*}(\hat \theta_{h'},\theta_0)} > \frac{\delta}{2}\Big)\nonumber\\
	&&\textstyle\quad\quad\quad + \IP\Big(\sup_{h\in H_n,h'\in H_n', |h-h'| \le n^{-r}}\frac{|Q_{n,h} - Q_{n,h'}|}{d_M^{*}(\hat \theta_h,\theta_0)} > \frac{\delta}{2}\Big)\nonumber\\
	&\le& \textstyle\# H_n'\cdot \Big(\frac{\|Q_{n,h'}\|_q}{d_M^{*}(\hat \theta_{h'},\theta_0)\cdot (\delta/2)}\Big)^q + C_{\xi} n^{-(1+\xi)}\frac{\|Z_n\|_2^2}{\delta/2}\nonumber\\
	&\le& \textstyle\frac{C}{(\delta/2)^q} n^{r-\gamma q} + C_{\xi}\frac{C_{\infty}^2\big(1 + (|\chi|_1 D_{2M})^M\big)^2}{(\delta/2)}\cdot n^{-(1+\xi)},\label{proof_da_approx_0_eq11}
\end{eqnarray}
which is absolutely summable for $q$ large enough, giving the result by applying Borel-Cantelli's lemma. In the following we will use this technique for similar expressions without explicitly showing results of the type \reff{proof_da_approx_0_eq3}. The proofs are similar to the proof of \reff{proof_da_approx_0_eq3}. Following the argumentation in \reff{proof_lemma_approxmart_eq1}, we obtain
\begin{eqnarray}
	&&\textstyle\Big\| d_I^{*}(\hat \theta_h,\theta_0) - \int_{0}^{1}|\nabla \hat L_{n,h}(u,\theta_0(u))|^2_{V(\theta_0(u))} w_{n,h}(u)\dif u\Big\|_q\nonumber\\
	&\le& \textstyle\frac{2C_{\infty,2q}p^2|V|_{\infty}|w|_{\infty}C_{\delta,2q}}{n(nh)}\cdot \int_{-1/2}^{1/2}\Big\{|K(v)| \nonumber\\
	&&\textstyle\quad \times \sum_{s=1}^{n} \|X_{s,n} - \tilde X_s(s/n)\|_{2qM}\sum_{r=1}^{n}\Big|K\Big(\frac{s-r}{nh}+vh\Big)\Big|\Big\}\dif v = O(n^{-1}). \label{proof_da_approx_0_eq4}
\end{eqnarray}
For $u\in [0,1]$, it holds that
\begin{eqnarray}
	&& |\nabla \hat L_{n,h}(u,\theta_0(u))|_{V(\theta_0(u))}^2 - \IE |\nabla \hat L_{n,h}(u,\theta_0(u))|_{V(\theta_0(u))}^2\nonumber\\
	&=& \big\{A_1(u) - \IE A_1(u)\big\} + 2A_2(u),\label{keine_ew_noetig}
\end{eqnarray}
where $A_1(u) := |\nabla \hat L_{n,h}(u,\theta_0(u)) - \IE \nabla \hat L_{n,h}(u,\theta_0(u))|_{V(\theta_0(u))}^2$ and $A_2(u) := \langle \nabla \hat L_{n,h}(u,\theta_0(u)) -   \IE\nabla \hat L_{n,h}(u,\theta_0(u)), V(\theta_0(u)) \IE \nabla \hat L_{n,h}(u,\theta_0(u))\rangle$.
We now derive upper bounds for
\begin{eqnarray*}
	\textstyle\|d_I^{*}(\hat \theta_h,\theta_0) - \IE d_{I}^{*}(\hat \theta_h,\theta_0)\|_q &\le&\textstyle \big\|\int_{0}^{1}\big\{A_1(u) - \IE A_1(u)\big\}w_{n,h}(u) \dif u\big\|_q\\
	&&\textstyle\quad\quad + 2\big\|\int_{0}^{1}A_2(u) w_{n,h}(u) \dif u\big\|_q.
\end{eqnarray*}
Lemma \ref{hilfslemma1}, \reff{hilfslemma1_eq2} gives
\begin{eqnarray}
	 &&\textstyle\Big\|\int_{0}^{1}\big\{A_1(u) - \IE A_1(u)\big\} w_{n,h}(u) \dif u\Big\|_q\nonumber\\
	 &\le&\textstyle \frac{V_{1,max}|w|_{\infty}}{n(nh)}\Big\|\sup_{i,j=1,...,d}\int_{-1/2}^{1/2}\Big\{|K(v)|\cdot \Big|\sum_{r,s=1}^{n}K\Big(\frac{r-s}{nh}+v\Big)\nonumber\\
	 &&\textstyle\quad\quad\quad\quad\quad\quad\quad\quad\quad\times \nabla_i \l(\tilde Y_r(r/n),\theta_0(r/n-vh))\nonumber\\
	 &&\textstyle\quad\quad\quad\quad\quad\quad\quad\quad\quad\times \nabla_j \l(\tilde Y_s(s/n),\theta_0(r/n-vh))\Big|\Big\} \dif v\Big\|_q\nonumber\\
	 &\le&\textstyle \rho_{2,\psi C_{\delta,\cdot},q}\cdot \frac{p^2 V_{1,max} |w|_{\infty}}{n(nh)} \int_{-1/2}^{1/2}\Big\{|K(v)|\nonumber\\
	 &&\textstyle\quad\quad\quad\quad\quad\times \Big(\sum_{r,s=1}^{n}K\Big(\frac{r-s}{nh} + v\Big)^2 \Big)^{1/2} \Big\}\dif v,\nonumber\\
	 &\le&\textstyle \rho_{2,\psi C_{\delta,\cdot},q}\cdot p^2 V_{1,max} |w|_{\infty} |K|_{\infty}^2 (n(nh))^{-1/2},\label{proof_da_approx_0_eq5}
\end{eqnarray}
where $\psi_q(t)$ and $C_{\delta,q}$ are defined in Lemma \ref{pmomente}. Furthermore, we have
\begin{eqnarray}
	 &&\textstyle\Big\|\int_{0}^{1}A_2(u) w_{n,h}(u)\dif u\Big\|_q\nonumber\\
	 &\le&\textstyle \frac{1}{n}\Big\|\sup_{i,j=1,...,d}\int_{-1/2}^{1/2}\Big\{|K(v)|\cdot \Big| \sum_{s=1}^{n} V(\theta_0(t/n))_{ij}w_{n,h}(s/n-vh)\nonumber\\
	 &&\textstyle\quad\quad \times \IE \nabla_i \hat L_{n,h}(s/n-vh,\theta_0(s/n-vh))\nonumber\\
	 &&\textstyle\quad\quad \times \big\{\nabla \l(\tilde Y_{s}(s/n),\theta_0(s/n-vh))\nonumber\\
	 &&\textstyle\quad\quad\quad\quad -\IE \nabla_j \l(\tilde Y_{s}(s/n),\theta_0(s/n-vh))\big\} \Big| \Big\}\dif v\Big\|_q\nonumber\\
	 &\le&\textstyle \rho_{1,\psi C_{\delta,\cdot},q}\cdot \frac{p^2 V_{1,max}}{n^{1/2}}\sup_{i=1,...,d}\int_{-1/2}^{1/2}|K(v)|\nonumber\\
	 &&\textstyle\quad\quad \times \Big(\frac{1}{n}\sum_{s=1}^{n}|\IE \nabla_i \hat L_{n,h}(s/n-vh,\theta_0(s/n-vh))|^2\nonumber\\
	 &&\textstyle\quad\quad\quad\quad\quad\quad\quad\quad\times w_{n,h}(s/n-vh)^2\Big)^{1/2}\dif v.
\end{eqnarray}
Put $\hat b_{h,i}(u) = \IE \nabla_{i}\hat L_{n,h}(u,\theta_0(u))$ and $\hat B_h$ from Proposition \ref{misedecomposition}. By Lemma \ref{lemma_stetigkeit}, we have $|\hat b_{h,i}(u) - \hat b_{h,i}(v)| \le C_{-,1}\big\{\frac{|u-v|}{h} + |\theta_0(u) - \theta_0(v)|_1\big\}$ and $|\hat b_{h,i}(u)| \le C_{\infty,1}$. This implies
\begin{eqnarray*}
	&& \textstyle\Big|\frac{1}{n}\sum_{t=1}^n |\hat b_{h,i}(s/n-vh)|^2 w_{n,h}(s/n-vh)\\
	&&\textstyle\quad\quad\quad\quad - \int_{0}^{1}|\hat b_{h,i}(u-vh)|^2 w_{n,h}(u-vh) \dif u\Big|\\
	&\le& \textstyle\frac{B_w}{n} + \frac{2C_{\infty,1}C_{-,1}}{nh}\big\{1 + B_{\theta_0}\big\}.
\end{eqnarray*}
Since $\int_{0}^{1}|\hat b_{h,i}(u-vh)|^2 w_{n,h}(u-vh) \dif u \le \frac{p^2}{\lambda_0}\hat B_h$, we obtain
\begin{equation}
	\Big\|\int_{0}^{1}A_2(u) w_{n,h}(u)\dif u\Big\|_q = O\big(n^{-1/2}(\hat B_h + (nh)^{-1/2})\big)\label{proof_da_approx_0_eq6}
\end{equation}
The bounds \reff{proof_da_approx_0_eq4}, \reff{proof_da_approx_0_eq5} and \reff{proof_da_approx_0_eq6} imply $\sup_{h\in H_n}\big|\frac{d_I^{*}(\hat \theta_h,\theta_0) - d_M^{*}(\hat \theta_h,\theta_0)}{d_M^{*}(\hat \theta_h, \theta_0)}\big| \to 0$ a.s.
\end{proof}

\begin{proof}[Proof of Lemma \ref{da_strich_approx} (additional material)]
Using similar techniques as in \reff{proof_lemma_approxmart_eq1}, we obtain $\|W_{n} - \hat W_{n}\|_q \le 2C_{\infty,2q}p^2 V_{-1,max}C_{\delta,2q} n^{-1}$, which shows $W_{n}-\hat W_{n}\to 0$ a.s. Furthermore, by \reff{hilfslemma1_eq2}, we have for arbitrary $q >0$ that $\|\hat W_{n} - \IE \hat W_{n}\|_q \le \rho_{2,\psi C_{\delta,\cdot},q}p^2 V_{-1,max}n^{-1/2}$, showing that $\hat W_{n} - \IE \hat W_{n} \to 0$ a.s. By the bound \reff{proof_uniform_eq1},
\[
	\|\hat W_{n}\|_1 \le C_{\infty}^2\|1 + |Y_{t,n}^c|_{\chi,1}^M\|_{2}^2 \le C_{\infty}^2\big(1+ |\chi|_1 D_{2M}\big)^2,
\]
which finally shows that $|\IE \hat W_{n}|$ is bounded and thus $W_{n}$ is bounded a.s.
\end{proof}

\begin{proof}[Proof of Lemma \ref{da_cv_approx} (additional material)]
For shortness, let us define $\SUP := \sup_{i}\sup_{t=1,...,n}\sup_{u\in \supp(w_{n,h})}\sup_{\theta \in \Theta}$, where the supremum over $i$ should be understood as the maximum over all components of the argument (if it is a vector or matrix). It is well-known from Corollary \reff{lemma_likelihood_uniform} that for $k = 0,1,2,3$ and all $0 < \alpha \le \frac{1}{2}$,
\begin{equation}
	\sup_{h\in H_n}\SUP (nh)^{\frac{1}{2}-\alpha}|\IE_0 \nabla^k_i L_{n,h,-t}(u,\theta)| \to 0 \quad \text{a.s.}\label{proof_da_cv_approx_eq17}
\end{equation}
Furthermore, let $C^{\#}$ denote a generic constant depending only on $p, V_{-1,max}$ and $V_{1,max}$ which may change its value from line to line.

We first discuss $R_{n,h,1}$. By using \reff{proof_da_cv_approx_eq3}, $R_{n,h,1}$ can be expanded into three terms: 
\begin{eqnarray}
	&& R_{n,h,1}\nonumber\\
	&=& \textstyle -\frac{1}{n}\sum_{t=1}^{n}\nabla \l_{t,n}(\theta_0(t/n))\cdot (\IE V_{t,n})^{-1}\cdot \nabla L_{n,h,-t}(t/n,\theta_0(t/n)) w_{n,h}(t/n)\nonumber\\
	&&\textstyle\quad\quad -\frac{1}{n}\sum_{t=1}^{n}\nabla \l_{t,n}(\theta_0(t/n))\cdot (\IE V_{t,n})^{-1}\big(\IE V_{t,n} - V_{t,n}) (\IE V_{t,n})^{-1}\nonumber\\
	&&\textstyle\quad\quad\quad\quad \times \nabla L_{n,h,-t}(t/n,\theta_0(t/n)) w_{n,h}(t/n)\nonumber\\
	&&\textstyle\quad\quad -\frac{1}{n}\sum_{t=1}^{n}\nabla \l_{t,n}(\theta_0(t/n))\cdot V_{t,n}^{-1}\big(\IE V_{t,n} - V_{t,n})(\IE V_{t,n})^{-1}\nonumber\\
	&&\textstyle\quad\quad\quad\quad\times \big(\IE V_{t,n} - V_{t,n}) (\IE V_{t,n})^{-1}\cdot \nabla L_{n,h,-t}(t/n,\theta_0(t/n)) w_{n,h}(t/n)\label{decomp_eq1}
\end{eqnarray}
The first summand and the second summand in \reff{decomp_eq1} are discussed in Lemma \ref{lemma_approxmart}, \reff{lemma_approxmart_eq1} and \reff{lemma_approxmart_eq3}. Put $Z_{n,k} := 1 + \frac{1}{n}\sum_{t=1}^{n}|Y_{t,n}^c|_{\chi,1}^{kM}$. Due to the Cauchy Schwarz inequality, the third summand in \reff{decomp_eq1} is bounded by
\begin{eqnarray*}
	 R_{n,h,1,3} &:=& C^{\#}\cdot \SUP\big| \IE_0 \nabla_{i}^2 L_{n,h,-t}(u,\theta)\big|^2\\
	 &&\quad\quad \times \Big\{|w|_{\infty}\SUP \big| \IE_0\nabla_i L_{n,h,-t}(u,\theta)\big|\cdot C_{\infty} Z_{n,1}\\
	 	 &&\quad\quad\quad\quad\quad\quad\quad\quad + (2|w|_{\infty})^{1/2}C_{\infty}Z_{n,2}^{1/2}\cdot\lambda_0^{-1/2}(B_h^{dis})^{1/2}\Big\}
\end{eqnarray*}
Here we used the fact that for $v \in \IR^d$ and $j = 1,...,p$, we have $|v_j|^2 \le |v|_{\text{Id}}^2 \le \frac{1}{\lambda_0}|v|_{V(\theta)}^2$ for all $\theta \in \Theta$ by the assumption on $V(\theta)$. $B_h^{dis}$ is defined in Proposition \ref{misedecomposition}(ii). By \reff{proof_da_cv_approx_eq17}, \reff{bias_discrete_approx} and Proposition \ref{misedecomposition}(i),(ii) we conclude that $\sup_{h\in H_n}\frac{R_{n,h,1,3}}{d_M^{*}(\hat \theta_h,\theta_0)} \to 0$ a.s. and $\sup_{h\in H_n}\frac{R_{n,h,1}}{d_M^{*}(\hat \theta_h,\theta_0)} \to 0$ a.s.


The next part of the proof addresses $R_{n,h,2}$. Our goal is to eliminate $V_{t,n}^{-1}$ and the intermediate values. Put
\begin{eqnarray*}
	&& R_{n,h,2,1}\\
	&:=& \frac{1}{n}\sum_{t=1}^{n}w_{n,h}(t/n)\cdot \nabla \l_{t,n}(\theta_0(t/n))' (\IE V_{t,n})^{-1}\big(\IE \nabla^3 L_{n,h,-t}(t/n,\theta_0(t/n))\big)\\
	&&\quad \times \big[(\IE V_{t,n})^{-1} \nabla L_{n,h,-t}(t/n,\theta_0(t/n)),(\IE V_{t,n})^{-1} \nabla L_{n,h,-t}(t/n,\theta_0(t/n))\big].
\end{eqnarray*}
We now argue how $R_{n,h,2,1}$ can be obtained by successively replacing terms in $R_{n,h,2}$ with a rate smaller than $d_{M}^{*}(\hat \theta_h,\theta_0)$. The remainder by replacing $V_{t,n}^{-1}$ by $(\IE V_{t,n})^{-1}$ in $R_{n,h,2}$ is bounded by (see \reff{proof_da_cv_approx_eq25} and \reff{mle_konsistenz_gleichung2}):
\begin{eqnarray}
	&& C^{\#} \SUP |\IE_0 \nabla_{i}^2 L_{n,h,-t}(u,\theta)|\cdot \Big\{|w|_{\infty}C_{\infty} Z_n \cdot \SUP |\IE_0 \nabla_{i} L_{n,h,-t}(u,\theta)|^2\nonumber\\
	&& + \frac{1}{n}\sum_{t=1}^{n}\big(|Y_{t,n}^c|_{\chi,1}^M - \IE |Y_{t,n}^c|_{\chi,1}^{M}\big) w_{n,h}(t/n) \sup_{i=1,...,p}|\IE \nabla_i L_{n,h,-t}(t/n,\theta_0(t/n))|^2\nonumber\\
	&&\quad\quad\quad + C_{\infty}\IE Z_n\cdot |w|_{\infty}^{1/2}\lambda_0^{-1}B_h^{dis}\Big\}\label{proof_da_cv_approx_eq31}
\end{eqnarray}
which is of order $(nh)^{\frac{1}{2}-\alpha}( (nh)^{1-2\alpha} + n^{\frac{1}{2}-\alpha} + B_h^{dis})$ by \reff{proof_da_cv_approx_eq17}.\\
The same arguments hold for the remainder which is obtained by replacing $\nabla^3 L_{n,h,-t}(t/n,\tilde \theta_{h,-t}(t/n))$ with $\IE[\nabla^3 L_{n,h,-t}(t/n,\theta)]\big|_{\theta = \tilde \theta_{h,-t}(t/n)}$. For the next replacement note that for arbitrary $g \in \sL(M,C,\chi)$, we have
\begin{equation}
	\big|\frac{1}{n}\sum_{t=1}^{n}K_h\big(\frac{t}{n}-u\big)\big\{g(Y_{t,n}^c,\theta_1) - g(Y_{t,n}^c,\theta_2)\big\}\big| \le C_2 E_n(K_h(\cdot-u),\tilde g,\theta)\cdot |\theta_1 - \theta_2|_1,\label{proof_da_cv_approx_eq18}
\end{equation}
where $\tilde g(y,\theta) = 1+|y|_{\chi,1}^{M}$. By Lemma \ref{lemma_uniformconvergence}, $E_n(K_h(\cdot-u),\tilde g,\theta)$ is bounded almost surely. The remainder by replacing $\IE[\nabla^3 L_{n,h,-t}(t/n,\theta)]\big|_{\theta = \tilde \theta_{h,-t}(t/n)}$ with $\IE \nabla^3 L_{n,h,-t}(t/n,\theta_0(t/n))$ is bounded by (see \reff{mle_konsistenz_gleichung2} and \reff{proof_da_cv_approx_eq18})
\begin{eqnarray}
	&& C^{\#} \Big\{ |w|_{\infty} C_{\infty} Z_n \cdot \SUP |\IE_0 \nabla_i L_{n,h,-t}(u,\theta)|^3\nonumber\\
	&& + \frac{C_{\infty}}{n}\sum_{t=1}^{n}w_{n,h}(t/n)\big(|Y_{t,n}^c|_{\chi,1}^{M} - \IE |Y_{t,n}^c|_{\chi,1}^{M}\big) \sup_{i=1,...,p}|\IE \nabla_i L_{n,h,-t}(t/n,\theta_0(t/n))|^3 \nonumber\\
	&&\quad\quad + C_{\infty}|w|_{\infty} \cdot \sup_{t=1,...,n}\big(1 + \IE |Y_{t,n}^c|_{\chi,1}^M\big)\cdot \lambda_0^{-1}B_h^{dis} \cdot\SUP|\IE \nabla_i L_{n,h,-t}(u,\theta)|\Big\}\label{proof_da_cv_approx_eq32}.
\end{eqnarray}
While the second summand is discussed in Lemma \ref{lemma_approxmart}, \reff{lemma_approxmart_eq5}, the first and the third summand are of order $O((nh)^{\frac{3}{2}-3\alpha}) + O(h^{(\beta\wedge 1)}B_h^{dis})$ by Lemma \ref{lemma_bias_weak} and \reff{proof_da_cv_approx_eq17}.

Lastly we replace $\hat \theta_{h,-t}(t/n)-\theta_0(t/n)$ twice by $(\IE V_{t,n})^{-1}\nabla L_{n,h,-t}(t/n,\theta_0(t/n))$ by using the expansion
\begin{eqnarray*}
	&& \hat \theta_{h,-t}(u) - \theta_0(u)\\
	&=& -[\IE \nabla^2 L_{n,h,-t}(u,\theta_0(u))]^{-1} \big\{ \nabla L_{n,h,-t}(u,\theta_0(u))\\
	&&\quad\quad + \big(\IE \nabla^2 L_{n,h,-t}(u,\theta_0(u)) - \nabla^2 L_{n,h,-t}(u,\breve \theta_{h,-t}(u))\big)\\
	&&\quad\quad\quad\quad\times [\nabla^2 L_{n,h,-t}(u,\breve \theta_{h,-t}(u))]^{-1} \nabla L_{n,h,-t}(u,\theta_0(u))\big\},
\end{eqnarray*}
where $|\breve \theta_{h,-t}(u) - \theta_0(u)|_2 \le |\hat \theta_{h,-t}(u) - \theta_0(u)|_2$. The remainder is bounded by terms similar to \reff{proof_da_cv_approx_eq31} and \reff{proof_da_cv_approx_eq32}.

With Proposition \ref{misedecomposition}(i),(ii) we conclude that $\sup_{h\in H_n}\big|\frac{R_{n,h,2} - R_{n,h,2,1}}{d_M^{*}(\hat \theta_h,\theta_0)}\big| \to 0$ a.s. In Lemma \ref{lemma_approxmart}, \reff{lemma_approxmart_eq4} it is shown that $\sup_{h\in H_n}\big|\frac{R_{n,h,2,1}}{d_M^{*}(\hat \theta_h,\theta_0)}\big| \to 0$ a.s.

Finally, put
\begin{eqnarray*}
	R_{n,h,3,1} &:=& \frac{1}{n}\sum_{t=1}^{n}w_{n,h}(t/n)\\
	&&\quad\quad \times \big| (\IE V_{t,n})^{-1}\nabla L_{n,h,-t}(t/n,\theta_0(t/n))\big|^2_{\nabla \l_{t,n}(\theta_0(t/n)) - V(\theta_0(t/n))}.
\end{eqnarray*}
We shortly explain how $R_{n,h,3,1}$ can be obtained from $R_{n,h,3}$. The intermediate value $\tilde \theta_{h,-t}(t/n)$ is replaced by  $\theta_0(t/n)$ with remainder given by
\begin{eqnarray*}
	&& \big|\frac{1}{n}\sum_{t=1}^{n}|\hat \theta_{h,-t}(t/n) - \theta_0(t/n)|^2_{\nabla^2 \l_{t,n}(\tilde \theta_{h,-t}(t/n)) - \nabla^2 \l_{t,n}(\theta_0(t/n))} w_{n,h}(t/n)\big|\\
	&\le& \frac{C^{\#}}{n}\sum_{t=1}^{n} C_2 (1 + |Y_{t,n}^c|_{\chi,1}^M) w_{n,h}(t/n)\\
	&&\quad \times \sup_{i=1,...,p}\big\{|\IE_0 \nabla_i L_{n,h,-t}(t/n,\theta_0(t/n))|^3 + |\IE \nabla_i L_{n,h,-t}(t/n,\theta_0(t/n))|^3\big\}
\end{eqnarray*}
which is similar handled as in \reff{proof_da_cv_approx_eq32}. The replacement of $\hat \theta_{h,-t}(t/n) - \theta_0(t/n)$ by $(\IE V_{t,n})^{-1}\nabla L_{n,h,-t}(t/n,\theta_0(t/n))$ is done as for $R_{n,h,2}$. We conclude with Proposition \ref{misedecomposition}(i),(ii) that $\sup_{h \in H_n}\big|\frac{R_{n,h,3} - R_{n,h,3,1}}{d_M^{*}(\hat \theta_h,\theta_0)}\big| \to 0$ a.s. The convergence $\sup_{h\in H_n}\big|\frac{R_{n,h,3,1}}{d_M^{*}(\hat \theta_h,\theta_0)}\big| \to 0$ a.s. is proved in Lemma \ref{lemma_approxmart}, \reff{lemma_approxmart_eq2}.

\end{proof}

\begin{proof}[Proof of Lemma \ref{lemma_approxmart}:]
	We have uniformly in $u,\theta$ that $\delta_{q}^{\nabla \l(\tilde Y(u),\theta)}(k) \le C_{\delta,q} \psi_q(k)$ (see the bound in \reff{dependence_measure_std}), where $\psi_q$ and $C_{\delta,q}$ are defined in Lemma \ref{pmomente}. Furthermore, we can use the same argumentation as in \reff{lemma_bias_weak_eq1} to see that uniformly in $u,u',\theta$ it holds that $\|\nabla \l(\tilde Y_0(u),\theta) - \nabla \l(\tilde Y_0(u'),\theta)\|_2 \le C_{\delta,2}\|\tilde X_0(u) - \tilde X_0(u')\|_{2M} \le C_{\delta,2}C_A |\theta_0(u) - \theta_0(u')|_1$. Since $\theta_0$ is Hoelder continuous, there exists $\tilde L > 0$ such that $|\theta_0(u) - \theta_0(u')|_1 \le d \tilde L |u-u'|^{\beta \wedge 1}$. As in the proof of Lemma \ref{da_cv_approx}, let $C^{\#}$ denote a generic constant depending only on $p, V_{-1,max}$ and $V_{1,max}$ which may change its value from line to line. We use the same technique as in the proof of Lemma \ref{da_approx_0}, \reff{proof_da_approx_0_eq11}, but omit the proofs on Lipschitz continuity in $h$ (cf. \reff{proof_da_approx_0_eq3}) since they do not pose any extra difficulty.
	
	We start by proving \reff{lemma_approxmart_eq1}. Put
	\begin{eqnarray*}
		R_{0,n} &:=& \frac{1}{n}\sum_{t=1}^{n} w_{n,h}(t/n)\big\{\nabla \l(Y_{t,n}^c,\theta_0(t/n)) (\IE V_{t,n})^{-1} \nabla L_{n,h,-t}(t/n, \theta_0(t/n))\\
		&&\quad\quad - \nabla \l(\tilde Y_{t}(t/n),\theta_0(t/n)) (\IE V_{t,n})^{-1}  \nabla \hat L_{n,h,-t}(t/n, \theta_0(t/n))\big\}.
	\end{eqnarray*}
	By Hoelder's inequality (cp. also Lemma \ref{lemma_stat}) it follows that
	\begin{eqnarray}
		&& \| R_{0,n}\|_q\label{proof_lemma_approxmart_eq1}\\
		&\le& \frac{C^{\#}C_{\infty,2q}|w|_{\infty}|K|_{\infty}}{n}\cdot \sup_{j=1,...,d}\sum_{t=1}^{n}\sup_{\theta \in \Theta}\big\|\nabla_j \l(Y_{t,n}^c,\theta) - \nabla_j\l(\tilde Y_t(t/n),\theta)\big\|_{2q}\nonumber\\
		&\le& \frac{C^{\#} C_{\infty,2q}|w|_{\infty}|K|_{\infty} C_{\delta,2q}}{n}\cdot  \sum_{t=1}^{n}\|X_{t,n} - \tilde X_t(t/n)\|_{2qM} = O(n^{-1}).\nonumber
	\end{eqnarray}	 
	Since $\nabla \l(\tilde Y_t(u),\theta_0(u))$ are martingale differences, $\nabla \hat L_{n,h,-t}(t/n,\theta_0(t/n)) - \IE \nabla \hat L_{n,h,-t}(t/n,\theta_0(t/n)) - \nabla \tilde L_{n,h,-t}(t/n,\theta_0(t/n))$ has expectation 0 and we obtain
	\begin{eqnarray*}
		&&R_{1,n}\\
		&:=& \frac{1}{n}\Big|\IE\Big[\sum_{t=1}^{n}w_{n,h}(t/n)\nabla \l(\tilde Y_t(t/n),\theta_0(t/n))\cdot (\IE V_{t,n})^{-1}\\
		&&\quad\quad\quad\times  \big\{\nabla \hat L_{n,h,-t}(t/n,\theta_0(t/n)) - \IE \nabla \hat L_{n,h,-t}(t/n,\theta_0(t/n))\\
		&&\quad\quad\quad\quad\quad - \nabla \tilde L_{n,h,-t}(t/n,\theta_0(t/n))\big\}\Big]\Big|\\
		&\le& \frac{|w|_{\infty}}{n(nh)}\sum_{k=0}^{\infty}\Big| \sum_{s\not= t}K\Big(\frac{s-t}{nh}\Big)\IE\big[ \nabla \l(\tilde Y_t(t/n),\theta_0(t/n))(\IE V_{t,n})^{-1}\\
		&&\quad\quad\quad\times P_{s-k}\big\{\nabla \l(\tilde Y_s(s/n), \theta_0(t/n)) - \nabla \l(\tilde Y_s(t/n),\theta_0(t/n))\big\}\big]\Big|\\
		&\le& \frac{|w|_{\infty}|K|_{\infty}}{n(nh)}\sum_{k=0}^{\infty} \sum_{s=1}^{n}\big\|\nabla \l(\tilde Y_{s-k}((s-k)/n),\theta_0((s-k)/n))\big\|_2 (\IE V_{t,n})^{-1}\\
		&&\times \min\big\{2C_{\delta,2} \psi_2(k), \big\|\nabla \l(\tilde Y_s(\frac{s}{n}),\theta_0(\frac{t}{n})) - \nabla \l(\tilde Y_s(\frac{s-k}{n}),\theta_0(\frac{s-k}{n}))\big\|_2\big\}\\
		&\le& \frac{C^{\#} |w|_{\infty}|K|_{\infty}C_{\infty,2}}{nh}\sum_{k=0}^{\infty} \min\big\{2C_{\delta,2}\psi_2(k), C_{\delta,2}C_A p \tilde L \big(\frac{k}{n}\big)^{\beta \wedge 1}\big\},
	\end{eqnarray*}
	which shows that $\sup_{h\in H_n}\frac{R_{1,n}}{d_{M}^{*}(\hat \theta_h,\theta_0)} \to 0$ since $\psi_2$ is absolutely summable. Define
	\[
		R_{2,n} := \frac{1}{n}\sum_{t=1}^{n}\nabla \l(\tilde Y_{t}(t/n),\theta_0(t/n))(\IE V_{t,n})^{-1}\IE \nabla \hat L_{n,h,-t}(t/n,\theta_0(t/n))w_{n,h}(t/n).
	\]
	Recall that $\psi_{q}(k)$ is absolutely summable. Furthermore, for some $v \in \IR^d$, note that for arbitrary $j = 1,...,p$, we have $|v_j|^2 \le |v|_{\text{Id}}^2 \le \frac{1}{\lambda_0}|v|_{V(\theta)}^2$ for all $\theta \in \Theta$ by the assumption on $V(\theta)$. By Lemma \ref{hilfslemma1}, we have
	\begin{eqnarray}
		\|R_{2,n}\|_q &\le& C^{\#}C_{\delta,q}\rho_{1,\psi,q}|_{\infty}\cdot n^{-1/2}\nonumber\\
		&&\quad\quad\times \sup_{j=1,...,d}\Big(\frac{1}{n}\sum_{t=1}^{n}\big|\IE \nabla_j \hat L_{n,h,-t}(t/n,\theta_0(t/n))\big|^2 w_{n,h}(t/n)^2\Big)^{1/2}\nonumber\\
		&\le& C^{\#}C_{\delta,q}\rho_{1,\psi,q}\frac{|w|_{\infty}^{1/2}}{\lambda_0^{1/2}}\cdot n^{-1/2} B_h^{dis},\label{proof_lemma_approxmart_eq2}
	\end{eqnarray}
	which shows that  $\sup_{h\in H_n}\frac{|R_{2,n}|}{d_{M}^{*}(\hat \theta_h,\theta_0)} \to 0$ a.s.  Finally, put
	\[
		R_{3,n} := \frac{1}{n}\sum_{t=1}^{n}\nabla \l(\tilde Y_t(t/n),\theta_0(t/n))(\IE V_{t,n})^{-1}\nabla \tilde L_{n,h,-t}(t/n,\theta_0(t/n)).
	\]
	Since $\nabla \l(\tilde Y_0(u),\theta_0(u))$ are martingale differences, it holds that $\IE R_{3,n} = 0$.  Since $\chi_i = O(i^{-(3+\varepsilon)})$ and $\delta_q(k) = O(k^{-(3+\varepsilon)})$, we conclude that $\sum_{i \ge 0}i \chi_i < \infty$ and $\sum_{k \ge 0}k \delta_q(k) < \infty$, which also shows $\sum_{k\ge 0} k\psi_q(k) < \infty$. Thus Lemma \ref{hilfslemma1} is applicable and we obtain
	\begin{eqnarray}
		\|R_{3,n}\|_q &\le& C^{\#}\rho_{2,\psi C_{\delta,\cdot},q}|w|_{\infty}\cdot (n(nh))^{-1} \Big(\sum_{t=1}^{n}\sum_{s\not= t}K\Big(\frac{t-s}{nh}\Big)^2\Big)^{1/2}\nonumber\\
		&=& O\big((n(nh))^{-1/2}\big),\label{proof_lemma_approxmart_eq3}
	\end{eqnarray}
	which shows that  $\sup_{h\in H_n}\frac{|R_{3,n}|}{d_{M}^{*}(\hat \theta_h,\theta_0)} \to 0$ a.s.
	
	We now prove \reff{lemma_approxmart_eq2}. Similar arguments as in \reff{proof_lemma_approxmart_eq1} can be used to show that 
	\begin{eqnarray*}
		S_{0,n} &:=& \frac{1}{n}\sum_{t=1}^{n}w_{n,h}(t/n)\\
		&& \times \big\{|(\IE V_{t,n})^{-1}\nabla  L_{n,h,-t}(t/n,\theta_0(t/n))|_{\nabla^2 \l(Y_{t,n}^c,\theta_0(t/n)) - V(\theta_0(t/n))}^2\\
		&&\quad - |(\IE V_{t,n})^{-1}\nabla \hat L_{n,h,-t}(t/n,\theta_0(t/n))|_{\nabla^2 \l(\tilde Y_{t}(t/n),\theta_0(t/n)) - V(\theta_0(t/n))}^2\big\}
	\end{eqnarray*}
	fulfills $\sup_{h\in H_n}\frac{|S_{0,n}|}{d_M^{*}(\hat \theta_h,\theta_0)} \to 0$ a.s. Put
	\begin{eqnarray*}
		S_{3,n} &:=& \frac{1}{n}\sum_{t=1}^{n}  w_{n,h}(t/n) |(\IE V_{t,n})^{-1}\big\{\nabla \hat L_{n,h,-t}(t/n,\theta_0(t/n))\\
		&&\quad\quad\quad - \IE \nabla \hat L_{n,h,-t}(t/n,\theta_0(t/n))\big\}|_{\nabla^2 \l(\tilde Y_t(t/n), \theta_0(t/n)) - V(\theta_0(t/n))}^2.
	\end{eqnarray*}
	Decompose
	\begin{eqnarray*}
		&& \frac{1}{n}\sum_{t=1}^{n}\big\{|V^{-1}\nabla \hat L_{n,h,-t}(t/n,\theta_0(t/n))|_{\nabla^2 \l(\tilde Y_t(t/n), \theta_0(t/n)) - V(\theta_0(t/n))}w_{n,h}(t/n)\\
		&&- S_{3,n} = 2S_{1,n} + S_{2,n},
	\end{eqnarray*}
	where
	\begin{eqnarray*}
		&& S_{1,n}\\
		&:=& \frac{1}{n}\sum_{t=1}^{n}\big\langle (\IE V_{t,n})^{-1}\big\{\nabla \hat L_{n,h,-t}(t/n,\theta_0(t/n)) - \IE \nabla \hat L_{n,h,-t}(t/n,\theta_0(t/n))\big\},\\
		&&\quad\quad\quad\quad \big\{\nabla^2 \l(\tilde Y_t(t/n),\theta_0(t/n)) - V(\theta_0(t/n))\big\}\\
		&&\quad\quad\quad\quad\quad\quad\times (\IE V_{t,n})^{-1}\IE \nabla \hat L_{n,h,-t}(t/n,\theta_0(t/n))\big\rangle w_{n,h}(t/n),\\
		&& S_{2,n}\\
		&:=& \frac{1}{n}\sum_{t=1}^{n}w_{n,h}(t/n)\\
		&& \quad\quad\times |(\IE V_{t,n})^{-1}\IE \nabla \hat L_{n,h,-t}(t/n,\theta_0(t/n))|_{\nabla^2 \l(\tilde Y_t(t/n), \theta_0(t/n)) - V(\theta_0(t/n))}^2.
	\end{eqnarray*}
	Since $\nabla^2 \l \in g(M,\chi,C)$, we obtain by Lemma \ref{hilfslemma1}, similarly to \reff{proof_lemma_approxmart_eq2},
	\[
		\|S_{2,n}\|_q \le C^{\#}\rho_{1,\psi C_{\delta,\cdot},q}\cdot C_{\infty,1}\frac{|w|_{\infty}^{1/2}}{\lambda_0^{1/2}}\cdot n^{-1/2}B_{h}^{dis}.
	\]
	By Lemma \ref{hilfslemma1}, \reff{hilfslemma1_eq4},
	\begin{eqnarray*}
		&& |\IE S_{1,n}|\\
		&\le& \frac{C^{\#} |K|_{\infty} \tilde \rho_{2,\psi C_{\delta,\cdot},q}}{n(nh)}\sup_{j=1,...,d}\sum_{t=1}^{n}\IE|\nabla_j \hat L_{n,h,-t}(t/n,\theta_0(t/n))|\cdot w_{n,h}(t/n).
	\end{eqnarray*}
	By the Cauchy-Schwarz inequality, we have
	\begin{eqnarray*}
		&&\sup_{j=1,...,d}\frac{1}{n}\sum_{t=1}^{n}|\IE\nabla_j \hat L_{n,h,-t}(t/n,\theta_0(t/n))|\cdot w_{n,h}(t/n)\\
		&\le& \sup_{j=1,...,d}\big(\frac{1}{n}\sum_{t=1}^{n}|\IE\nabla_j \hat L_{n,h,-t}(t/n,\theta_0(t/n))|^2\cdot w_{n,h}(t/n)^2\big)^{1/2}\\
		&\le& |w|_{\infty}^{1/2}\lambda_0^{-1/2}(\hat B_h^{dis})^{1/2},
	\end{eqnarray*}
	which shows that $\sup_{h\in H_n}\frac{|\IE S_{1,n}|}{d^{*}(\hat \theta_h,\theta_0)}\to 0$.
	
	By Lemma \ref{hilfslemma1}, we obtain similarly to \reff{proof_lemma_approxmart_eq3},
	\begin{eqnarray*}
		&& \|S_{1,n} - \IE S_{1,n}\|_q\\
		&\le& C^{\#}\rho_{2,\psi C_{\delta,\cdot},q} (n(nh))^{-1}\\
		&&\quad\quad \times \Big(\sum_{s,t=1}^{n}K\Big(\frac{t-s}{nh}\Big)^2|\IE \nabla_j L_{n,h,-t}(t/n,\theta_0(t/n))|^2 w_{n,h}(t/n)^2\Big)^{1/2}\\
		&\le& C^{\#}\rho_{2,\psi C_{\delta,\cdot},q} |w|_{\infty}C_{\infty,1}(n(nh))^{-1}\Big(\sum_{s,t=1}^{n}K\Big(\frac{t-s}{nh}\Big)^2\Big)^{1/2}\\
		&=& O\big((n(nh))^{-1/2}\big),
	\end{eqnarray*}
	which shows that $\sup_{h\in H_n}\frac{|S_{1,n} - \IE S_{1,n}|}{d^{*}(\hat \theta_h,\theta_0)}\to 0$ a.s. By Lemma \ref{hilfslemma1}, \reff{hilfslemma1_eq5},
	\begin{eqnarray*}
		|\IE S_{3,n}| &\le& \frac{C^{\#}}{n(nh)^2}\sup_{k,l\in\IZ}\sum_{1 \le t \le n}\big|K\big(\frac{k}{nh}\big)\big|\cdot \big|K\big(\frac{l}{nh}\big)\big| \cdot |w_{n,h}(t/n)|\\
		&\le& \frac{C^{\#} |K|_{\infty}^2 |w|_{\infty}}{(nh)^2},
	\end{eqnarray*}
	thus $\sup_{h\in H_n}\frac{| \IE S_{3,n}|}{d^{*}(\hat \theta_h,\theta_0)}\to 0$. Define
	\begin{eqnarray*}
		S_{4,n} &:=& \frac{1}{n}\sum_{t=1}^{n} \IE|(\IE V_{t,n})^{-1}\big\{\nabla \hat L_{n,h,-t}(t/n,\theta_0(t/n))\\
		&&\quad\quad\quad\quad\quad\quad\quad\quad\quad - \IE \nabla \hat L_{n,h,-t}(t/n,\theta_0(t/n))\big\}|_{A}^2\\
		&&\quad\quad\quad\times w_{n,h}(t/n)\big|_{A = \nabla^2 \l(\tilde Y_t(t/n), \theta_0(t/n)) - V(\theta_0(t/n))},
	\end{eqnarray*}
	then by Lemma \ref{hilfslemma1} and inequality \reff{likelihood_rate} it holds that
	\begin{eqnarray*}
		&& \|S_{4,n}\|_q\\
		&\le& C^{\#}\rho_{1,\psi C_{\delta,\cdot},q}|w|_{\infty} n^{-1/2} \sup_{i,j = 1,...,d}\Big(\frac{1}{n}\sum_{t=1}^{n}|\Cov(\nabla_i \hat L_{n,h,-t}(t/n,\theta_0(t/n)),\\
		&&\quad\quad\quad\quad\quad\quad\quad\quad\quad\quad\quad\quad\quad\quad\quad\quad\quad\nabla_j \hat L_{n,h,-t}(t/n,\theta_0(t/n)))|^2\Big)^{1/2}\\
		&\le& \rho_{1,\psi C_{\delta,\cdot},q}^3 C^{\#} |w|_{\infty}|K|_{\infty}^2 n^{-1/2}(nh)^{-1}.
	\end{eqnarray*}
	Define
	\begin{eqnarray*}
		S_{5,n} &:=& \frac{1}{n}\sum_{t=1}^{n} \IE \big\langle v, \big(\nabla^2 \l(\tilde Y_t(t/n), \theta_0(t/n)) - V(\theta_0(t/n))\big)\\
		&& \times (\IE V_{t,n})^{-1}\big\{\nabla \hat L_{n,h,-t}(t/n,\theta_0(t/n)) - \IE \nabla \hat L_{n,h,-t}(t/n,\theta_0(t/n))\big\}\big\rangle\\
		&& \times w_{n,h}(t/n)\big|_{v = (\IE V_{t,n})^{-1}\{\nabla \hat L_{n,h,-t}(t/n,\theta_0(t/n)) - \IE \nabla \hat L_{n,h,-t}(t/n,\theta_0(t/n))\}},
	\end{eqnarray*}
	then by Lemma \ref{hilfslemma1} and inequality \reff{likelihood_rate} it holds that
	\begin{eqnarray*}
		&& \|S_{5,n}\|_q\\
		&\le& \rho_{1,\psi C_{\delta,\cdot},q} C^{\#} |w|_{\infty} n^{-1/2} \sup_{i,j = 1,...,d}\Big(\frac{1}{n}\sum_{t=1}^{n}|\Cov(\nabla_i \hat L_{n,h,-t}(t/n,\theta_0(t/n)),\\
		&&\quad\quad\quad\quad\quad\quad\quad\quad\quad\quad\quad\quad\quad\quad\quad\quad\quad\nabla_j \hat L_{n,h,-t}(t/n,\theta_0(t/n)))|^2\Big)^{1/2}\\
		&\le& \rho_{1,\psi C_{\delta,\cdot},q}^3 C^{\#} |w|_{\infty}|K|_{\infty}^2 n^{-1/2}(nh)^{-1}.
	\end{eqnarray*}
	Finally, by Lemma \ref{hilfslemma1}, \reff{hilfslemma1_eq3}, we have
	\begin{eqnarray*}
		&&\|S_{3,n} - S_{4,n} - 2S_{5,n} - \IE S_{3,n}\|_q\\
		&\le& \rho_{3,\psi C_{\delta,\cdot},q}C^{\#} |w|_{\infty}\frac{1}{n(nh)^2}\Big(\sum_{r,s,t=1}^{n}K\Big(\frac{s-t}{nh}\Big)^2 K\Big(\frac{r-t}{nh}\Big)^2\Big)^{1/2}\\
		&=& O( n^{-1/2}(nh)^{-1}).
	\end{eqnarray*}
	We now discuss \reff{lemma_approxmart_eq3}. Put $\hat V_{t,n} := \nabla^2 \hat L_{n,h,-t}(t/n,\theta_0(t/n))$. Similar arguments as in \reff{proof_lemma_approxmart_eq1} can be used to show that
	\begin{eqnarray*}
		T_{0,n}&:=&\textstyle \frac{1}{n}\sum_{t=1}^{n}w_{n,h}(t/n)\big\{\nabla \l(Y_{t,n}^c,\theta_0(t/n))(\IE V_{t,n})^{-1}(V_{t,n} - \IE V_{t,n})\\
		&&\quad\quad\quad\quad\quad\quad\quad\quad\quad\quad \times(\IE V_{t,n})^{-1}\nabla L_{n,h,-t}(t/n,\theta_0(t/n))\\
		&&\quad\quad\quad\quad\quad\quad\quad\quad - \nabla \l(\tilde Y_{t}(t/n),\theta_0(t/n))(\IE V_{t,n})^{-1}(\hat V_{t,n} - \IE \hat V_{t,n})\\
		&&\quad\quad\quad\quad\quad\quad\quad\quad\quad\quad \times(\IE V_{t,n})^{-1}\nabla \hat L_{n,h,-t}(t/n,\theta_0(t/n))\big\}.
	\end{eqnarray*}
	fulfills $\|T_{0,n}\|_q = O(n^{-1})$. Define
	\begin{eqnarray*}
		T_{1,n} &:=& \frac{1}{n}\sum_{t=1}^{n}\nabla \l(\tilde Y_{t}(t/n),\theta_0(t/n))(\IE V_{t,n})^{-1}(\hat V_{t,n} - \IE \hat V_{t,n}) (\IE V_{t,n})^{-1}\\
		&&\quad \times \big\{\nabla \hat L_{n,h,-t}(t/n,\theta_0(t/n)) - \IE L_{n,h,-t}(t/n,\theta_0(t/n))\big\}w_{n,h}(t/n).
	\end{eqnarray*}
	Put
	\begin{eqnarray*}
		T_{2,n} &:=& \frac{1}{n}\sum_{t=1}^{n}\nabla \l(\tilde Y_{t}(t/n),\theta_0(t/n))(\IE V_{t,n})^{-1}(\hat V_{t,n} - \IE \hat V_{t,n}) (\IE V_{t,n})^{-1}\\
		&&\quad\quad\quad\quad\quad\quad\times \nabla \hat L_{n,h,-t}(t/n,\theta_0(t/n)) w_{n,h}(t/n) - T_{1,n}.
	\end{eqnarray*}
	By Lemma \ref{hilfslemma1}, it holds that
	\begin{eqnarray*}
		&& \|T_{2,n}\|_q\\
		&\le& \frac{C^{\#}\rho_{2,\psi C_{\delta,\cdot},q}}{n(nh)}\\
		&&\times \sup_{i=1,...,d}\Big(\sum_{s,t=1}^{n}K\Big(\frac{s-t}{nh}\Big)^2 \big|\IE \nabla_i \hat L_{n,h,-t}(t/n,\theta_0(t/n))\big|^2 w_{n,h}(t/n)^2 \Big)^{1/2}\\
		&=& O\big((n(nh))^{-1/2}\big).
	\end{eqnarray*}
	Due to the similar structure, the argumentation for $T_{1,n}$ can be mimicked from $S_{3,n}$ above and leads to $\|T_{1,n}\|_q = O( (nh)^{-2} + (nh)^{-1}n^{-1/2})$.\\
	The stochastic structure of \reff{lemma_approxmart_eq4} is exactly the same as in \reff{lemma_approxmart_eq2}, thus the proof follows the same lines. Finally, it is easy to see from Lemma \ref{hilfslemma1} that
	\begin{eqnarray*}
		\big\|\frac{1}{n}\sum_{t=1}^{n}\IE_0[ |Y_{t,n}^c|_{\chi,1}^{M}] w_{n,h}(t/n) \sup_{i=1,...,p}|\IE \nabla_i \hat L_{n,h}(t/n, \theta_0(t/n))|\big\|_q = O(n^{-1/2}\hat B_h^{dis}),
	\end{eqnarray*} 
	which shows \reff{lemma_approxmart_eq5}.
\end{proof}

\section{Bounds for moments of sums, quadratic and cubic forms of Bernoulli shifts}

The inequalities derived in this section are needed to prove the moment inequalities for the local likelihoods and its derivatives in Section \ref{section_uniformconvergence}. Results for linear and quadratic forms were obtained in \cite{xiao2012}, however we need them in a more general setting.

\begin{lemma}\label{hilfslemma1}
Let $q \ge 2$. Let $\delta_q(k)$, $k \ge 0$ be some sequence of nonnegative real numbers. Put $\Delta_{k,q} := \sum_{l=k}^{\infty}\delta_q(l)$. Let $a_{t}, a_{s,t}, a_{r,s,t}$, $r,s,t = 1,...,n$, be some deterministic sequences of real numbers.
\begin{enumerate}
	\item[(i)] Let $V_{t,n}^{(1)}$, $t = 1,...,n$ be a triangular array with $\sup_{n\in\IN}\delta_{q}^{V_{\cdot,n}^{(1)}}(k) \le \delta_q(k)$ and $\IE V_{t,n}^{(1)}=0$. Put $\rho_{1,\delta,q} = q^{1/2}\Delta_{0,q}$. Then it holds that
	\begin{equation}
		\textstyle\big\|\sum_{t=1}^{n}a_{t} V_{t,n}^{(1)}\big\|_q \le \rho_{1,\delta,q} \big(\sum_{t=1}^{n}a_{t}^2\big)^{1/2}.\label{hilfslemma1_eq1}
	\end{equation}
	\item[(ii)] For $i = 1,2$, let $V_{t,n}^{(i)}(s)$, $s,t=1,...,n$ be triangular arrays with $\IE V_{t,n}^{(i)}(s)=0$ and $\sup_{n\in\IN}\sup_{s=1,...,n}\delta_{q}^{V_{\cdot,n}^{(i)}(s)}(k) \le \delta_q(k)$ as well as
	\begin{equation}
		\sup_{n\in\IN}\sum_{s=1}^{n-1}\delta_q^{V_{\cdot,n}^{(i)}(s+1) - V_{\cdot,n}^{(i)}(s)}(k) \le \delta_q(k).\label{hilfslemma_eq6}
	\end{equation}
	Put $\rho_{2,\delta,q} = 8q^{1/2} \Delta_{0,2q}\big[q^{1/2}\Delta_{0,2q} + \sum_{k=0}^{\infty}\Delta_{k,2q}\big]$ and $\tilde \rho_{2,\delta} := \Delta_{0,2}^2$. Then it holds that
	\begin{eqnarray}
		&& \textstyle\big\|\sum_{s,t=1}^{n}a_{s,t} \big\{V_{s,n}^{(2,1)}(t) V_{t,n}^{(2,2)}(s) - \IE V_{s,n}^{(2,1)}(t) V_{t,n}^{(2,2)}(s)\big\} \big\|_q\nonumber\\
		&\le& \textstyle\rho_{2,\delta,q} \big(\sum_{s,t=1}^{n}a_{s,t}^2\big)^{1/2},\label{hilfslemma1_eq2}\\
		&& \textstyle\big| \sum_{s,t=1}^{n}a_{s,t} \IE V_{s,n}^{(2,1)}(t)V_{t,n}^{(2,2)}(s)\big|\nonumber\\
		&\le& \textstyle\tilde \rho_{2,\delta} \sup_{k\in\IZ} \sum_{1 \le t,t+k \le n}|a_{t+k,t}|.\label{hilfslemma1_eq4}
	\end{eqnarray}
	\item[(iii)] For $i = 1,2,3$, let $V_{t,n}^{(i)}(r,s)$, $r,s,t=1,...,n$ be triangular arrays with $\IE V_{t,n}^{(i)}(r,s)=0$ and $\sup_{n\in\IN}\sup_{r,s=1,...,n}\delta_{q}^{V_{\cdot,n}^{(i)}(r,s)}(k) \le \delta_q(k)$. Assume that $V_{t,n}^{(i)}(r,\cdot)$ and $V_{t,n}^{(i)}(\cdot,s)$ fulfill \reff{hilfslemma_eq6} uniformly in $r,s$ as well as
	\begin{equation}
		\sup_{n\in\IN}\sum_{s,t=1}^{n-1}\delta_q^{V_{\cdot,n}^{(i)}(s+1,t+1) - V_{\cdot,n}^{(i)}(s,t+1)-V_{\cdot,n}^{(i)}(s+1,t)+V_{\cdot,n}^{(i)}(s,t)}(k) \le \delta_q(k).\label{hilfslemma1_eq6}
	\end{equation}
	Put
	\begin{eqnarray*}
		\rho_{3,\delta,q} &=& 240q^{1/2}\Delta_{0,3q}\cdot \big\{ q\Delta_{0,3q}^2 + q^{1/2} \Delta_{0,3q} \sum_{k=0}^{\infty}\Delta_{k,3q}\\
		&&\quad\quad\quad\quad\quad\quad\quad\quad + \sum_{i=0}^{\infty}\sum_{k\ge i}\Delta_{k,3q} + \big(\sum_{k=0}^{\infty}\Delta_{k,3q}\big)^2\big\}
	\end{eqnarray*}
	and $\tilde \rho_{3,\delta} := 12\Delta_{0,3}^2\big[\Delta_{0,3}+ \sum_{k=0}^{\infty}\Delta_{k,3}\big]$. Then it holds that
	\begin{eqnarray}
		&& \textstyle\big\|\sum_{r,s,t=1}^{n}a_{r,s,t} \big\{V_{r,n}^{(1)} V_{s,n}^{(2)} V_{t,n}^{(3)} - \IE V_{r,n}^{(1)} V_{s,n}^{(2)} V_{t,n}^{(3)}\nonumber\\
	&&\textstyle\quad  - V_{r,n}^{(1)}\IE V_{s,n}^{(2)}V_{t,n}^{(3)} - V_{s,n}^{(2)}\IE V_{r,n}^{(1)}V_{t,n}^{(3)} - V_{t,n}^{(3)}\IE V_{r,n}^{(1)}V_{s,n}^{(2)}\big\}\big\|_q\nonumber\\
	&\le& \rho_{3,\delta,q} \big(\sum_{r,s,t=1}^{n}a_{r,s,t}^2\big)^{1/2},\label{hilfslemma1_eq3}\\
		&& \textstyle\big| \sum_{s,t=1}^{n}a_{r,s,t} \IE V_{r,n}^{(1)}V_{s,n}^{(2)}V_{t,n}^{(3)}\big|\nonumber\\
		&\le& \textstyle\tilde \rho_{3,\delta} \sup_{k,l\in\IZ} \sum_{1 \le t,t+k,t+l \le n}|a_{t+k,t+l,t}|.\label{hilfslemma1_eq5}
	\end{eqnarray}
\end{enumerate}

\end{lemma}
\begin{proof}
	To keep the notation simple, set $a_{t}, a_{s,t}, a_{r,s,t}$ to $0$ if one of the indices $r,s,t$ is not in $\{1,...,n\}$. We start with proving the stochastic inequalities \reff{hilfslemma1_eq1}, \reff{hilfslemma1_eq2} and \reff{hilfslemma1_eq3}. Since $\IE V_{t,n}^{(1)} = 0$, it holds almost surely that $\sum_{t=1}^{n}a_{t} V_{t,n}^{(1)} = \sum_{k=0}^{\infty}\sum_{t=1}^{n}a_{t} P_{t-k}V_{t,n}^{(1)}$. $(a_{t}\cdot P_{t-k}V_{t,n}^{(1)})$, $t = 1,...,n$ is a martingale difference sequence, thus with \cite{rio2009}, Theorem 2.1 we obtain
	\begin{eqnarray*}
		\Big\|\sum_{t=1}^{n}a_t V_{t,n}^{(1)}\Big\|_q &\le& (q-1)^{1/2}\sum_{k=0}^{\infty}\Big(\sum_{t=1}^{n}a_{t}^2  \|P_{t-k}V_{t,n}^{(1)}\|_q^2\Big)^{1/2}\\
		&\le& (q-1)^{1/2}\sum_{k=0}^{\infty}\delta_q(k)\cdot \Big(\sum_{t=1}^{n}a_{t}^2\Big)^{1/2}.
	\end{eqnarray*}
	We now prove \reff{hilfslemma1_eq2}. Define $D_{s,t,k_1,k_2} := a_{s,t}\big\{P_{s-k_1}V_{s,n}^{(1)}(t)P_{t-k_2}V_{t,n}^{(2)}(s) - \IE P_{s-k_1}V_{s,n}^{(1)}(t) P_{t-k_2}V_{t,n}^{(2)}(s)\big\}$. Note that
	\begin{eqnarray*}
		&& \Big\|\sum_{s,t=1}^{n}a_{s,t} \big\{V_{s,n}^{(1)}(t) V_{t,n}^{(2)}(s) - \IE V_{s,n}^{(1)}(t) V_{t,n}^{(2)}(s)\big\} \Big\|_q\\
		&\le& \sum_{k_1,k_2 \ge 0}\Big\| \sum_{s,t=1}^{n}D_{s,t,k_1,k_2}\Big\|_q \le A_1 + A_{2,1} + A_{2,2},
	\end{eqnarray*}
	where
	\begin{eqnarray*}
		A_1 &:=& \sum_{k_1,k_2 \ge 0}\big\| \sum_{s=1}^{n}D_{s,s-k_1 + k_2,k_1,k_2}\big\|_q,\\
		A_{2,1} &:=& \sum_{k_1,k_2 \ge 0}\big\| \sum_{t=1}^{n}\sum_{s < t-k_2 + k_1}D_{s,t,k_1,k_2}\big\|_q,\\
		A_{2,2} &:=& \sum_{k_1,k_2 \ge 0}\big\| \sum_{s=1}^{n}\sum_{t < s-k_1 + k_2}D_{s,t,k_1,k_2}\big\|_q.
	\end{eqnarray*}
	Here $(\sum_{s < t-k_2 + k_1}D_{s,t,k_1,k_2})_{t = 1,...,n}$ as well as $(a_{s,t} P_{s-k_1}V_{s,n}^{(1)})_s$ are martingale difference sequences. By applying Theorem 2.1 in \cite{rio2009},
	\[
		A_{2,1} \le (q-1)^{1/2}\sum_{k_1,k_2 \ge 0} \Big(\sum_{t=1}^{n}\Big\|\sum_{s < t-k_2 + k_1}D_{s,t,k_1,k_2}\Big\|_q^2\Big)^{1/2}.
	\]
	Define $\tilde a_{s,t} := a_{s,t}\Ii_{\{s < t-k_2+k_1\}}$. By partial summation and the Cauchy-Schwarz inequality, we have
	\begin{eqnarray*}
		&& \Big\|\sum_{s < t-k_2 + k_1}D_{s,t,k_1,k_2}\Big\|_q\\
		&\le& \sum_{s=1}^{n}\Big\| \sum_{s'=1}^{s}\tilde a_{s',t} P_{s'-k_1}V_{s',n}^{(1)}(t)\Big\|_{2q}\cdot\big\|P_{t-k_2}V_{t,n}^{(2)}(s+1) - P_{t-k_2}V_{t,n}^{(2)}(s)\big\|_{2q}\\
		&&\quad\quad + \Big\| \sum_{s'=1}^{n}\tilde a_{s',t} P_{s'-k_1}V_{s',n}^{(1)}(t)\Big\|_{2q} \cdot \big\|P_{t-k_2}V_{t,n}^{(2)}(n)\big\|_{2q}\\
		&\le& 2 (2q-1)^{1/2}\Big(\sum_{s=1}^{n}a_{s,t}^2\Big)^{1/2} \delta_{2q}(k_1) \delta_{2q}(k_2).
	\end{eqnarray*}
	We finally obtain $A_{2,1} \le 2(q-1)^{1/2}(2q-1)^{1/2} \Delta_{0,2q}^2 \cdot \big(\sum_{s,t=1}^{n}a_{s,t}^2\big)^{1/2}$, 
	a similar upper bound holds for $A_{2,2}$. To discuss $A_1$, note that
	\[
		(a_{s,s-k_1+k_2}P_{s-k_1-i}[ P_{s-k_1}V_{s,n}^{(1)}(s-k_1+k_2) P_{s-k_1}V_{s-k_1+k_2,n}^{(2)}(s)])_s
	\]
	is a martingale difference sequence. The arguments $(s-k_1+k_2)$ of $V_{s,n}^{(1)}$ and $(s)$ of $V_{s-k_1+k_2,n}^{(2)}$ are omitted in the next steps. In the following, fix some $s,l \ge 0$. Let $\varepsilon_{s-l}^{*}$ be an i.i.d. copy of $\varepsilon_{s-l}$. Let $\sF_s^{*(s-l)} := (\varepsilon_s,...,\varepsilon_{s-l+1},\varepsilon_{s-l}^{*},\varepsilon_{s-l-1},....)$. For fixed $k_1, i \ge 0$ and some random variable $V_s = H_z(\sF_s)$, we define $V_s^{*} := H_z(\sF_{s}^{*(s-k_1)})$ and $V_s^{**} := H_z(\sF_{s}^{*(s-k_1-i)})$. Similarly, $\sF_{s-k_1}^{*} := \sF_{s-k_1}^{*(s-k_1)}$ and $\sF_{s-k_1}^{**} := \sF_{s-k_1}^{*(s-k_1-i)}$. We use the abbreviation $V^{*,**} := (V^{*})^{**}$. It holds that
	\begin{eqnarray*}
		&& \|P_{s-k_1-i}[ P_{s-k_1}V_{s,n}^{(1)} P_{s-k_1}V_{s-k_1+k_2,n}^{(2)}]\|_q\\
		&\le& \big\|P_{s-k_1}V_{s,n}^{(1)} P_{s-k_1}V_{s-k_1+k_2,n}^{(2)} - \big[P_{s-k_1}V_{s,n}^{(1)} P_{s-k_1}V_{s-k_1+k_2,n}^{(2)}\big]^{**}\big\|_q\\
		&\le& \big\|P_{s-k_1}V_{s,n}^{(1)} - \big(P_{s-k_1}V_{s,n}^{(1)}\big)^{**}\big\|_{2q}\|P_{s-k_1}V_{s-k_1+k_2,n}^{(2)}\|_{2q}\\
		&&\quad + \big\|\big(P_{s-k_1}V_{s,n}^{(1)}\big)^{**}\big\|_{2q} \|P_{s-k_1}V_{s-k_1+k_2,n}^{(2)} - \big(P_{s-k_1}V_{s-k_1+k_2,n}^{(2)}\big)^{**}\|_{2q}. 
	\end{eqnarray*}
	Note that $\|P_{s-k_1}V_{s-k_1+k_2,n}^{(2)}\|_{2q} \le \delta_{2q}(k_2)$ and $\| (P_{s-k_1} V_{s,n}^{(1)})^{**}\|_{2q} \le \delta_{2q}(k_1)$. Furthermore,
	\begin{eqnarray*}
		&& P_{s-k_1}V_{s,n}^{(1)} - \big(P_{s-k_1}V_{s,n}^{(1)}\big)^{**}\\
		&=& \big\{\IE\big[\IE[V_{s,n}^{(1)}|\sF_{s-k_1}] - \IE[\big(V_{s,n}^{(1)}\big)^{*}|\sF_{s-k_1}^{*}]\big|\sF_{s-k_1}\big]\big\}\\
		&&\quad\quad\quad - \big\{\IE\big[\IE[(V_{s,n}^{(1)})^{**}|\sF_{s-k_1}^{**}] - \IE[\big(V_{s,n}^{(1)}\big)^{*,**}|\sF_{s-k_1}^{*,**}]\big|\sF_{s-k_1}^{**}\big]\big\}\\
		&=& \IE[V_{s,n}^{(1)} - \big(V_{s,n}^{(1)}\big)^{**}|\sF_{s-k_1}, \varepsilon_{s-k_1-i}^{*}]\\
		&&\quad\quad\quad - \IE\big[\IE[\big(V_{s,n}^{(1)}\big)^{*} - \big(V_{s,n}^{(1)}\big)^{*,**}|\sF_{s-k_1}^{*}, \varepsilon_{s-k_1-i}^{*} \big]\big|\sF_{s-k_1}, \varepsilon_{s-k_1-i}^{*}\big]
	\end{eqnarray*}
	and thus, by Jensen's inequality,
	\begin{equation}
		\big\|P_{s-k_1}V_{s,n}^{(1)} - \big(P_{s-k_1}V_{s,n}^{(1)}\big)^{**}\big\|_{2q} \le 2\big\|V_{s,n}^{(1)} - \big(V_{s,n}^{(1)}\big)^{**}\big\|_{2q} \le 2 \delta_{2q}(k_1 + i).\label{CV_proof_hilfslemma_eq1}
	\end{equation}
	Note that in \reff{CV_proof_hilfslemma_eq1}, there may be better estimations in special cases. We obtain
	\begin{eqnarray*}
		A_1 &\le& (q-1)^{1/2}\sum_{k_1,k_2 \ge 0}\sum_{i \ge 0}\Big(\sum_{s=1}^{n}a_{s,s-k_1+k_2}^2\\
		&&\quad\quad\quad\quad\quad\quad\quad\quad \times \big\|P_{s-k_1-i}[ P_{s-k_1}V_{s,n}^{(1)} P_{s-k_1}V_{s-k_1+k_2,n}^{(2)}]\big\|_q^2\Big)^{1/2}\\
		&\le&  2(q-1)^{1/2}\sum_{k_1,k_2,i \ge 0}\big[\delta_{2q}(k_1)\delta_{2q}(k_2+i) + \delta_{2q}(k_2) \delta_{2q}(k_1+i)\big]\\
		&&\quad\quad\quad\quad\quad\quad\quad\quad \times \Big(\sum_{s=1}^{n}a_{s,s-k_1+k_2}^2\Big)^{1/2}\\
		&\le& 4(q-1)^{1/2}\Delta_{0,2q}\cdot \sum_{k=0}^{\infty}\Delta_{k,2q}\cdot \Big(\sum_{s,t=1}^{n}a_{s,t}^2\Big)^{1/2}.
	\end{eqnarray*}
	We now prove \reff{hilfslemma1_eq3}. To do so, define
	\begin{eqnarray*}
		&& D_{r,s,t,k_1,k_2,k_3}\\
		&:=& a_{r,s,t}\big\{P_{r-k_1}V_{r,n}^{(1)}P_{s-k_2}V_{s,n}^{(2)} P_{t-k_3}V_{t,n}^{(3)} - \IE[P_{r-k_1}V_{r,n}^{(1)}P_{s-k_2}V_{s,n}^{(2)} P_{t-k_3}V_{t,n}^{(3)}]\\
		&&\quad - P_{r-k_1}V_{r,n}^{(1)}\IE P_{s-k_2}V_{s,n}^{(2)}P_{t-k_3}V_{t,n}^{(3)} - P_{s-k_2}V_{s,n}^{(2)}\IE P_{r-k_1}V_{r,n}^{(1)}P_{t-k_3}V_{t,n}^{(3)}\\
		&&\quad - P_{t-k_3}V_{t,n}^{(3)}\IE P_{r-k_1}V_{r,n}^{(1)}P_{s-k_3}V_{s,n}^{(2)}\big\}.
	\end{eqnarray*}
	Here, we can bound $\sum_{k_1,k_2,k_3 \ge 0}\big\| \sum_{r,s,t=1}^{n}D_{r,s,t,k_1,k_2,k_3}\big\|_q$ by four different types of terms. The first type (all indices are different) is of the form
	\[
		A_3 := \sum_{k_1,k_2,k_3 \ge 0}\Big\| \sum_{r=1}^{n}\sum_{s: s-k_2 < r-k_1}\sum_{t: t-k_3 < s-k_2} D_{r,s,t,k_1,k_2,k_3}\Big\|_q. 
	\]
	Note that the three sequences $\big(\sum_{s:s-k_2<r-k_1}\sum_{t:t-k_3 < s-k_2}D_{r,s,t,k_1,k_2,k_3}\big)_r$ and $\big(\sum_{t:t-k_3 < s-k_2}a_{r,s,t}P_{s-k_2}V_{s,n}^{(2)}P_{t-k_3}V_{t,n}^{(3)}\big)_s$ and $(a_{r,s,t}P_{t-k_3}V_{t,n}^{(3)})_t$ are martingale differences. Put $\tilde a_{r,s,t} := a_{r,s,t}\Ii_{\{t-k_3 < s-k_2 < r-k_1\}}$. By partial summation, we have
		\begin{eqnarray*}
			&&\textstyle\big|\sum_{s,t=1}^{n}\tilde a_{r,s,t}P_{r-k_1}V_{r,n}^{(1)}(s,t)P_{s-k_2}V_{s,n}^{(2)}(r,t) P_{t-k_3}V_{t,n}^{(3)}(r,s)\big|\\
			&\le& \textstyle\big\{\sum_{s,t=1}^{n-1}\big|P_{r-k_1}V_{t,n}^{(1)}(s+1,t+1)-P_{r-k_1}V_{t,n}^{(1)}(s,t+1)\\
			&&\quad\quad\quad - P_{r-k_1}V_{r,n}^{(1)}(s+1,t) + P_{r-k_1}V_{r,n}^{(1)}(s,t)\big|\\
			&&\quad + \textstyle\sum_{s=1}^{n-1}\big|P_{r-k_1}V_{r,n}^{(1)}(n,s+1) - P_{r-k_1}V_{r,n}^{(1)}(n,s)\big|\\
			&&\quad + \textstyle\sum_{t=1}^{n-1}\big|P_{r-k_1}V_{r,n}^{(1)}(t+1,n) - P_{r-k_1}V_{r,n}^{(1)}(t,n)\big| + |P_{r-k_1}V_{r,n}^{(1)}(n,n)|\big\}\\
			&&\quad\quad\times \textstyle\big|\sum_{s'=1}^{s}\sum_{t'=1}^{t}\tilde a_{r,s',t'}P_{s'-k_2}V_{s',n}^{(2)}(r,t') P_{t'-k_3}V_{t',n}^{(3)}(r,s')\big|,
	\end{eqnarray*}
	leading with Theorem 2.1. in \cite{rio2009} and Hoelder's inequality to the upper bound
	\begin{eqnarray*}
		&& 4(q-1)^{1/2}\sum_{k_1,k_2,k_3 \ge 0}\delta_{3q}(k_1)\\
		&&\times \Big(\sum_{r=1}^{n}\sup_{s,t=1,...,n}\Big\| \sum_{s'=1}^{s}\sum_{t'=1}^{t}\tilde a_{r,s',t'}P_{s'-k_2}V_{s',n}^{(2)}(r,t') P_{t'-k_3}V_{t',n}^{(3)}(r,s')\Big\|_{3q/2}^2\Big)^{1/2}.
	\end{eqnarray*}
	for $A_3$. Using the same partial summation argument as in the discussion of $A_{2,1}$ above, we obtain
	\[
		A_3 \le  8(q-1)^{1/2}\big(\frac{3}{2}q-1\big)^{1/2} (3q-1)^{1/2} \Delta_{0,3q}^{3}\big(\sum_{r,s,t = 1}^{n}a_{r,s,t}^2\big)^{1/2}.
	\]
	The second type (the two smaller indices are equal) is of the form
	\begin{eqnarray*}
		A_4 &:=& \sum_{k_1,k_2,k_3 \ge 0}\Big\| \sum_{r=1}^{n}\sum_{s:s-k_2 < r-k_1}D_{r,s,s-k_2+k_3,k_1,k_2,k_3}\Big\|_q\\
		&\le& (q-1)^{1/2}\sum_{k_1,k_2,k_3 \ge 0}\Big(\sum_{r=1}^{n}\Big\|\sum_{s:s-k_2 < r-k_1}a_{r,s,s-k_2+k_3} P_{r-k_1}V_{r,n}^{(1)}(s,t)\\
		&&\quad\quad \times \big[P_{s-k_2}V_{s,n}^{(2)}(r,s-k_2+k_3)P_{s-k_2}V_{s-k_2+k_3,n}^{(3)}(r,s)\\
		&&\quad\quad\quad\quad - \IE P_{s-k_2}V_{s,n}^{(2)}(r,s-k_2+k_3) P_{s-k_2}V_{s-k_2+k_3,n}^{(3)}(r,s)\big]\Big\|_{q}^2\Big)^{1/2}. 
	\end{eqnarray*}
	Put $\hat a_{r,s,t} := a_{r,s,t}\Ii_{\{s-k_2<r-k_1\}}$By applying similar partial summation techniques as for $A_{2,1}$, we obtain the upper bound
	\begin{eqnarray*}
		 &\le& 2(q-1)^{1/2}\sum_{k_1,k_2,k_3,i \ge 0}\delta_{3q}(k_1)\cdot \Big(\sum_{r=1}^{n}\sup_{s=1,...,n}\Big\|\sum_{s'=1}^{s}\hat a_{r,s',s'-k_2+k_3}\\
		&& \times P_{s'-k_2-i}\big[P_{s'-k_2}V_{s',n}^{(2)}(r,s'-k_2+k_3)P_{s'-k_2}V_{s'-k_2+k_3,n}^{(3)}(r,s')\big]\Big\|_{3q/2}^2\Big)^{1/2}.
	\end{eqnarray*}
	for $A_4$. The same technique as applied in $A_1$ leads to the bound
	\[
		A_4 \le 8(q-1)^{1/2}\big(\frac{3q}{2}-1\big)^{1/2} \Delta_{0,3q}^2 \sum_{k=0}^{\infty}\Delta_{k,3q}\cdot \big(\sum_{r,s,t=1}^{n}a_{r,s,t}^2\big)^{1/2}.
	\]	
	Note that $(P_{r-k_1-i}D_{r,r-k_1+k_2,r-k_1+k_3,k_1,k_2,k_3})_r$ is a martingale difference sequence. For brevity, we will omit the additional arguments $s,t$ of $V_{r,n}^{(i)}(s,t)$ in the following part. The third type (all three indices are equal) is of the form
	\begin{eqnarray*}
		A_5 &:=& \sum_{k_1,k_2,k_3 \ge 0}\Big\| \sum_{r=1}^{n}D_{r,r-k_1+k_2,r-k_1+k_3,k_1,k_2,k_3}\Big\|_q \le A_{5,1} + A_{5,2}, 
	\end{eqnarray*}
	where
	\begin{eqnarray*}
		&& A_{5,1}\\
		&\le& \textstyle (q-1)^{1/2}\sum_{k_1,k_2,k_3 \ge 0}\Big\{ \Big(\sum_{r=1}^{n} a_{r,r-k_1+k_2,r-k_1+k_3}^2\\
		&&\quad\quad\textstyle \times|\IE P_{r-k_1}V_{r,n}^{(1)} P_{r-k_1}V_{r-k_1+k_2,n}^{(2)}|^2 \cdot \|P_{r-k_1}V_{r-k_1+k_3,n}^{(3)}\|_q^{2}\Big)^{1/2}\\
		&& + \textstyle\Big(\sum_{r=1}^{n} a_{r,r-k_1+k_2,r-k_1+k_3}^2\\
		&&\quad\quad\textstyle \times |\IE P_{r-k_1}V_{r-k_1+k_2,n}^{(2)} P_{r-k_1}V_{r-k_1+k_3,n}^{(3)}|^2 \cdot \|P_{r-k_1}V_{r,n}^{(1)}\|_q^{1}\Big)^{1/2}\\
		&& \textstyle +\Big(\sum_{r=1}^{n} a_{r,r-k_1+k_2,r-k_1+k_3}^2\\
		&&\quad\quad\textstyle \times |\IE P_{r-k_1}V_{r,n}^{(1)} P_{r-k_1}V_{r-k_1+k_3,n}^{(3)}|^2 \cdot \|P_{r-k_1}V_{r-k_1+k_2,n}^{(2)}\|_q^{2}\Big)^{1/2}\Big\}\\
		&\le& \textstyle(q-1)^{1/2}\sum_{k_1,k_2,k_3 \ge 0}\big[\delta_2(k_1) \delta_2(k_2) \delta_q(k_3) + \delta_2(k_2) \delta_2(k_3) \delta_q(k_1)\\
		&&\quad\quad\quad\quad\quad\quad\quad\quad\quad\quad\quad\quad + \delta_2(k_1) \delta_2(k_3) \delta_q(k_3)\big]\cdot \Big(\sum_{r,s,t=1}^{n}a_{r,s,t}^2\Big)^{1/2}\\
		&\le& \textstyle 3(q-1)^{1/2} \Delta_{0,q}^3\cdot \Big(\sum_{r,s,t=1}^{n}a_{r,s,t}^2\Big)^{1/2},
	\end{eqnarray*}
	and
	\begin{eqnarray*}
		A_{5,2} &\le& \textstyle (q-1)^{1/2}\sum_{k_1,k_2,k_3,i \ge 0}\Big(\sum_{r=1}^{n} a_{r,r-k_1+k_2,r-k_1+k_3}^2\\
		&&\quad \textstyle\times\Big\|P_{r-k_1-i}\big[P_{r-k_1}V_{r,n}^{(1)} P_{r-k_1}V_{r-k_1+k_2,n}^{(2)}P_{r-k_1}V_{r-k_1+k_3,n}^{(3)}\big]\Big\|_q^2\Big)^{1/2}.
	\end{eqnarray*}
	Similarly to \reff{CV_proof_hilfslemma_eq1}, we obtain
	\begin{eqnarray*}
		&& \big\|P_{r-k_1-i}\big[P_{r-k_1}V_{r,n}^{(1)} P_{r-k_1}V_{r-k_1+k_2,n}^{(2)}P_{r-k_1}V_{r-k_1+k_3,n}^{(3)}\big]\big\|_q\\
		&\le& 2\delta_{3q}(k_1)\delta_{3q}(k_2) \delta_{3q}(k_3+i) + 2\delta_{3q}(k_1)\delta_{3q}(k_3) \delta_{3q}(k_2+i)\\
		&&\quad\quad\quad + 2\delta_{3q}(k_2)\delta_{3q}(k_3) \delta_{3q}(k_1+i),
	\end{eqnarray*}
	which leads to
	\[
		A_{5,2} \le 6(q-1)^{1/2}\Delta_{0,3q}^2 \sum_{k=0}^{\infty}\Delta_{k,3q}\cdot \Big(\sum_{r,s,t=1}^{n}a_{r,s,t}^2\Big)^{1/2}.
	\]
	The fourth type (the two bigger indices are equal) has the form
	\begin{eqnarray*}
		A_6 &:=& \sum_{k_1,k_2,k_3 \ge 0}\Big\| \sum_{r=1}^{n}\sum_{t:t-k_3 < r-k_1}D_{r,r-k_1+k_2,t-k_3,k_1,k_2,k_3}\Big\|_q.
	\end{eqnarray*}
	$A_6$ is bounded by the sum of three terms $A_{6,1} + A_{6,2} + A_{6,3}$, which will be defined in the following. Put $a_{r,s,t}^{\circ} := a_{r,s,t}\Ii_{\{t-k_3 < r-k_1-i\}}$. For brevity in the following argumentation, put $\tilde r := r-k_1+k_2$. By using similar techniques as for $A_1$, we obtain
	\begin{eqnarray*}
		&&\big\|P_{r-k_1-i}\big[P_{r-k_1}V_{r,n}^{(1)}(\tilde r,t+1)P_{r-k_1}V_{\tilde r,n}^{(2)}(r,t+1)\\
		&&\quad\quad\quad\quad\quad\quad\quad\quad - P_{r-k_1}V_{r,n}^{(1)}(\tilde r,t)P_{r-k_1}V_{\tilde r,n}^{(2)}(r,t)\big]\big\|_{3q/2}\\
		&\le& \big\|P_{r-k_1-i}\big[P_{r-k_1}\big\{V_{r,n}^{(1)}(\tilde r,t+1) - V_{r,n}^{(1)}(\tilde r,t)\big\}P_{r-k_1}V_{\tilde r,n}^{(2)}(r,t+1)\big\|_{3q/2}\\
		&&\quad\quad + \big\|P_{r-k_1-i}\big[P_{r-k_1}V_{r,n}^{(1)}(\tilde r,t)P_{r-k_1}\big\{V_{\tilde r,n}^{(2)}(r,t+1)-V_{\tilde r,n}^{(2)}(r,t)\big\}\big\|_{3q/2}\\
		&\le& 2 \delta_{3q}^{V_{\cdot,n}^{(1)}(r,t+1)-V_{\cdot,n}^{(1)}(r,t)}(k_1+i)\cdot \delta_{3q}(k_2)\\
		&&\quad\quad\quad + 2 \delta_{3q}(k_1) \delta_{3q}^{V_{\cdot,n}^{(2)}(r,t+1) - V_{\cdot,n}^{(2)}(r,t)}(k_2+i).
	\end{eqnarray*}
	Applying the same partial summation techniques as for the term $A_{2,1}$, we conclude
	\begin{eqnarray*}
		&& A_{6,1}\\
		&:=& \textstyle\sum_{k_1,k_2,k_3,i \ge 0}\Big\| \sum_{r=1}^{n}\sum_{t:t-k_3 < r-k_1-i}a_{r,r-k_1+k_2,t}P_{t-k_3}V_{t,n}^{(3)}(r,r-k_1+k_2)\\
		&&\quad\quad\quad\quad \times P_{r-k_1-i}\big[P_{r-k_1}V_{r,n}^{(1)}(r-k_1+k_2,t) P_{r-k_1}V_{r-k_1+k_2,n}^{(2)}(r,t)\big]\Big\|_q\\
		&\le& \textstyle 4(q-1)^{1/2}\cdot \sum_{k_1,k_2,k_3,i \ge 0}\big[\delta_{3q}(k_1+i)\delta_{3q}(k_2) + \delta_{3q}(k_1)\delta_{3q}(k_2+i)\big]\\
		&&\textstyle\quad\quad\times \Big(\sum_{r=1}^{n} \sup_{t=1,...,n}\Big\|\sum_{t'=1}^{t}a_{r,r-k_1+k_2,t}^{\circ}\\
		&&\quad\quad\quad\quad\quad\quad\quad\quad\quad\quad\quad\quad\quad\quad\times P_{t-k_3}V_{t,n}^{(3)}(r,r-k_1+k_2)\Big\|_{3q}^2\Big)^{1/2}\\
		&\le& \textstyle 4(q-1)^{1/2}\cdot (3q-1)^{1/2}\\
		&& \textstyle\quad\quad\quad \times \sum_{k_1,k_2,k_3,i \ge 0}\big[\delta_{3q}(k_1)\delta_{3q}(k_2+i) + \delta_{3q}(k_2)\delta_{3q}(k_1+i)\big]\\
		&&\quad\quad\quad\quad\quad\quad\quad\quad\quad\quad\textstyle \times\delta_{3q}(k_3)\cdot \Big(\sum_{rs,t=1}^{n}a_{r,s,t}^2\Big)^{1/2}\\
		&\le& \textstyle 8(q-1)^{1/2}\cdot \big(3q-1\big)^{1/2}\Delta_{0,3q}^2 \cdot \sum_{k=0}^{\infty}\Delta_{k,3q}\cdot \Big(\sum_{rs,t=1}^{n}a_{r,s,t}^2\Big)^{1/2}.
	\end{eqnarray*}
	With slight changes in the argumentation, the second term has the upper bound
	\begin{eqnarray*}
		A_{6,2} &:=& \sum_{k_1,k_2,k_3,i \ge 0}\Big\| \sum_{t=1}^{n}\sum_{r: r-k_1-i < t-k_3}a_{r,r-k_1+k_2,t}P_{t-k_3}V_{t,n}^{(3)}(r,r-k_1+k_2)\\
		&&\quad\quad \times P_{r-k_1-i}\big[P_{r-k_1}V_{r,n}^{(1)}(r-k_1+k_2,t) P_{r-k_1}V_{r-k_1+k_2,n}^{(2)}(r,t)\big]\Big\|_q\\
		&\le& 4(q-1)^{1/2}\cdot \big(\frac{3q}{2}-1\big)^{1/2}\Delta_{0,3q}^2 \cdot \sum_{k=0}^{\infty}\Delta_{k,3q}\cdot \Big(\sum_{rs,t=1}^{n}a_{r,s,t}^2\Big)^{1/2}.
	\end{eqnarray*}
	In the following, we will again omit the arguments $s,t$ of $V_{r,n}^{(i)}(s,t)$. The third term reads
	\begin{eqnarray*}
		&& A_{6,3}\\
		&=& \sum_{k_1,k_2,k_3,i \ge 0}\Big\|\sum_{r=1}^{n}a_{r,r-k_1+k_2,r-k_1+k_3-i}\\
		&&\quad \times P_{r-k_1-i}\big[P_{r-k_1}V_{r,n}^{(1)} P_{r-k_1}V_{r-k_1+k_2,n}^{(2)}\big]\cdot P_{r-k_1-i}V_{r-k_1+k_3-i,n}^{(3)} - \IE[...]\Big\|_q\\
		&\le& (q-1)^{1/2}\sum_{k_1,k_2,k_3,i,j\ge 0}\Big(\sum_{r=1}^{n}a_{r,r-k_1+k_2,r-k_1+k_3-i}^2\\
		&&\quad \times \Big\|P_{r-k_1-i-j}\big\{P_{r-k_1-i}\big[P_{r-k_1}V_{r,n}^{(1)} P_{r-k_1}V_{r-k_1+k_2,n}^{(2)}\big]\\
		&&\quad\quad\quad\quad\quad\quad\quad\quad\quad\quad\quad\quad\times P_{r-k_1-i}V_{r-k_1+k_3-i,n}^{(3)}\big\}\Big\|_q^2\Big)^{1/2}
	\end{eqnarray*}
	Put $W_1 := P_{r-k_1}V_{r,n}^{(1)}$, $W_2 := P_{r-k_1}V_{r-k_1+k_2,n}^{(2)}$, $Z := P_{r-k_1-i}V_{r-k_1+k_3,n}^{(3)}$ and $W := P_{r-k_1-i}[W_1 \cdot W_2]$. Using similar techniques as in $A_1$, we obtain
	\begin{eqnarray*}
		&& \big\| P_{r-k_1-i-j}[W\cdot Z]\big\|_q\\
		&\le& \big\|WZ - (WZ)^{*(r-k_1-i-j)}\big\|_q\\
		&\le& \|W - W^{*(r-k_1-i-j)}\|_{3q/2}\|Z\|_{3q} + \|Z - Z^{*}(r-k_1-i-j)\|_{3q}\|W^{*}\|_{3q/2}\\
		&\le& 4\big[\delta_{3q}(k_1+i+j)\delta_{3q}(k_2) + \delta_{3q}(k_2+i+j)\delta_{3q}(k_1)\big]\delta_{3q}(k_3)\\
		&&\quad\quad\quad + 4\delta(k_3+j)\big[\delta_{3q}(k_1)\delta_{3q}(k_2+i) +\delta_{3q}(k_2)\delta_{3q}(k_1+i)\big],
	\end{eqnarray*}
	which shows
	\[
		A_{6,3} \le 8(q-1)^{1/2}\Big[\Delta_{0,3q}^2 \sum_{k \ge 0}\sum_{i \ge k}\Delta_{i,3q}  + \Big(\sum_{k=0}^{\infty}\Delta_{k,3q}\Big)^2 \Delta_{0,3q}\Big]\cdot \Big(\sum_{r,s,t}a_{r,s,t}^2\Big)^{1/2},
	\]
	and finishes the proof of \reff{hilfslemma1_eq3}.
	
	To prove \reff{hilfslemma1_eq4} and \reff{hilfslemma1_eq5}, we will omit the arguments $s,t$ of $V_{r,n}^{(i)}(s,t)$ since all bounds are uniform in these arguments. For \reff{hilfslemma1_eq4}, we use the inequalities
	\begin{eqnarray*}
		&& \textstyle\Big| \sum_{s,t=1}^{n}a_{s,t}\IE V_{s,n}^{(1)}(t)V_{t,n}^{(2)}\Big|\\
		&\le& \textstyle\sum_{k,l=0}^{\infty}\Big|\sum_{s,t=1}^{n}a_{s,t}\IE P_{s-k}V_{s,n}^{(1)}\cdot P_{t-l}V_{t,n}^{(2)}\Big|\\
		&\le&\textstyle \sum_{k,l=0}^{\infty}\sum_{1 \le t,t+k \le n}|a_{t+k-l,t}|\cdot \|P_{t-l}V_{t+k-l,n}^{(1)}\|_2 \cdot \|P_{t-l}V_{t,n}^{(2)}\|_2\\
		&\le& \textstyle\Delta_{0,2}^2 \sup_{k\in\IZ}\sum_{1 \le t,t+k \le n}^{n}|a_{t+k,t}|.
	\end{eqnarray*}
	To prove \reff{hilfslemma1_eq5}, note that
	\begin{eqnarray*}
		&& \textstyle\Big| \sum_{s,t=1}^{n}a_{r,s,t}\IE V_{r,n}^{(1)}V_{s,n}^{(2)}V_{t,n}^{(3)}\Big|\\
		&\le&  \sum_{i,j,k=0}^{\infty}\Big|\sum_{r,s,t=1}^{n}a_{r,s,t}\IE P_{r-i}V_{r,n}^{(1)} P_{s-j}V_{s,n}^{(2)} P_{t-k}V_{t,n}^{(3)}\Big|.
	\end{eqnarray*}
	The above term is bounded by two types of terms. The first type (all three indices $r-i,s-j,t-k$ are equal) is of the form
	\begin{eqnarray*}
		\tilde A_1 &:=& \textstyle\sum_{i,j,k=0}^{\infty}\sum_{1 \le t,t+i-k,t+j-k \le n}|a_{t+i-k,t+j-k,t}|\\
		&&\quad\quad\quad\times  \|P_{t-k}V_{t+i-k,n}^{(1)}\|_3 \|P_{t-k}V_{t+j-k,n}^{(2)}\|_3 \|P_{t-k}V_{t,n}^{(3)}\|_{3}\\
		&\le&\textstyle \Delta_{0,3}^3\sup_{k,l\in\IZ}\sum_{1 \le t,t+k,t+l \le n}|a_{t+k,t+l,t}|,
	\end{eqnarray*}
	the second type (the two bigger indices of $r-i,s-j,t-k$ are equal) is of the form
	\begin{eqnarray*}
		\tilde A_2 &:= & \textstyle\sum_{i,j,k,l=0}^{\infty}\sum_{1 \le t,t+i-k-l,t+j-k \le n} |a_{t+i-k-l,t+j-k,t}| \\
		&&\quad\quad\quad\times |\IE P_{t-k-l}V_{t+i-k-l,n}^{(1)} P_{t-k-l}[P_{t-k}V_{t+j-k,n}^{(2)} P_{t-k}V_{t,n}^{(3)}]|\\
		&\le& \textstyle\sum_{i,j,k,l=0}^{\infty} 2\delta_{3q}(i)\big[\delta_{3q}(j) \delta_{3q}(k+l) + \delta_{3q}(k) \delta_{3q}(j+l)\big]\\
		&&\quad\quad\quad \times \sup_{k,l\in\IZ}\sum_{1 \le t,t+k,t+l \le n}|a_{t+k,t+l,t}|\\
		&\le& \textstyle4\Delta_{0,3}^{2}\sum_{k=0}^{\infty}\Delta_{k,3}\cdot \sup_{k,l\in\IZ}\sum_{1 \le t,t+k,t+l \le n}|a_{t+k,t+l,t}|.
	\end{eqnarray*}

\end{proof}

\section{Proofs of section \ref{sec4}}

For linear time series, we use the model which was set up in \cite{dahlhauspolonik2009}.

\begin{proposition}[Linear time series models]\label{example_linearmodel}
	Suppose that Assumption \ref{ass3}\ref{ass3_m1} holds. Assume that
	\[
		X_{t,n} = \sum_{k=0}^{\infty}a_{t,n}(k)\varepsilon_{t-k}
	\]
	with some coefficients $a_{t,n}(k)$ and $a_{\theta}(k)$ satisfying
	\begin{equation}
		\sum_{t=1}^{n}|a_{t,n}(k) - a_{\theta_0(t/n)}(k)| \le C_B(k)\label{example_linearmodel_eq1}
	\end{equation}
	with some absolutely summable sequence $C_B(k)$. Assume that $\varepsilon_t$, $t\in\IZ$ are i.i.d. and have all moments, especially $\IE \varepsilon_0 = 0$ and $\IE \varepsilon_0^2 = 1$.\\
	Define $A_{\theta}(\lambda) := \sum_{k=0}^{\infty}a_{\theta}(k)e^{i\lambda k}$, the spectral density $f_{\theta}(\lambda) := \frac{1}{2\pi}|A_{\theta}(\lambda)|^{2}$ and real numbers $\gamma_{\theta}(k) := \frac{1}{2\pi}\int_{-\pi}^{\pi}A_{\theta}(\lambda)^{-1} e^{-i\lambda k} \dif \lambda$. Assume that
	\begin{enumerate}
		\item[(a)] For $\theta,\theta' \in \Theta$, it holds that $f_{\theta} = f_{\theta'}$ implies $\theta = \theta'$.
		\item[(b)] $|A_{\theta}(\lambda)| \ge \delta_A > 0$ uniformly in $\theta\in \Theta,\lambda\in [0,2\pi]$. $A_{\theta}(\lambda)$ is four times continuously differentiable in $\theta$. Assume that there exist $\beta_i > 3$, $L_i > 0$ such that $\nabla^{i} A_{\theta}(\cdot) \in \Sigma(\beta_i,L_i)$ ($i = 0,1,2,3,4$).
		\item[(c)] The minimal eigenvalue of $\frac{1}{4\pi}\int \nabla \log f_{\theta}(\lambda)\cdot \nabla \log f_{\theta}(\lambda)' \dif \lambda$ is bounded away from 0 uniformly in $\theta \in \Theta$.
	\end{enumerate}
	Then, assuming that $\varepsilon_t$ has a standard Gaussian distribution, $\l$ from \reff{infinitepastlikelihood} has the form
	\begin{equation}
		\l(z,\theta) = \frac{1}{2}\log\big(\frac{2\pi}{\gamma_{\theta}(0)^2}\big) + \frac{1}{2}\Big(\sum_{k=0}^{\infty}\gamma_{\theta}(k) z_{k+1}\Big)^2, \label{example_linearmodel_eq2_appendix}
	\end{equation}
	and Assumptions \ref{ass1}, \ref{ass2} and \ref{ass3} are fulfilled for \reff{example_linearmodel_eq2}, and it holds with the fourth cumulant $\kappa_{4}(\varepsilon_0)$ of $\varepsilon_0$ that
	\begin{eqnarray*}
		V(\theta) &=& \frac{1}{4\pi}\int_{-\pi}^{\pi}\nabla \log f_{\theta}(\lambda)\cdot \nabla \log f_{\theta}(\lambda)' \dif \lambda,\\
		I(\theta) &=& V(\theta) + \kappa_4(\varepsilon_0) \cdot \frac{\nabla \gamma_{\theta}(0) \nabla \gamma_{\theta}(0)'}{\gamma_{\theta}(0)^2}.
	\end{eqnarray*}
	If additionally Assumption \ref{ass5}\ref{ass5_b1} holds and
	\begin{enumerate}
		\item[(d)] $A_{\theta}(\lambda)$ is $l+1$-times continuously differentiable in $\theta$ and fulfills component-wise $\nabla^{i}A_{\theta}(\cdot) \in \Sigma(\beta_A,L_A)$ $(i = 0,...,l+1$).
	\end{enumerate}
	then Assumption \ref{ass5} is fulfilled and the bias term has the form
	\begin{equation}
		  \big| \IE[\partial_u^2 \nabla \l(\tilde Y_t(u),\theta)]\big|_{\theta = \theta_0(u)}\big|_{V(\theta_0(u))^{-1}}^2 = \big|\partial_u^2 \theta_0(u) + w(u)\big|_{V(\theta_0(u))}^2,\label{example_linearmodel_eq14}
	\end{equation}
	where $w(u) := V(\theta_0(u))^{-1}\tilde V(\theta_0(u))[\partial_u \theta_0(u), \partial_u \theta_0(u)]$ and $\tilde V(\theta) \in \IR^{p\times p \times p}$ is defined via $\tilde V(\theta)_{ijk} := \frac{1}{4\pi}\int_{-\pi}^{\pi}\frac{\nabla^2_{ij}f_{\theta}(\lambda)}{f_{\theta}(\lambda)}\cdot \frac{\nabla_k f_{\theta}(\lambda)}{f_{\theta}(\lambda)} \dif \lambda$.
\end{proposition}
\begin{proof}[Proof of Proposition \ref{example_linearmodel}] Put $\tilde X_t(\theta) = \sum_{k=0}^{\infty}a_{\theta}(k) \varepsilon_{t-k}$. By \reff{example_linearmodel_eq1}, we have for all $q \ge 1$ that
\[
	\sum_{t=1}^{n}\|X_{t,n} - \tilde X_t(t/n)\|_q \le \sum_{k=0}^{\infty}|a_{t,n}(k) - a_{\theta_0(t/n)}(k)|\cdot \|\varepsilon_{t-k}\|_q \le \|\varepsilon_0\|_q\cdot \sum_{k=0}^{\infty}C_B(k).
\]
It holds that $a_{\theta}(k) = \frac{1}{2\pi}\int_{-\pi}^{\pi}A_{\theta}(\lambda) e^{-i\lambda k} \dif \lambda$. By condition (b) and \cite{katznelson2004}, chapter I, section 4, we have that $\sup_{\theta \in \Theta}|\nabla^i a_{\theta}(k)|_{\infty} = O(k^{-(3+\eta)})$ with some $\eta > 0$ for $i = 0,1,2,3$.
\begin{eqnarray*}
	&& \|\tilde X_t(\theta) - \tilde X_t(\theta')\|_q\\
	&\le& \frac{1}{2\pi}\sum_{k=0}^{\infty}\Big| \int_{0}^{1}\langle \int_{-\pi}^{\pi}\nabla A_{\theta + s(\theta'-\theta)}(\lambda) e^{-i\lambda k} \dif \lambda, \theta-\theta'\rangle \dif s\Big| \cdot \|\varepsilon_{t-k}\|_q\\
	&\le& \sum_{k=0}^{\infty}\sup_{\theta \in \Theta}|\nabla a_{\theta}(k)|_{\infty} \cdot \|\varepsilon_0\|_q \cdot |\theta-\theta'|_1.
\end{eqnarray*}
Furthermore we obtain $\delta_{q}^{\tilde X(\theta)}(k) = |a_{\theta}(k)| \cdot \|\varepsilon_0 - \varepsilon_0^{*}\|_q = O(k^{-(3+\eta)})$. Since $|A_{\theta}(\lambda)| \ge \delta_A > 0$, it follows from the inverse Fourier transform that the process $\tilde X_t(\theta)$ is invertible in the sense that
\begin{eqnarray*}
	\textstyle\sum_{k=0}^{\infty}\gamma_{\theta}(k) \tilde X_{t-k}(\theta) &=& \sum_{k,l=0}^{\infty}\gamma_{\theta}(k) a_{\theta}(l) \varepsilon_{t-k-l}\\
	&=&\textstyle \sum_{m=0}^{\infty} \Big(\sum_{l=0}^{\infty}\gamma_{\theta}(m-l) a_{\theta}(l)\Big) \varepsilon_{t-m} = \varepsilon_t,
\end{eqnarray*}
since
\begin{eqnarray*}
	\sum_{l=0}^{\infty}\gamma_{\theta}(m-l) a_{\theta}(l) &=& \frac{1}{2\pi}\int_{-\pi}^{\pi} A_{\theta}(\lambda)^{-1} e^{-i\lambda m} \sum_{l=0}^{\infty}a_{\theta}(l)e^{i\lambda l} \dif \lambda\\
	&=& \frac{1}{2\pi}\int_{-\pi}^{\pi}e^{-i \lambda m} \dif \lambda = \Ii_{\{m = 0\}}.
\end{eqnarray*}
This shows $\IE[\tilde X_t(\theta)|\sF_{t-1}] = -\frac{1}{\gamma_{\theta}(0)}\sum_{k=1}^{\infty}\gamma_{\theta}(k) \tilde X_{t-1}(\theta)$ and $\Var(\tilde X_t(\theta)|\sF_{t-1}) = \frac{1}{\gamma_{\theta}(0)^2}$, thus the negative logarithm of the Gaussian conditional likelihood \reff{infinitepastlikelihood} has the form \reff{example_linearmodel_eq2_appendix}. From the conditions on $A_{\theta}(\lambda)$ in (b) it is straightforward to see that $\ell$ satisfies Assumption \ref{ass3}\ref{ass3_m5}.

Furthermore, we have
\begin{align}
	& \nabla \l(\tilde Y_t(\theta'),\theta)\big|_{\theta' = \theta}\nonumber\\
	&= -\frac{\nabla \gamma_{\theta}(0)}{\gamma_{\theta}(0)} + \Big(\sum_{k=0}^{\infty}\gamma_{\theta}(k) \tilde X_{t-k}(\theta)\Big)\cdot \Big(\sum_{k=0}^{\infty}\nabla \gamma_{\theta}(k) \tilde X_{t-k}(\theta)\Big)\nonumber\\
	&= \varepsilon_t \cdot \Big(\sum_{k=1}^{\infty}\nabla \gamma_{\theta}(k) \tilde X_{t-k}(\theta)\Big) + \frac{\nabla \gamma_{\theta}(0)}{\gamma_{\theta}(0)}\cdot\Big(\varepsilon_t \gamma_{\theta}(0) \tilde X_t(\theta) - 1\Big)\nonumber\\
	&= \varepsilon_t \cdot \Big[\sum_{k=1}^{\infty}\nabla \gamma_{\theta}(k) \tilde X_{t-k}(\theta) + \nabla \gamma_{\theta}(0)(\tilde X_t(\theta) - a_{\theta}(0)\varepsilon_t)\Big]\nonumber\\
	&\quad\quad\quad\quad + \frac{\nabla \gamma_{\theta}(0)}{\gamma_{\theta}(0)}\cdot\big(\varepsilon_t^2- 1\big)\label{example_linearmodel_eq10}
\end{align}
since $\gamma_{\theta}(0) a_{\theta}(0) = 1$ with the same Fourier argument as before. From \reff{example_linearmodel_eq10} it is immediate that $\nabla \l(\tilde Y_t(\theta'),\theta)\big|_{\theta' = \theta}$ is a martingale difference sequence w.r.t. $\sF_t$.
Using Bochner's theorem, we have $c_{\theta}(k) := \IE[\tilde X_{t}(\theta) \tilde X_{t-k}(\theta)] = \int_{-\pi}^{\pi}f_{\theta}(\lambda) e^{i\lambda k} \dif \lambda$ and thus
\begin{eqnarray}
	\IE \Big(\sum_{k=0}^{\infty}\gamma_{\theta}(k) \tilde X_{t-k}(u)\Big)^2 &=& \sum_{k,l=0}^{\infty}\gamma_{\theta}(k)\gamma_{\theta}(l) c_{\theta_0(u)}(k-l)\nonumber\\
	&=& \int_{-\pi}^{\pi}f_{\theta_0(u)}(\lambda) A_{\theta}(\lambda)^{-1} A_{\theta}(-\lambda)^{-1} \dif \lambda\nonumber\\
	&=& \frac{1}{2\pi}\int_{-\pi}^{\pi}\frac{f_{\theta_0(u)}(\lambda)}{f_{\theta}(\lambda)} \dif \lambda.\label{example_linearmodel_eq11}
\end{eqnarray}
Furthermore, Kolmogorov's formula (cf. \cite{brockwell}, Theorem 5.8.1) implies that $\frac{1}{2}\log(\frac{1}{2\pi \gamma_{\theta}(0)^2}) = \frac{1}{4\pi}\int_{-\pi}^{\pi}\log f_{\theta}(\lambda) \dif \lambda$,  which shows that 
\begin{equation}
	L(u,\theta) = \IE \ell(\tilde Y_t(u),\theta) = \frac{1}{4\pi}\int_{-\pi}^{\pi}\big\{\log f_{\theta}(\lambda) + \frac{f_{\theta_0(u)}(\lambda)}{f_{\theta}(\lambda)}\big\} \dif \lambda + \log(2\pi).\label{example_linearmodel_eq13}
\end{equation}
Since $\log(x) \le x-1$ with equality if and only if $x = 1$ and since condition (a) holds, $\theta_0(u)$ is the unique minimizer of $\theta \mapsto L(u,\theta)$. Differentiating \reff{example_linearmodel_eq13} twice with respect to $\theta$ (and replacing $\theta_0(u)$ by $\theta$ afterwards), we obtain
\begin{eqnarray}
	V(\theta) &=& 
		\frac{1}{4\pi}\int_{-\pi}^{\pi}\frac{\nabla f_{\theta}(\lambda) \nabla f_{\theta}(\lambda)'}{f_{\theta}(\lambda)^2} \dif \lambda\nonumber\\
		&=& \frac{1}{4\pi}\int_{-\pi}^{\pi}\nabla \log f_{\theta}(\lambda) \cdot \nabla \log f_{\theta}(\lambda)'\dif \lambda.\label{example_linearmodel_eq16}
\end{eqnarray}
Condition (c) implies that the minimal eigenvalue of $V(\theta)$ is bounded away from 0 uniformly in $\theta \in \Theta$. Define $G_{t}(\theta) := \sum_{k=0}^{\infty}\nabla \gamma_{\theta}(k) \tilde X_{t-k}(\theta) - \nabla \gamma_{\theta}(0) a_{\theta}(0) \varepsilon_t$ which is $\sF_{t-1}$-measurable and $\IE G_t(\theta) = 0$. Recall $\gamma_{\theta}(0) a_{\theta}(0) = 1$. From \reff{example_linearmodel_eq11} it follows that
\begin{eqnarray*}
	I(\theta) &=& \IE\big[\nabla \l(\tilde Y_t(\theta'),\theta)\big|_{\theta' = \theta}\cdot \nabla \l(\tilde Y_t(\theta'),\theta)'\big|_{\theta' = \theta}\big]\\
	&=& \IE[G_t(\theta)G_t(\theta)'] + \frac{\nabla \gamma_{\theta}(0)\nabla \gamma_{\theta}(0)'}{\gamma_{\theta}(0)^2}\IE[(\varepsilon_0^2-1)^2],
\end{eqnarray*}
where
\begin{eqnarray*}
	\IE[G_t(\theta) G_t(\theta)'] &=& \sum_{k,l=0}^{\infty}\nabla \gamma_{\theta}(k) \nabla \gamma_{\theta}(l)' c_{\theta}(k-l) - \frac{\nabla \gamma_{\theta}(0)\nabla \gamma_{\theta}(0)'}{\gamma_{\theta}(0)^2}\\
	&=& \frac{1}{4\pi}\int_{-\pi}^{\pi}\frac{\nabla f_{\theta}(\lambda) \nabla f_{\theta}(\lambda)'}{f_{\theta}(\lambda)^2}\dif \lambda - 2\frac{\nabla \gamma_{\theta}(0) \nabla \gamma_{\theta}(0)'}{\gamma_{\theta}(0)^2},
\end{eqnarray*}
thus
\[
	I(\theta) = V(\theta) + \kappa_4(\varepsilon_0)\cdot \frac{\nabla \gamma_{\theta}(0) \nabla \gamma_{\theta}(0)'}{\gamma_{\theta}(0)^2},
\]
where $\kappa_4(\varepsilon_0) = \IE \varepsilon_0^4 - 3$ denotes the fourth cumulant of $\varepsilon_0$.

Now let $\theta_0(u)$ be twice continuously differentiable. Note that
\[
	\partial_{z_i}\partial_{z_j}\nabla \ell(z,\theta) = 2\big( \gamma_{\theta}(j-1) \nabla \gamma_{\theta}(i-1) + \gamma_{\theta}(i-1) \nabla \gamma_{\theta}(j-1)\big),
\]
which implies Assumption \ref{ass4}.

Define $\tilde X_t(\theta) = \sum_{k=0}^{\infty}a_{\theta}(k) \varepsilon_{t-k}$. Then $\nabla \tilde X_t(\theta) = \sum_{k=0}^{\infty}\nabla a_{\theta}(k) \varepsilon_{t-k}$, $\nabla^2 \tilde X_t(\theta) = \sum_{k=0}^{\infty}\nabla^2 a_{\theta}(k) \varepsilon_{t-k}$.
We have 
\begin{eqnarray*}
	&& \IE\big[\partial_u^2 \nabla \l(\tilde Y_t(u),\theta)\big|_{\theta = \theta_0(u)}\big]\\
	&=& \partial_u^2 \Big[\frac{1}{4\pi}\int_{-\pi}^{\pi}\Big\{\frac{\nabla f_{\theta}(\lambda)}{f_{\theta}(\lambda)} - \frac{f_{\theta_0(u)}(\lambda)}{f_{\theta}(\lambda)^2}\nabla f_{\theta}(\lambda)\Big\} \dif \lambda\Big]\Big|_{\theta = \theta_0(u)}\\
	&=& -\frac{1}{4\pi}\int_{-\pi}^{\pi}\partial_u^2 f_{\theta_0(u)}(\lambda)\cdot \nabla \big(f_{\theta_0(u)}(\lambda)^{-1}\big) \dif \lambda\\
	&=& -\frac{1}{4\pi}\int_{-\pi}^{\pi}\Big\{\frac{\langle \nabla^2 f_{\theta_0(u)}(\lambda)\partial_u\theta_0(u),\partial_u\theta_0(u)\rangle\cdot \nabla f_{\theta_0(u)}(\lambda)}{f_{\theta_0(u)}(\lambda)^2}\\
	&&\quad\quad\quad\quad\quad\quad + \frac{\langle \nabla f_{\theta_0(u)}(\lambda),\partial_u^2\theta_0(u)\rangle \nabla f_{\theta_0(u)}(\lambda)}{f_{\theta_0(u)}(\lambda)^2}\Big\} \dif \lambda\\
	&=& -V(\theta_0(u))\cdot \partial_u^2 \theta_0(u) - \tilde V(\theta_0(u))[\partial_u\theta_0(u),\partial_u\theta_0(u)].
\end{eqnarray*}
Thus we obtain the following decomposition of the bias term
\begin{eqnarray*}
	&& \big| \IE[\partial_u^2 \nabla \l(\tilde Y_t(u),\theta)]\big|_{\theta = \theta_0(u)}\big|_{V(\theta_0(u))^{-1}}^2\\
	&=& |\partial_u^2 \theta_0(u)|_{V(\theta_0(u))}^2 + |\tilde V(\theta_0(u))[\partial_u\theta_0(u),\partial_u\theta_0(u)]|_{V(\theta_0(u))^{-1}}^2\\
	&&\quad\quad + 2 \langle \partial_u^2\theta_0(u),\tilde V(\theta_0(u))[\partial_u\theta_0(u),\partial_u\theta_0(u)]\rangle\\
	&=& \big|\partial_u^2 \theta_0(u) + w(u)\big|_{V(\theta_0(u))}^2.
\end{eqnarray*}

\end{proof}

\begin{proposition}[Recursively defined time series]\label{example_recursion}
	Assume that $\varepsilon_t$, $t\in\IZ$ are i.i.d. and have all moments with $\IE \varepsilon_0 = 0$ and $\IE \varepsilon_0^2 = 1$. Suppose that Assumption \ref{ass3}\ref{ass3_m1} is fulfilled. Assume that $X_{t,n}$ fulfills
	\[
		X_{t,n} = \mu\big(Z_{t-1,n},\theta_0\big(\frac{t}{n}\big)\big) + \sigma\big(Z_{t-1,n},\theta_0\big(\frac{t}{n}\big)\big)\varepsilon_t, \quad t = 1,...,n.
	\]
	Assume that $\mu,\sigma:\IR^r \times \Theta \to \IR$ satisfy
	\begin{equation}
		\sup_{\theta}\sup_{y\not= y'}\frac{|\mu(y,\theta) - \mu(y',\theta)|}{|y-y'|_{\chi,1}} + \sup_{\theta}\sup_{y\not=y'}\frac{|\sigma(y,\theta) - \sigma(y',\theta)|}{|y-y'|_{\chi,1}}\|\varepsilon_0\|_q \le 1\label{example_recursion_eq1}
	\end{equation}
	for all $q \ge 1$ with some $\chi \in \IR^r_{\ge 0}$ with $|\chi|_1 < 1$. Assume that $\sigma(\cdot) \ge \sigma_0$ with some constant $\sigma_0 > 0$. Assume that $\nabla \sigma \not= 0$, and
	\begin{enumerate}
		\item[(a)] For $g \in \{\nabla^{i} \mu, \nabla^{i}\sigma: i = 0,1,2,3\}$, it holds $\sup_y\sup_{\theta \not= \theta'}\frac{|g(y,\theta) - g(y,\theta')|}{|\theta-\theta'|_1 \cdot (1+|y|_1)} < \infty$ and $\sup_{\theta \in \Theta}\sup_{y \not= y'}\frac{|g(y,\theta) - g(y,\theta')|}{|y-y'|_1} < \infty$ for each component.
		\item[(b)] for fixed $u \in [0,1]$, the two conditions $
			\mu(\tilde Y_0(\theta_0(u)),\theta) = \mu(\tilde Y_{0}(\theta_0(u)),\theta_0(u))$ a.s. and $\sigma(\tilde Y_0(\theta_0(u)),\theta)  = \sigma(\tilde Y_0(\theta_0(u)),\theta_0(u))$ a.s. imply $\theta = \theta_0(u)$. 
		\item[(c)] uniformly in $\theta \in \Theta$, the smallest eigenvalues of the matrices $W_{\mu}(\theta) := \IE[\frac{\nabla \mu(\tilde Y_0(\theta'),\theta) \nabla \mu(\tilde Y_0(\theta'),\theta)'}{\sigma^2(\tilde Y_0(\theta),\theta)}]\big|_{\theta' = \theta}$ and $W_{\sigma}(\theta) := \IE[\frac{\nabla \sigma(\tilde Y_0(\theta'),\theta) \nabla \sigma(\tilde Y_0(\theta'),\theta)'}{\sigma^2(\tilde Y_0(\theta),\theta)}]\big|_{\theta' = \theta}$ are bounded from below by some $\tilde \lambda_0 > 0$.
	\end{enumerate}
	Then Assumptions \ref{ass1}, \ref{ass2} and \ref{ass3} are fulfilled for $\ell$ chosen as in \reff{infinitepastlikelihood} assuming that $\varepsilon_t$ is standard Gaussian distributed, i.e. for
	\begin{equation}
		\l(x,y,\theta) := \frac{1}{2}\log\big(2\pi \sigma(y,\theta)^2\big)+ \frac{1}{2}\Big(\frac{x-\mu(y,\theta)}{\sigma(y,\theta)}\Big)^2.\label{example_recursion_eq2_appendix}
	\end{equation}
	In this case, it holds with $W_{\mu\sigma} := \IE[\frac{\nabla \mu(\tilde Y_0(\theta'),\theta) \nabla \sigma(\tilde Y_0(\theta'),\theta)'}{\sigma^2(\tilde Y_0(\theta),\theta)}]\big|_{\theta' = \theta}$ that
	\begin{eqnarray*}
		V(\theta) &=& W_{\mu}(\theta) + 2 W_{\sigma}(\theta),\\
		I(\theta) &=& W_{\mu}(\theta) + (\IE \varepsilon_0^4 - 1) W_{\sigma}(\theta) + \IE \varepsilon_0^3 (W_{\mu\sigma} + W_{\mu\sigma}').
	\end{eqnarray*}
\end{proposition}
\begin{proof}[Proof of Proposition \ref{example_recursion}] Condition \reff{example_recursion_eq1} implies that $X_{t,n}$ exists and is a.s. unique (cf. Proposition 4.3 in \cite{dahlhaus2017} or Theorem 5.1 in \cite{shaowu2007}) and that there exist $C_{\rho} > 0$, $0 < \rho < 1$ such that $\sup_{n\in\IN}\sup_{t=1,...,n}\delta_{q}^{X_{\cdot,n}}(k), \sup_{\theta \in \Theta}\delta_{q}^{\tilde X(\theta)}(k) \le C \rho^k$ for all $k \ge 0$.
With slight changes in the argumentation of their proof, \reff{ass1_s1_eq1} and $D_q := \max\{\sup_{\theta \in \Theta}\|\tilde X_0(\theta)\|_q, \sup_{n\in\IN}\sup_{t=1,...,n}\|X_{t,n}\|_q\} < \infty$ follows similar as Proposition 4.3 and Lemma 4.4 in \cite{dahlhaus2017}. Thus, the condition of Assumption \ref{ass1} and \ref{ass2} are fulfilled.

Put $L(u,\theta) = \IE \ell(\tilde Y_0(u),\theta)$. Let us omit the argument $\tilde Y_0(u)$ of $\mu(\cdot,\theta)$ and $\sigma(\cdot,\theta)$ in the following. It holds that
\begin{eqnarray*}
	&& L(u,\theta) - L(u,\theta_0(u))\\
	&=& \IE\Big(\frac{\mu(\theta) - \mu(\theta_0(u))}{\sigma(\theta)}\Big)^2 + \IE\Big[\frac{\sigma^2(\theta_0(u))}{\sigma^2(\theta)} - \log\frac{\sigma^2(\theta_0(u))}{\sigma^2(\theta)}-1\Big].
\end{eqnarray*}
By a Taylor expansion of $x \mapsto x - \log(x)-1$, we obtain that the second summand is lower bounded by $\frac{1}{4}\IE\Big[\frac{(\sigma^2(\theta_0(u)) - \sigma^2(\theta))^2}{(\sigma^2(\theta_0(u)) - \sigma^2(\theta))^2+\Sigma^4(\theta)}\Big]$. By these inequalities, condition (b) shows that the unique minimizer of $\theta \mapsto L(u,\theta)$ is given by $\theta_0(u)$.\\
Omitting the arguments, $(z,\theta)$ (where $z = (x,y)$), we have
\begin{eqnarray*}
	\nabla \l &=& -\big(\frac{x-\mu}{\sigma}\big)\cdot \frac{\nabla \mu}{\sigma} + \frac{\nabla \sigma}{\sigma}\cdot \Big[1-\big(\frac{x-\mu}{\sigma}\big)^2\Big],\\
	\nabla^2 \l &=& \frac{\nabla \mu \nabla \mu'}{\sigma^2} + \big(\frac{x-\mu}{\sigma}\big)\cdot \Big[2\frac{\nabla \mu \nabla \sigma' + \nabla \sigma \nabla \mu'}{\sigma^2} - \frac{\nabla^2 \mu}{\sigma}\Big]\\
	&&\quad\quad + \frac{\nabla^2 \sigma}{\sigma}\Big[1-\big(\frac{x-\mu}{\sigma}\big)^2\Big] + \frac{\nabla \sigma \nabla \sigma'}{\sigma^2}\Big[3\big(\frac{x-\mu}{\sigma}\big)^2 - 1\Big].
\end{eqnarray*}
Since $\frac{\tilde X_t(\theta) - \mu(\tilde Y_{t-1}(\theta),\theta)}{\sigma(\tilde Y_{t-1}(\theta),\theta)} = \varepsilon_t$ and $\IE \varepsilon_t = 0$, $\IE \varepsilon_0^2 = 1$, we obtain that $\nabla \l(\tilde Y_t(\theta),\theta)$ is a martingale difference sequence with respect to $\sF_{t}$ in each component. Furthermore,
\begin{eqnarray*}
	&& V(\theta) = \IE \nabla^2 \l(\tilde Y_0(\theta),\theta)\\
	&=& \IE\big[\frac{\big(\nabla \mu(\tilde Y_0(\theta'),\theta) \nabla \mu(\tilde Y_0(\theta'),\theta)' + 2\nabla \sigma(\tilde Y_0(\theta'),\theta) \nabla \sigma(\tilde Y_0(\theta'),\theta)'\big)}{\sigma^2(\tilde Y_0(\theta),\theta)}\big]\big|_{\theta' = \theta},
\end{eqnarray*}
whose smallest eigenvalue is bounded from below by condition (c). By condition (a) and the fact that $\sigma(\cdot) \ge \sigma_0 > 0$, straightforward calculations show that Assumption \ref{ass3}\ref{ass3_m5} is fulfilled with the truncated (and thus summable) sequence $\chi_i = \Ii_{\{1 \le i \le r\}}$ and some $M \ge 2$.
\end{proof}

We do not want to go into details when Assumption \ref{ass5} is fulfilled in the situation of Proposition \ref{example_recursion}. Regarding the results in Section 4 in \cite{dahlhaus2017} we need additionally that $\mu,\sigma$ are four times continuously differentiable and fulfill some moment conditions.

  
\section{The bias in the tvAR(1) model}

We use the results from Proposition \ref{example_linearmodel}. Assume that for $t\in\IZ$, $\tilde X_t(\theta) = \alpha\cdot \tilde X_{t-1}(\theta) + \sigma\varepsilon_t$. Here it holds that $V(\theta) = \text{diag}(\frac{\sigma^2}{1-\alpha^2},\frac{2}{\sigma^2})$, where $\theta = (\alpha,\sigma)'$. Here, it holds that $f_{\theta}(\lambda) = \frac{\sigma^2}{2\pi}|1 - \alpha e^{i\lambda}|^{-2}$ and thus
\[
	\nabla (f_{\theta}(\lambda)^{-1}) = \begin{pmatrix}
		\frac{4\pi}{\sigma^2}\big(\alpha - \cos(\lambda)\big)\\
		\frac{-4\pi}{\sigma^3}\big(1+\alpha^2 - 2\alpha \cos(\lambda)\big)
	\end{pmatrix}.
\]
Moreover, the variance and the first covariance take the form $c_{\theta}(0) := \int_{-\pi}^{\pi} f_{\theta}(\lambda) \dif \lambda = \frac{\sigma^2}{1-\alpha^2}$ and $c_{\theta}(1) := \int_{-\pi}^{\pi} e^{i\lambda} f_{\theta}(\lambda) \dif \lambda = \frac{\sigma^2 \alpha}{1-\alpha^2}$. This leads to (let $\bar \nabla$ denote the derivative with respect to $\bar\theta$)
\begin{eqnarray*}
	\tilde V(\theta)_{ijk} &=& -\bar\nabla_{ij}^2\Big[\frac{1}{4\pi}\int_{-\pi}^{\pi}f_{\bar\theta}(\lambda) \cdot \nabla (f_{\theta}(\lambda)^{-1}) \dif \lambda\big|_{\bar \theta = \theta}\Big]\\
	&=& \begin{cases}
		-\frac{1}{\sigma^2}\big\{ \alpha \cdot \nabla_{ij}c_{\theta}(0) - \nabla_{ij}c_{\theta}(1)\big\}, & k = \alpha\\
		\frac{1}{\sigma^3}\big\{ (1+\alpha^2)\cdot \nabla_{ij}c_{\theta}(0) - 2\alpha\cdot \nabla_{ij}c_{\theta}(1)\big\}, & k = \sigma
	\end{cases}.
\end{eqnarray*}
The derivatives of the covariances read
\[
	\nabla c_{\theta}(0) = c_{\theta}(0)\begin{pmatrix}
		\frac{2\alpha}{1-\alpha^2}\\
		\frac{2}{\sigma}
	\end{pmatrix}, \quad\quad \nabla c_{\theta}(1) = \begin{pmatrix}1\\0\end{pmatrix}c_{\theta}(0) + \alpha \nabla c_{\theta}(0),
\]
and
\begin{eqnarray*}
	\nabla^2 c_{\theta}(0) &=& \begin{pmatrix}
		\frac{2\alpha}{1-\alpha^2}\\
		\frac{2}{\sigma}
	\end{pmatrix}\nabla c_{\theta}(0)' + c_{\theta}(0)\cdot \begin{pmatrix}\frac{2(1+\alpha^2)}{(1-\alpha^2)^2} & 0\\
	0 & -\frac{2}{\sigma^2}\end{pmatrix}\\
	&=& c_{\theta}(0)\cdot \begin{pmatrix}
		\frac{2(1+3\alpha^2)}{(1-\alpha^2)^2} & \frac{1}{\sigma}\frac{4\alpha}{1-\alpha^2}\\
		\frac{1}{\sigma}\frac{4\alpha}{1-\alpha^2} & \frac{2}{\sigma^2}
	\end{pmatrix},\\
	\nabla^2 c_{\theta}(1) &=& \begin{pmatrix}1\\0\end{pmatrix}\nabla c_{\theta}(0)' + \nabla c_{\theta}(0) \begin{pmatrix}1\\0\end{pmatrix}' + \alpha \nabla^2 c_{\theta}(0)\\
	&=& c_{\theta}(0)\cdot \begin{pmatrix}
		\frac{2\alpha(3+\alpha^2)}{(1-\alpha^2)^2} & \frac{2}{\sigma}\frac{1+\alpha^2}{1-\alpha^2}\\
		\frac{2}{\sigma}\frac{1+\alpha^2}{1-\alpha^2} & \frac{2\alpha}{\sigma^2}
	\end{pmatrix}.
\end{eqnarray*}
We obtain
\begin{eqnarray*}
	\tilde V(\theta)_{ij\alpha} &=& \frac{c_{\theta}(0)}{\sigma^2}\begin{pmatrix}
		\frac{4\alpha}{1-\alpha^2} & \frac{2}{\sigma}\\
		\frac{2}{\sigma} & 0
	\end{pmatrix},\\
	\tilde V(\theta)_{ij \sigma} &=& \frac{1}{\sigma^3}\big\{(1-\alpha^2) \nabla^2 c_{\theta}(0) - 2\alpha \cdot c_{\theta}(0)\cdot \begin{pmatrix}\frac{4\alpha}{1-\alpha^2} & \frac{2}{\sigma}\\
	\frac{2}{\sigma} & 0\end{pmatrix}\big\}\\
	&=& \frac{c_{\theta}(0)}{\sigma^3}\begin{pmatrix}
		2 & 0\\
		0 & 2\frac{(1-\alpha^2)}{\sigma^2}
	\end{pmatrix},
\end{eqnarray*}
thus
\begin{eqnarray*}
	&& \tilde V(\theta_0(u))[\partial_u\theta_0(u),\partial_u\theta_0(u)]\\
	&=& \frac{c_{\theta_0(u)}(0)}{\sigma(u)^2}\begin{pmatrix}
		4\Big(\frac{\alpha(u)}{1-\alpha(u)^2}(\partial_u \alpha(u))^2 + \frac{(\partial_u \alpha(u)) (\partial_u \sigma(u))}{\sigma(u)}\Big)\\
		\frac{2}{\sigma}\Big((\partial_u \alpha(u))^2 + \frac{1-\alpha(u)^2}{\sigma(u)^2}(\partial_u \sigma(u))^2\Big)
	\end{pmatrix} = V(\theta_0(u)) w(u),
\end{eqnarray*}
where
\[
	w(u) =  \begin{pmatrix}
		4\Big(\frac{\alpha(u) (\partial_u \alpha(u))^2}{1-\alpha(u)^2} + \partial_u \alpha(u)\cdot \frac{\partial_u \sigma(u)}{\sigma(u)}\Big)\\
		\sigma(u)\Big(\frac{(\partial_u \alpha(u))^2}{1-\alpha(u)^2} + \big(\frac{\partial_u \sigma(u)}{\sigma(u)}\big)^2\Big)
	\end{pmatrix}
\]
By using \reff{example_linearmodel_eq14}, we finally obtain
\begin{eqnarray*}
	&& \big|\IE\big[\partial_u^2 \nabla \l(\tilde Y_t(u),\theta)\big|_{\theta = \theta_0(u)}\big]\big|_{V(\theta_0(u))^{-1}}^2\\
	&=& \frac{1}{1-\alpha(u)^2}\Big| \partial_u^2 \alpha(u) + 4\Big(\frac{\alpha(u) (\partial_u \alpha(u))^2}{1-\alpha(u)^2} + \partial_u \alpha(u)\cdot \frac{\partial_u \sigma(u)}{\sigma(u)}\Big)\Big|^2\\
	&&\quad\quad + \frac{2}{\sigma(u)^2}\Big| \partial_u^2 \sigma(u) + \sigma(u)\Big(\frac{(\partial_u \alpha(u))^2}{1-\alpha(u)^2} + \big(\frac{\partial_u \sigma(u)}{\sigma(u)}\big)^2\Big)\Big|^2
\end{eqnarray*}

If $\theta = \alpha$, i.e. $\sigma \equiv \sigma_0 > 0$ is assumed to be constant and known, we obtain
\[
	\big|\IE\big[\partial_u^2 \nabla \l(\tilde Y_t(u),\theta)\big|_{\theta = \theta_0(u)}\big]\big|_{V(\theta_0(u))^{-1}}^2 = \frac{1}{1-\alpha(u)^2}\Big| \partial_u^2 \alpha(u) + 4 \frac{\alpha(u) (\partial_u \alpha(u))^2}{1-\alpha(u)^2}\Big|^2.
\]

\section{The bias in the tvMA(1) model}
We use the results from Proposition \ref{example_linearmodel}. Assume that for $t\in\IZ$, $\tilde X_t(\theta) = \sigma \varepsilon_t + \alpha \sigma \varepsilon_{t-1}$. Then we have $f_{\theta}(\lambda) = \frac{\sigma^2}{2\pi}|1+\alpha e^{i\lambda}|^2 = \frac{\sigma^2}{2\pi}(1+\alpha^2 + 2\alpha \cos(\lambda))$. Note that $f_{\theta}(\lambda)^{-1} = (\frac{2\pi}{\sigma^2})^2 \cdot f_{\theta}^{AR}(\lambda)$ with $f_{\tilde \theta}^{AR}(\lambda) = \frac{\sigma^2}{2\pi}|1+\alpha e^{i\lambda}|^2$ corresponding to the spectral density of an AR(1) process with parameter $-\alpha$ instead of $\alpha$. Recall that Kolmogorov's formula implies $\frac{1}{2}\log(\frac{\sigma^2}{2\pi}) = \frac{1}{4\pi}\int_{-\pi}^{\pi}\log f_{\theta}^{AR}(\lambda) \dif \lambda$ and thus
\[
	\begin{pmatrix}0\\
\frac{1}{\sigma}\end{pmatrix} = \frac{1}{4\pi}\int_{-\pi}^{\pi}\log f_{\theta}^{AR}(\lambda) \dif \lambda.
\]
By \reff{example_linearmodel_eq16}, this leads to the same $V$ as in the AR case:
\begin{eqnarray*}
	V(\theta) &=& \frac{1}{4\pi}\int_{-\pi}^{\pi}\Big[ \begin{pmatrix}0\\ \frac{-4}{\sigma}\end{pmatrix} + \nabla \log f_{\theta}^{AR}(\lambda)\Big]\cdot \Big[ \begin{pmatrix}0\\ \frac{-4}{\sigma}\end{pmatrix} + \nabla \log f_{\theta}^{AR}(\lambda)\Big]' \dif \lambda\\
	&=& \frac{1}{4\pi}\int_{-\pi}^{\pi}\nabla \log f_{\theta}^{AR}(\lambda)\cdot \nabla \log f_{\theta}^{AR}(\lambda)' \dif \lambda = \begin{pmatrix}
		\frac{c_{\theta}(0)}{\sigma^2} & 0\\
		0 & \frac{2}{\sigma^2}
	\end{pmatrix}.
\end{eqnarray*}
To calculate $\tilde V(\theta)$, note that
\[
	\nabla^2 f_{\theta}(\lambda) = \begin{pmatrix}\frac{\sigma^2}{\pi} & \frac{2\sigma}{\pi}(\alpha+\cos(\lambda))\\
	\frac{2\sigma}{\pi}(\alpha+\cos(\lambda)) & \frac{1}{\pi}(1+\alpha^2 + 2\alpha \cos(\lambda))\end{pmatrix}.
\]
We have
\begin{eqnarray*}
	&& \tilde V(\theta)_{\cdot,\cdot,k}\\
	&=& -\bar \nabla_k \Big[\frac{1}{4\pi}\int_{-\pi}^{\pi}\nabla^2 f_{\theta}(\lambda) \cdot \big[ \big(\frac{2\pi}{\bar \sigma^2}\big)^2 \cdot f_{\bar\theta}^{AR}(\lambda)\big] \dif \lambda\Big]\\
	&=& -\pi\cdot \begin{pmatrix}\frac{\sigma^2}{\pi}\nabla_k\big(\frac{c_{\theta}(0)}{\sigma^4}\big) & \frac{2\sigma}{\pi}(\alpha \nabla_k\big(\frac{c_{\theta}(0)}{\sigma^4}\big)+\nabla_k\big(\frac{c_{\theta}(1)}{\sigma^4}\big))\\
	\frac{2\sigma}{\pi}(\alpha \nabla_k\big(\frac{c_{\theta}(0)}{\sigma^4}\big)+ \nabla_k\big(\frac{c_{\theta}(1)}{\sigma^4}\big)) & \frac{1}{\pi}( (1+\alpha^2)\nabla_k\big(\frac{c_{\theta}(0)}{\sigma^4}\big) + 2\alpha \nabla_k\big(\frac{c_{\theta}(1)}{\sigma^4}\big))\end{pmatrix},
\end{eqnarray*}
where $c_{\theta}(0) = \frac{\sigma^2}{1-\alpha^2}$ and $c_{\theta}(1) = -\alpha \cdot c_{\theta}(0)$. With 
\[
	\nabla \big(\frac{c_{\theta}(0)}{\sigma^4}\big) = \frac{c_{\theta}(0)}{\sigma^4}\begin{pmatrix}
		\frac{2\alpha}{1-\alpha^2}\\
		-\frac{2}{\sigma}
	\end{pmatrix}, \quad\quad \nabla \big(\frac{c_{\theta}(1)}{\sigma^4}\big) = \frac{c_{\theta}(0)}{\sigma^4}\begin{pmatrix}
		-1\\
		0
	\end{pmatrix} + (-\alpha) \nabla  \big(\frac{c_{\theta}(0)}{\sigma^4}\big),
\]
we obtain
\[
	\tilde V(\theta)_{ij\alpha} = \frac{c_{\theta}(0)}{\sigma^2}\cdot \begin{pmatrix}
		- \frac{2\alpha}{1-\alpha^2} & \frac{2}{\sigma}\\
		\frac{2}{\sigma} & 0
	\end{pmatrix}, \quad\quad \tilde V(\theta)_{ij\sigma} = \frac{c_{\theta}(0)}{\sigma^2}\begin{pmatrix}
		\frac{2}{\sigma} & 0\\
		0 & \frac{2(1-\alpha^2)}{\sigma^3}
	\end{pmatrix},
\]
thus
\begin{eqnarray*}
	&& \tilde V(\theta_0(u))[\partial_u \theta_0(u),\partial_u \theta_0(u)]\\
	&=& \frac{c_{\theta_0(u)}(0)}{\sigma(u)^2}\begin{pmatrix}
		-\frac{2\alpha(u)}{1-\alpha(u)^2}(\partial_u \alpha(u))^2 + \frac{4}{\sigma(u)}(\partial_u \alpha(u))\cdot (\partial_u \sigma(u))\\
		\frac{2}{\sigma(u)} (\partial_u \alpha(u))^2 + \frac{2(1-\alpha(u)^2)}{\sigma(u)^3}(\partial_u \sigma(u))^2
	\end{pmatrix}\\
	&=& V(\theta_0(u)) w(u),
\end{eqnarray*}
where
\[
	w(u) := \begin{pmatrix}
		-\frac{2\alpha(u)(\partial_u \alpha(u))^2}{1-\alpha(u)^2} + 4 \partial_u \alpha(u) \frac{\partial_u \sigma(u)}{\sigma(u)}\\
		\sigma(u)\Big(\frac{(\partial_u \alpha(u))^2}{1-\alpha(u)^2} + \big(\frac{\partial_u \sigma(u)}{\sigma(u)}\big)^2\Big).
	\end{pmatrix}
\]
Using \reff{example_linearmodel_eq14}, we finally obtain
\begin{eqnarray*}
	&& \big|\IE\big[\partial_u^2 \nabla \l(\tilde Y_t(u),\theta)\big|_{\theta = \theta_0(u)}\big]\big|_{V(\theta_0(u))^{-1}}^2\\
	&=& \frac{1}{1-\alpha(u)^2}\Big| \partial_u^2 \alpha(u) + 2\Big(-\frac{\alpha(u) (\partial_u \alpha(u))^2}{1-\alpha(u)^2} + 2\partial_u \alpha(u)\cdot \frac{\partial_u \sigma(u)}{\sigma(u)}\Big)\Big|^2\\
	&&\quad\quad + \frac{2}{\sigma(u)^2}\Big| \partial_u^2 \sigma(u) + \sigma(u)\Big(\frac{(\partial_u \alpha(u))^2}{1-\alpha(u)^2} + \big(\frac{\partial_u \sigma(u)}{\sigma(u)}\big)^2\Big)\Big|^2.
\end{eqnarray*}

\subsection{Bias terms in Examples \ref{example_ar_mean} and \ref{example_ar_var}}

\begin{lemma}\label{technical_ar_mean}
	In the situation of Example \ref{example_ar_mean}, it holds that
	\[
		\big|\IE\big[\partial_u^2  \nabla \l(\tilde Y_t(u),\theta)\big|_{\theta = \theta_0(u)}\big]\big|_{V(\theta_0(u))^{-1}}^2 = \big|\partial_u^2 \theta_0(u) + 2w(u)\big|_{V(\theta_0(u))}^2,
	\]
	where $P(u) := \partial_u W(\theta_0(u))$ and
	\[
		w(u) := \begin{pmatrix}
			W(\theta_0(u))^{-1}P(u) \partial_u \alpha(u)\\
			\frac{\sigma(u)}{2}\Big[|\partial_u \alpha(u)|_{\frac{W(\theta_0(u))}{\sigma(u)^2}}^2 + \big(\frac{\partial_u \sigma(u)}{\sigma(u)}\big)^2\Big]
		\end{pmatrix}
	\]
\end{lemma}
\begin{proof} Note that $\theta = (\alpha',\sigma)'$.
It holds that
	\begin{eqnarray*}
		\partial_u^2 \nabla_{\alpha} \l(\tilde Y_t(u),\theta) &=& \frac{1}{\sigma^2}\partial_u^2 \Big[ \big(\tilde X_t(u)-\langle \alpha, \mu(\tilde Z_{t-1}(u))\rangle\big)\cdot \mu(\tilde Z_{t-1}(u))\Big]\\
		&=& \frac{1}{\sigma^2}\partial_u\Big[ \big(\partial_u \tilde X_t(u) - \langle \alpha, \partial_u \mu(\tilde Z_{t-1}(u))\rangle\big)\cdot \mu(\tilde Z_{t-1}(u))\\
		&&\quad\quad\quad + \big(\tilde X_t(u) - \langle \alpha, \mu(\tilde Z_{t-1}(u))\rangle\big)\cdot \partial_u \mu(\tilde Z_{t-1}(u))\Big]\\
		&=& \frac{1}{\sigma^2}\Big[\big(\partial_u^2 \tilde X_t(u) - \langle \alpha, \partial_u^2 \mu(\tilde Z_{t-1}(u))\big) \mu(\tilde Z_{t-1}(u))\\
		&&\quad\quad + 2 \big(\partial_u \tilde X_t(u) - \langle \alpha, \partial_u \mu(\tilde Z_{t-1}(u))\rangle\big)\cdot \partial_u \mu(\tilde Z_{t-1}(u))\\
		&&\quad\quad + \big(\tilde X_t(u) - \langle \alpha, \mu(\tilde Z_{t-1}(u))\rangle\big)\cdot \partial_u^2 \mu(\tilde Z_{t-1}(u))\Big].
	\end{eqnarray*}
	The recursion implies
	\begin{eqnarray*}
		\tilde X_t(u) &=& \langle \alpha(u), \mu(\tilde Z_{t-1}(u))\rangle + \sigma(u) \varepsilon_t,\\
		\partial_u \tilde X_t(u) &=& \langle \partial_u\alpha(u), \mu(\tilde Z_{t-1}(u))\rangle + \langle \alpha(u), \partial_u \mu(\tilde Z_{t-1}(u))\rangle + \partial_u \sigma(u) \varepsilon_t,\\
		\partial_u^2 \tilde X_t(u) &=& \langle \partial_u^2 \alpha(u),  \mu(\tilde Z_{t-1}(u))\rangle + 2\langle \partial_u\alpha(u), \partial_u \mu(\tilde Z_{t-1}(u))\rangle\\
		&&\quad\quad + \langle \alpha(u), \partial_u^2 \mu(\tilde Z_{t-1}(u))\rangle + \partial_u^2 \sigma(u) \varepsilon_t.
	\end{eqnarray*}
	Using the equations given by the recursion, we end up with
	\begin{eqnarray*}
		&& \IE\big[\partial_u^2 \nabla_{\alpha} \l(\tilde Y_t(u),\theta)\big|_{\theta = \theta_0(u)}\big]\\
		&=& \frac{1}{\sigma(u)^2}\IE\Big[\langle \partial_u^2 \alpha(u), \mu(\tilde Z_{0}(u))\rangle \mu(\tilde Z_{0}(u))\\
		&&\quad\quad\quad\quad\quad + 2 \langle \partial_u \alpha(u), \partial_u \mu(\tilde Z_{0}(u))\rangle \mu(\tilde Z_{0}(u))\\
		&&\quad\quad\quad\quad\quad + 2\langle \partial_u \alpha(u), \mu(\tilde Z_{0}(u))\rangle \partial_u \mu(\tilde Z_{0}(u))\Big]\\
		&=& \frac{1}{\sigma(u)^2}\big[W(\theta_0(u))\partial_u^2 \alpha(u) + 2 P(u) \partial_u \alpha(u)\big].
	\end{eqnarray*}
	Furthermore, we obtain
	\begin{eqnarray*}
		 && \partial_u^2 \nabla_{\sigma}\l(\tilde Y_t(u),\theta)\\
		 &=& \partial_u^2\Big[-\frac{1}{\sigma^3}\big(\tilde X_t(u) - \langle \alpha,\mu(\tilde Z_{t-1}(u))\rangle\big)^2 + \frac{1}{\sigma}\Big]\\
		 &=& -\frac{2}{\sigma^3}\partial_u^2\Big[\big(\tilde X_t(u) - \langle \alpha,\mu(\tilde Z_{t-1}(u))\rangle\big)\cdot\big(\partial_u \tilde X_t(u) - \langle \alpha, \partial_u \mu(\tilde Z_{t-1}(u))\rangle \big)\Big]\\
		 &=& -\frac{2}{\sigma^3}\Big[\big(\partial_u \tilde X_t(u) - \langle \alpha, \partial_u \mu(\tilde Z_{t-1}(u))\rangle \big)^2\\
		 &&\quad\quad\quad\quad + \big(\tilde X_t(u) - \langle \alpha,\mu(\tilde Z_{t-1}(u))\rangle\big)\big(\partial_u^2 \tilde X_t(u) - \langle \alpha,\partial_u^2 \mu(\tilde Z_{t-1}(u))\rangle\big)\Big],
	\end{eqnarray*}
	and thus
	\begin{eqnarray*}
		&& \IE\big[\partial_u^2  \nabla_{\sigma}\l(\tilde Y_t(u),\theta)\big|_{\theta = \theta_0(u)}\big]\\
		&=& \frac{2}{\sigma(u)^3}\Big[ \IE \langle \partial_u \alpha(u), \mu(\tilde Z_0(u))\rangle^2 + (\partial_u \sigma(u))^2 + \sigma(u) \partial_u^2 \sigma(u)\Big]\\
		&=& \frac{2}{\sigma(u)^3}\Big[ |\partial_u \alpha(u)|_{W(\theta_0(u))}^2 + (\partial_u \sigma(u))^2 + \sigma(u) \partial_u^2 \sigma(u)\Big].
	\end{eqnarray*}
	Note that
	\[
		V(\theta)^{-1} = \sigma^2\begin{pmatrix}W(\theta)^{-1} & 0\\
	0 & \frac{1}{2}\end{pmatrix}.
	\]
	We finally obtain the bias-term
	\begin{eqnarray}
		&& \big|\IE\big[\partial_u^2  \nabla \l(\tilde Y_t(u),\theta)\big|_{\theta = \theta_0(u)}\big]\big|_{V(\theta_0(u))^{-1}}^2\nonumber\\
		&=& \frac{1}{\sigma(u)^2}\Big[|\partial_u^2 \alpha(u)|_{W(\theta_0(u))}^2  + 4|\partial_u \alpha(u)|_{P(u)W(\theta_0(u))^{-1}P(u)}^2\nonumber\\
		&&\quad\quad + 4 \partial_u^2 \alpha(u)' P(u) \partial_u \alpha(u)\Big]\nonumber\\
		&&\quad\quad + \frac{2}{\sigma(u)^4}\Big(|\partial_u \alpha(u)|_{W(\theta_0(u))}^2 + (\partial_u \sigma(u))^2 + \sigma(u) \partial_u^2 \sigma(u)\Big)^2\nonumber\\
		&=& \big|\partial_u^2 \theta_0(u) + 2w(u)\big|_{V(\theta_0(u))}^2.\label{technical_ar_mean_eq1}
	\end{eqnarray}
	In the special case of the AR(1) process $\tilde X_t(\theta) = \alpha\cdot \tilde X_{t-1}(\theta) + \sigma \varepsilon_t$ it is known that $W(\theta) = \IE[\tilde X_t(\theta)^2] = \frac{\sigma^2}{1-\alpha^2}$ and thus $P(u) = \partial_u W(\theta_0(u)) =  \frac{2\sigma(u)^2}{1-\alpha(u)^2}\big[\frac{\partial_u \sigma(u)}{\sigma(u)} + \frac{\alpha(u) \partial_u\alpha(u)}{1-\alpha(u)^2}\big] = 2W(\theta_0(u))\cdot \big[\frac{\partial_u \sigma(u)}{\sigma(u)} + \frac{\alpha(u) \partial_u\alpha(u)}{1-\alpha(u)^2}\big]$, i.e.
	\[
		w(u) = \begin{pmatrix}
			2 \partial_u \alpha(u) \cdot \Big(\frac{\partial_u \sigma(u)}{\sigma(u)} + \frac{\alpha(u) \partial_u\alpha(u)}{1-\alpha(u)^2}\Big)\\
			\frac{\sigma(u)}{2}\Big(\frac{(\partial_u \alpha(u))^2}{1-\alpha(u)^2} + \big(\frac{\partial_u \sigma(u)}{\sigma(u)}\big)^2\Big)
		\end{pmatrix}.
	\]
	This leads to
	\begin{eqnarray*}
		&& \Big|\IE\big[\partial_u^2  \nabla \l(\tilde Y_t(u),\theta)\big|_{\theta = \theta_0(u)}\big]\Big|_{V(\theta_0(u))^{-1}}^2\\
		&=& \frac{1}{1-\alpha(u)^2}\Big[\partial_u^2 \alpha(u) + 4 \partial_u \alpha(u) \cdot \Big(\frac{\partial_u \sigma(u)}{\sigma(u)} + \frac{\alpha(u) \partial_u\alpha(u)}{1-\alpha(u)^2}\Big)\Big]^2\\
		&&\quad\quad + \frac{2}{\sigma(u)^2}\Big[\partial_u^2 \sigma(u) + \sigma(u)\cdot\Big(\frac{(\partial_u \alpha(u))^2}{1-\alpha(u)^2} + \big(\frac{\partial_u \sigma(u)}{\sigma(u)}\big)^2\Big)\Big]^2
	\end{eqnarray*}
	If $\sigma \equiv \sigma_0$ is assumed to be known (and not a parameter, meaning that $\theta = \alpha$), then this simplifies to
	\begin{eqnarray*}
		&&\big|\IE\big[\partial_u^2  \nabla \l(\tilde Y_t(u),\theta)\big|_{\theta = \theta_0(u)}\big]\big|_{V(\theta_0(u))^{-1}}^2\\
		&=& \frac{1}{1-\alpha(u)^2}\Big[ \partial_u^2 \alpha(u) + 4\cdot \frac{\alpha(u) (\partial_u \alpha(u))^2}{1-\alpha(u)^2}\Big]^2.
	\end{eqnarray*}
\end{proof}

\begin{lemma}\label{technical_ar_mean}
	In the situation of Example \ref{example_ar_var}, it holds that
	\[
		\big|\IE\big[\partial_u^2  \nabla \l(\tilde Y_t(u),\theta)\big|_{\theta = \theta_0(u)}\big]\big|_{V(\theta_0(u))^{-1}} = \big| \partial_u^2 \theta_0(u) + 2 w(u)\big|_{V(\theta_0(u))}^2,
	\]
	where $P(u) := \partial_u V(\theta_0(u))$ and $w(u) := V(\theta_0(u))^{-1}P(u) \partial_u \theta_0(u)$.
\end{lemma}
\begin{proof}[Proof of Lemma \ref{technical_ar_mean}]
	Let $A(u,\theta) := \frac{\mu(\tilde Z_{t-1}(u))}{2 \langle \theta, \mu(\tilde Z_{t-1}(u))\rangle^2}$ and $B(u,\theta) := \langle \theta, \mu(\tilde Z_{t-1}(u))\rangle - \tilde X_t(u)$. It holds that $\nabla \l(\tilde Y_t(u),\theta) = A(u,\theta)\cdot B(u,\theta)$. Thus
	\[
		\partial_u^2 \nabla \l(\tilde Y_t(u),\theta) = \partial_u^2 A(u,\theta) \cdot B(u,\theta) + 2 \partial_u A(u,\theta) \cdot \partial_u B(u,\theta) + A(u,\theta) \cdot \partial_u^2 B(u,\theta).
	\]
	Since $B(u,\theta_0(u)) = \langle \theta_0(u),\mu(\tilde Z_{t-1}(u))\rangle \cdot (1 - \varepsilon_t^2)$ is a martingale difference sequence, the first summand in the above formula will vanish when the expectation ist applied. To keep the formulas short, we will abbreviate $\mu := \mu(\tilde Z_{t-1}(u))$ in the following. It holds that
	\[
		\partial_u A(u,\theta) = \frac{1}{2\langle \theta,\mu \rangle^3}\big( \langle \theta, \mu \rangle\cdot \partial_u \mu - 2 \langle \theta, \partial_u \mu\rangle \cdot \mu\big). 
	\]
	Furthermore, we have
	\begin{eqnarray*}
		\partial_u B(u,\theta) &=& \langle \theta, \partial_u \mu \rangle  - \partial_u \tilde X_t(u)^2, \quad\quad\quad \partial_u^2 B(u,\theta) = \langle \theta, \partial_u^2 \mu\rangle.- \partial_u^2 \tilde X_t(u)^2.
	\end{eqnarray*}
	The recursion $\tilde X_t(u)^2 = \langle \theta_0(u), \mu\rangle \rangle \varepsilon_t^2$ implies
	\begin{eqnarray*}
		\partial_u \tilde X_t(u)^2 &=& \big( \langle \partial_u\theta_0(u), \mu\rangle + \langle \theta_0(u), \partial_u \mu\rangle\big) \varepsilon_t^2,\\
		\partial_u^2 \tilde X_t(u)^2 &=& \big( \langle \partial_u^2 \theta_0(u), \mu\rangle + 2\langle \partial_u \theta_0(u), \partial_u \mu\rangle + \langle \theta_0(u), \partial_u^2 \mu\rangle\big) \varepsilon_t^2.\\
	\end{eqnarray*}
	Recall that $V(\theta) = \frac{1}{2}\IE\big[\frac{\mu(\tilde Z_0(\theta))\cdot \mu(\tilde Z_0(\theta))'}{\langle \theta,\mu(\tilde Z_0(\theta))\rangle^2}\big]$. Using the equations above and $P(u) = \partial_u V(\theta_0(u))$, we obtain
	\begin{eqnarray*}
		&& \IE\big[\partial_u^2 \nabla \l(\tilde Y_t(u),\theta)\big|_{\theta = \theta_0(u)}\big]\\
		&=& -2 \partial_u A(u,\theta)\big|_{\theta = \theta_0(u)} \cdot \langle \partial_u \theta_0(u), \mu\rangle\\
		&&\quad\quad\quad - A(u,\theta_0(u))\cdot \Big[\langle \partial_u^2 \theta_0(u), \mu\rangle + 2 \langle \partial_u \theta_0(u), \partial_u \mu\rangle\Big]\\
		&=& 2\frac{\langle \partial_u \theta_0(u),\mu\rangle \langle \theta_0(u), \partial_u \mu \rangle \cdot \mu}{\langle \theta_0(u), \mu\rangle^3} - \frac{\langle \partial_u \theta_0(u),\mu\rangle \cdot \partial_u \mu}{\langle \theta_0(u), \mu\rangle^2}\\
		&&\quad\quad\quad - \frac{\langle \partial_u^2 \theta_0(u), \mu\rangle \mu}{2\langle \theta_0(u), \mu\rangle^2} - \frac{\langle \partial_u \theta_0(u), \partial_u \mu\rangle \mu}{\langle \theta_0(u), \mu\rangle^2}\\
		&=& -V(\theta_0(u))\partial_u^2 \theta_0(u) - 2 P(u) \partial_u \theta_0(u).
	\end{eqnarray*}
	This leads to
	\begin{eqnarray*}
		&& \Big|\IE\big[\partial_u^2  \nabla \l(\tilde Y_t(u),\theta)\big|_{\theta = \theta_0(u)}\big]\Big|_{V(\theta_0(u))^{-1}}^2\\
		&=& |\partial_u^2 \theta_0(u)|_{V(\theta_0(u))}^2 + 4 |\partial_u \theta_0(u)|_{P(u) V(\theta_0(u))^{-1} P(u)}^2 + 4 \partial_u^2 \theta_0(u) P(u) \partial_u \theta_0(u)\\
		&=& \big| \partial_u^2 \theta_0(u) + 2 w(u)\big|_{V(\theta_0(u))}^2.
	\end{eqnarray*}
	
	In the special case of the ARCH(1) process $\tilde X_t(\theta) = \sqrt{\alpha_1 + \alpha_2 \tilde X_{t-1}(\theta)^2} \cdot \varepsilon_t$ with $\theta = (\alpha_1,\alpha_2)'$, we have
	\begin{eqnarray*}
		&&V(\theta)\\
		&=& \frac{1}{2}\IE\Big[\frac{1}{(\alpha_1 + \alpha_2 \tilde X_0(\theta)^2)^2}\cdot \begin{pmatrix}
			1 & \tilde X_0(\theta)^2\\
			\tilde X_0(\theta)^2 & \tilde X_0(\theta)^4
		\end{pmatrix}\Big],\\
		&& P(u)\\
		&=& -\IE \Big[\frac{\partial_u \alpha_1(u) + \partial_u \alpha_2(u) \tilde X_0(u)^2 + \alpha_2(u) \partial_u \tilde X_0(u)^2}{(\alpha_1(u) + \alpha_2(u) \tilde X_0(u)^2)^{3}}\begin{pmatrix}
			1 & \tilde X_0(u)^2\\
			\tilde X_0(u)^2 & \tilde X_0(u)^4
		\end{pmatrix}\Big]\\
		&&\quad\quad + \frac{1}{2}\IE\Big[\frac{1}{(\alpha_1(u) + \alpha_2(u) \tilde X_0(u)^2)^2}\begin{pmatrix}
			0 & \partial_u \tilde X_0(u)^2\\
			\partial_u \tilde X_0(u)^2 & 2 \tilde X_0(u)^2 \partial_u \tilde X_0(u)^2
		\end{pmatrix}\Big],
	\end{eqnarray*}
	where $\tilde X_t(u)^2$, $\partial_u \tilde X_t(u)^2$ can be obtained by using the recursion formulas $\tilde X_t(u)^2 = \big(\alpha_1(u) + \alpha_2(u) \tilde X_{t-1}(u)^2\big)\varepsilon_t^2$ and
	\[
		\partial_u \tilde X_t(u)^2 = \big(\partial_u \alpha_1(u) + \partial_u \alpha_2(u) \tilde X_{t-1}(u)^2 + \alpha_2(u) \partial_u \tilde X_{t-1}(u)^2\big) \varepsilon_t^2.
	\]
\end{proof}

\bibliographystyle{imsart-nameyear}

\end{document}